\documentclass[a4paper,10pt,leqno]{amsart}
        \title[The Farrell-Jones Conjecture for cocompact lattices]
         {The Farrell-Jones Conjecture for  cocompact lattices in virtually
         connected Lie groups}
       \author{Bartels, A.}
       \address{Westf\"alische Wilhelms-Universit\"at M\"unster\\
               Mathematicians Institut\\
               Einsteinium.~62,
               D-48149 M\"unster, Germany}
        \email{bartelsa@math.uni-muenster.de}
        \urladdr{http://www.math.uni-muenster.de/u/bartelsa}
        \author{Farrell, F.T}
        \address{Department of Mathematics\\Suny, Binghamton\\ Ny, 13902, \\U.S.A.}
        \email{farrell@math.binghamton.edu}
        \author{L\"uck, W.}
        \address{Mathematicians Institut der Universit\"at Bonn\\
                Endenicher Allee 60\\
                53115 Bonn, Germany}
         \email{wolfgang.lueck@him.uni-bonn.de}
          \urladdr{http://www.him.uni-bonn.de/lueck}
         \date{March, 2013}

\usepackage{hyperref}
\usepackage{color}
\usepackage{pdfsync}
\usepackage{calc}
\usepackage{enumerate,amssymb}
\usepackage[arrow,curve,matrix,tips,2cell]{xy}
  \SelectTips{eu}{10} \UseTips
  \UseAllTwocells

\DeclareMathAlphabet{\matheurm}{U}{eur}{m}{n}
\DeclareMathOperator{\asmb}{asmb}
\DeclareMathOperator{\aut}{aut}

\DeclareMathOperator{\coker}{coker}
\DeclareMathOperator{\colim}{colim}

\DeclareMathOperator{\euc}{euc}

\DeclareMathOperator{\GL}{GL}

\DeclareMathOperator{\id}{id}
\DeclareMathOperator{\im}{im}

\DeclareMathOperator{\Isom}{Isom}

\DeclareMathOperator{\pt}{pt}
\DeclareMathOperator{\pr}{pr}

\DeclareMathOperator{\rk}{rk}

\DeclareMathOperator{\vcd}{vcd}
\DeclareMathOperator{\Wh}{Wh}

\newcommand{\Fin}{{\mathcal{F}\text{in}}}

\newcommand{\VCyc}{{\mathcal{V}\mathcal{C}\text{yc}}}

  \newcommand{\IQ}{\mathbb{Q}}
  \newcommand{\IR}{\mathbb{R}}

  \newcommand{\IZ}{\mathbb{Z}}

  \newcommand{\cala}{\mathcal{A}}

  \newcommand{\calf}{\mathcal{F}}
   
  \newcommand{\calg}{\mathcal{G}}
  \newcommand{\calh}{\mathcal{H}}

  \newcommand{\call}{\mathcal{L}}

  \newcommand{\calu}{\mathcal{U}}
  \newcommand{\calv}{\mathcal{V}}
  
  \newcommand{\calw}{\mathcal{W}}

  \newcommand{\bfK}{{\mathbf K}}
  \newcommand{\bfL}{{\mathbf L}}

\newcommand{\EGF}[2]{E_{#2}(#1)}

\newcommand{\intgf}[2]{\mbox{$\int_{#1} #2$}}

\newcommand{\OrGF}[2]{\matheurm{Or}_{#2}(#1)}

\theoremstyle{plain}
\newtheorem{theorem}{Theorem}[section]

\newtheorem{lemma}[theorem]{Lemma}
\newtheorem{corollary}[theorem]{Corollary}
\newtheorem{proposition}[theorem]{Proposition}

\newtheorem*{theorem*}{Theorem}
\newtheorem*{theoremA*}{Theorem A}
\newtheorem*{theoremB*}{Theorem B}

\theoremstyle{definition}
\newtheorem{definition}[theorem]{Definition}

\newtheorem{example}[theorem]{Example}

\newtheorem{remark}[theorem]{Remark}
\newtheorem{notation}[theorem]{Notation}

\newtheorem*{definition*}{Definition}

\theoremstyle{remark}

\makeatletter\let\c@equation=\c@theorem\makeatother

\newcommand{\ox}{{\otimes}}

\newcommand{\e}{{\varepsilon}}
\newcommand{\CAT}{\operatorname{CAT}}

\newcommand{\ev}{\operatorname{ev}}

\newcommand{\version}[1]              %marks the date of last editing and compilation
{\begin{center} last edited on #1\\
last compiled on \today\\
name of tex-file: \jobname
\end{center}
}

\begin{document}

\maketitle

\begin{abstract}
Let $G$ be a cocompact lattice in a virtually connected Lie group
or the fundamental group of a $3$-dimensional manifold.
We prove the $K$- and $L$-theoretic  Farrell-Jones Conjecture  
for $G$.
\end{abstract}

\newlength{\origlabelwidth} \setlength\origlabelwidth\labelwidth

%%%%%%%%%%%%%%%%%%%%%%%%%%%%%%%%%%%%%%%%%%%%%%%%%%%%%%%%%%%%%%%%%%%%
%%%%%%%%%%%%%%%%%%%%%%%%%% Introduction %%%%%%%%%%%%%%%%%%%%%%%%%%%%%%%%
%%%%%%%%%%%%%%%%%%%%%%%%%%%%%%%%%%%%%%%%%%%%%%%%%%%%%%%%%%%%%%%%%%%%

\typeout{------------------- Introduction -----------------}
\section*{Introduction}
\label{sec:introduction}

%%%%%%%%%%%%%%%%%%%%%%%%%%%%%%%%%%%%%%%%%%%%%%%%%%%%%%%%%%%%%%%%%%%%

\subsection{Motivation and Summary}
\label{subsec:Motivation_and_summary}

The algebraic $K$-theory and $L$-theory of group rings has gained a lot of attention in
the last decades, in particular since they play a prominent role in the classification of
manifolds. Computations are very hard and here the \emph{Farrell-Jones Conjecture} comes into
play.  It identifies the algebraic $K$-theory and $L$-theory 
of group rings with the evaluation of an equivariant homology theory on the
classifying space for the family of virtually cyclic subgroups. This is the analogue of
classical results in the representation theory of finite groups such as the induction theorem of Artin or Brauer,
where the value of a functor for finite groups is computed in terms of its values on a
smaller family, for instance of cyclic or hyperelementary subgroups; in the Farrell-Jones setting the
reduction is to virtually cyclic groups. The point is that this equivariant homology theory
is much more accessible than the algebraic $K$- and $L$-groups themselves.
Actually, most of all computations for infinite groups in the literature use the Farrell-Jones Conjecture 
and concentrate on the equivariant homology side.

The Farrell-Jones Conjecture is not only important for calculations, but also gives
structural insight, since the isomorphism occurring in its formulation has also geometric
interpretations. This has the consequence that the Farrell-Jones Conjecture implies a
variety of other well-known conjectures such as the ones due to \emph{Bass}, \emph{Borel}, \emph{Kaplansky},
\emph{Novikov}, and \emph{Serre} which concern character theory for infinite groups,
algebraic topology, the classification of manifolds,  the ring structure of group rings,
and group theory. We will discuss them in more detail in
Subsection~\ref{subsec:applications}.

The main result of this paper is to prove the Farrell-Jones Conjecture for a new prominent classes of groups,
namely, \emph{cocompact lattices in almost connected Lie groups}.
We mention that the operator theoretic analog of the Farrell-Jones-Conjecture,
the Baum-Connes Conjecture,  is  known only for a few groups in this class.
With the exception of the Novikov conjecture,  
the conjecture listed above have not been known for this 
class so that our result
presents also new contributions to them. Since we address a general version of the Farrell-Jones Conjectures,
where one allows coefficients in additive categories, very powerful inheritance properties are valid which we will describe
in Subsection~\ref{subsec:Inheritance_properties}. For instance, if this general version of the Farrell-Jones Conjecture
holds for a group, it holds automatically for all subgroups, and it passes to free and direct products and to colimits of
directed systems (with not necessarily injective) structure maps.

\subsection{Statement of results}
\label{subsec:statement_of_results}

Next we give the precise formulation of our main results, 
more technical explanations will follow in the main body of the text.

\begin{theorem}[Virtually poly-$\IZ$-groups]
  \label{the:FJC_virtually_poly_Z-groups}
  Let $G$ be a virtually poly-$\IZ$-group 
  (see Definition~\ref{def:virtually_poly-Z}).

  Then both the $K$-theoretic and the $L$-theoretic Farrell-Jones
  Conjecture with additive categories as coefficients with respect to
  the family $\VCyc$ (see
  Definition~\ref{def:K-theoretic_Farrell-Jones_Conjecture}
  and~\ref{def:L-theoretic_Farrell-Jones_Conjecture}) hold for $G$.
\end{theorem}

This is the main new ingredient in proving the following results.

A virtually connected Lie group is a Lie group with finitely many path components.
A subgroup $G \subseteq L$ of a Lie group $L$ is called \emph{lattice}
if $G$ is discrete and $L/G$ has finite volume and is called a \emph{cocompact lattice}
if $G$ is discrete and $L/G$ is compact.

\begin{theorem}[Cocompact lattices in virtually connected
  Lie groups]
  \label{the:FJC_lattices}
  Let $G$ be a cocompact lattice in a virtually connected Lie group.

  Then both the $K$-theoretic and the $L$-theoretic Farrell-Jones
  Conjecture with additive categories as coefficients with respect to
  the family $\VCyc$ (see
  Definition~\ref{def:K-theoretic_Farrell-Jones_Conjecture}
  and~\ref{def:L-theoretic_Farrell-Jones_Conjecture}) hold for $G$.
\end{theorem}

An argument due to 
Roushon~\cite{Roushon(2008FJJ3), Roushon(2008IC3surf)}
shows that the above results imply the 
corresponding result for fundamental groups
of $3$-manifolds.

\begin{corollary}[Fundamental groups of $3$-manifolds]
  \label{cor:fundamental_groups_of_3-manifolds}
  Let $\pi$ be the fundamental group of a $3$-manifold 
  (possibly non-compact,
  possibly non-orientable and possibly with boundary).

  Then both the $K$-theoretic and the $L$-theoretic Farrell-Jones
  Conjecture with additive categories as coefficients with respect to
  the family $\VCyc$ (see
  Definition~\ref{def:K-theoretic_Farrell-Jones_Conjecture}
  and~\ref{def:L-theoretic_Farrell-Jones_Conjecture}) hold
   for $\pi$.
\end{corollary}

We can also handle virtually weak strongly poly-surface groups
(see~Remark~\ref{rem:Virtually_weak_strongly_poly-surface_groups})
and virtually nilpotent groups 
(see Remark~\ref{rem:virtually_nilpotent_groups}).

\begin{remark}[Finite wreath products] \label{rem:finite_wreath_product}
  Actually, all the results above do hold for the more general 
  version of the
  Farrell-Jones Conjecture, where one 
  allows \emph{finite wreath products}, i.e,
  the ``with finite wreath product'' version holds for a group $G$, if the
  version above holds for the wreath product $G \wr F$ for any finite 
  group $F$.
  The ``with finite wreath product'' version has the extra feature 
  that it holds
  for a group $G$ if it holds for some subgroup $H \subseteq G$ 
  of finite index.
\end{remark}

% \begin{remark}[Related results]\label{rem:related results}
%   Let $G$ be a torsionfree discrete subgroup in a linear Lie groups (which is
%   not necessarily a lattice) or let $G$ be the fundamental groups of a
%   complete A-regular manifold (which is not necessarily of finite volume).
%   Then Farrell-Jones\cite[Lemma~0.12 and Lemma~0.14]{Farrell-Jones(1998)} have
%   proved the $K$-theoretic Farrell-Jones Conjecture up to dimension one with
%   coefficients in the (untwisted) ring $\IZ$ for the family $VCyc$ and the2307
%   $L$-theoretic Farrell-Jones Conjecture with coefficients in the (untwisted)
%   ring $\IZ$ for the family $VCyc$.
% \end{remark}

The paper is organized as follows.
We will briefly review the Farrell-Jones Conjecture and its relevance in
Section~\ref{sec:A_brief_review_of_the_Farrell_Jones_Conjecture_with_coefficients}.
In this section we also collect a number of results about the
Farrell-Jones Conjecture that will be used throughout this paper. 
In Sections~\ref{sec:Virtually_finitely_generated_abelian_groups}
we treat the case of virtual finitely  generated abelian groups.
Here we follow an argument of Quinn~\cite[Sections~2 and~3.2]{Quinn(2012virtab)}
and extend it to our setting.
The main difference is that the present proof depends on
a different control theory;
we use 
Theorem~\ref{the:Farrell-Jones_Conjecture_for_Farrell-Hsiang_groups} 
(which is proved in~\cite{Bartels-Lueck(2012Farrell-Hsiang)}) 
instead of results 
from~\cite{Quinn(2012controlledI)}.
The main work of this paper is done in 
Section~\ref{sec:Irreducible_special_affine_groups} 
where we treat special affine groups. 
This section builds on ideas from
Farrell-Hsiang~\cite{Farrell-Hsiang(1981b)}
and
Farrell-Jones~\cite{Farrell-Jones(1988b)}.
The proof of Theorem~\ref{the:FJC_virtually_poly_Z-groups}
is given in Section~\ref{sec:Virtually_poly-Z-groups}. 
Theorem~\ref{the:FJC_lattices} is proved in
Section~\ref{sec:Cocompact_lattices_in_virtually_connected_Lie_groups}
by reducing it to Theorem~\ref{the:FJC_virtually_poly_Z-groups}
and the main results from~\cite{Bartels-Lueck(2012annals), Wegner(2012higher-cat0)}.

Fundamental groups of $3$-manifolds are discussed 
in~\ref{sec:Fundamental_groups_of_3-manifolds}. 
In Section~\ref{sec:Reducing_the_family_VCyc} 
we reduce the family of virtually cyclic
subgroups to a smaller family extending previous 
results of Quinn~\cite{Quinn(2012virtab)} for 
untwisted coefficients in rings to the more general 
setting of coefficients in additive categories.

\subsection{Acknowledgements}
The work was financially supported by the Sonderforschungsbereich 878
\--- Groups, Geometry \& Actions \--- and the
Leibniz-Preis of the last author.
And the second author was
partially supported by an NSF grant.

We thank the referee for a long list of helpful comments.

%%%%%%%%%%%%%%%%%%%%%%%%%%%%%%%%%%%%%%%%%%%%%%%%%%%%%%%%%%%%%%%%%%%%%
%%%%%%% Section 1: A brief review of the Farrell Jones Conjecture with
%%%%%%% coefficients %%%%%%%%%%%
%%%%%%%%%%%%%%%%%%%%%%%%%%%%%%%%%%%%%%%%%%%%%%%%%%%%%%%%%%%%%%%%%%%%%

\typeout{----- A brief review of the Farrell Jones Conjecture with
  coefficients --------------}

\section{A brief review of the Farrell-Jones Conjecture with
  coefficients}
\label{sec:A_brief_review_of_the_Farrell_Jones_Conjecture_with_coefficients}

We briefly review the $K$-theoretic and $L$-theoretic
Farrell-Jones Conjecture with additive categories as
coefficients.

%%%%%%%%%%%%%%%%%%%%%%%%%%%%%%%%%%%%%%%%%%%%%%%%%%%%%%%%%%%%%%%%%%%%%

\subsection{The formulation of the Farrell-Jones Conjecture}

\label{subsec:The_formulation_of_the_Farrell-Jones_Conjecture}

\begin{definition}[$K$-theoretic Farrell-Jones Conjecture with
  additive categories as coefficients]
  \label{def:K-theoretic_Farrell-Jones_Conjecture}
  Let $G$ be a group and let $\calf$ be a family of subgroups.  Then
  $G$ satisfies the \emph{$K$-theoretic Farrell-Jones Conjecture with
    additive categories as coefficients with respect to $\calf$} if
  for any additive $G$-category $\cala$ the assembly map
  \begin{eqnarray*}
    & \asmb^{G,\cala}_n \colon H_n^G\bigl(\EGF{G}{\calf};\bfK_{\cala}\bigr) \to
    H_n^G\bigl(\pt;\bfK_{\cala}\bigr)
    = K_n\left(\intgf{G}{\cala}\right)
    &
  \end{eqnarray*}
  induced by the projection $\EGF{G}{\calf} \to \pt$ is bijective for
  all $n \in \IZ$.
\end{definition}

\begin{definition}[$L$-theoretic Farrell-Jones Conjecture with
  additive categories as coefficients]
  \label{def:L-theoretic_Farrell-Jones_Conjecture}
  Let $G$ be a group and let $\calf$ be a family of subgroups.  Then
  $G$ satisfies the \emph{$L$-theoretic Farrell-Jones Conjecture with
    additive categories as coefficients with respect to $\calf$} if
  for any additive $G$-category with involution $\cala$ the assembly
  map
  \begin{eqnarray*}
    & \asmb^{G,\cala}_n \colon
    H_n^G\bigl(\EGF{G}{\calf};\bfL_{\cala}^{\langle - \infty\rangle}\bigr) \to
    H_n^G\bigl(\pt;\bfL_{\cala}^{\langle - \infty\rangle}\bigr)
    = L_n^{\langle - \infty\rangle}\left(\intgf{G}{\cala}\right)
    &
  \end{eqnarray*}
  induced by the projection $\EGF{G}{\calf} \to \pt$ is bijective for
  all $n \in \IZ$.
\end{definition}

Here are some explanations.

Given a group $G$, a \emph{family of subgroups} $\calf$ is a
collection of subgroups of $G$ such that $H\in \calf, g \in G$ implies
$gHg^{-1} \in \calf$ and for any $H \in \calf$ and any subgroup $K
\subseteq H$ we have $K \in \calf$.

For the notion of a \emph{classifying space $\EGF{G}{\calf}$ for a
  family $\calf$} we refer for instance to the survey article~\cite{Lueck(2005s)}.

The natural choice for $\calf$ in the Farrell-Jones Conjecture is the
family $\VCyc$ of virtually cyclic subgroups but sometimes it is
useful to consider in between other families for technical reasons.

\begin{notation}[Abbreviation FJC]\label{not_abbrev_for_FJC}
In the sequel the abbreviation FJC stands
for ``Farrell-Jones Conjecture with additive categories as coefficients
with respect to the family $\VCyc$''.
\end{notation}

\begin{remark}[Relevance of the additive categories as coefficients]
The versions of the  Farrell-Jones Conjecture appearing in
Definition~\ref{def:K-theoretic_Farrell-Jones_Conjecture}
and Definition~\ref{def:L-theoretic_Farrell-Jones_Conjecture} are formulated
and analyzed in~\cite{Bartels-Lueck(2009coeff)},
\cite{Bartels-Reich(2007coeff)}.  They encompass the versions for
group rings $RG$ over arbitrary rings $R$, where one can built in a
twisting into the group ring or treat more generally crossed product
rings $R \ast G$ and one can allow orientation homomorphisms $w \colon
G \to \{\pm 1\}$ in the $L$-theory case. Moreover, inheritance properties
are built in and one does not have to pass to fibered versions anymore as explained
in Subsection~\ref{subsec:Inheritance_properties}.
\end{remark}

\begin{example}[Torsionfree $G$ and regular $R$]
If $R$ is regular and $G$ is torsionfree, then the Farrell-Jones
Conjecture reduces to the claim that the classical assembly maps
\begin{eqnarray*}
H_n(BG;\bfK_R)  & \to & K_n(RG);
\\
H_n\bigl(BG;\bfL_R^{\langle - \infty \rangle}\bigr) & \to & L_n^{\langle -\infty \rangle}(RG),
\end{eqnarray*}
are bijective for $n \in \IZ$, where $BG$ is the classifying space of $BG$ and
$H_*\bigl(-;\bfK_R\bigr)$ and $H_*\bigl(-;\bfL^{\langle -\infty \rangle}_R\bigr)$
are generalized homology theories with
$H_n\bigl(\pt;\bfK_R\bigr) \cong K_n(R)$ and
$H_n\bigl(\pt;\bfL^{\langle -\infty \rangle}_R\bigr) \cong L_n^{\langle -\infty \rangle}(R)$
for $n \in \IZ$.
\end{example}

%Here is a variant of the $K$-theoretic
%Farrell-Jones Conjecture~\ref{def:K-theoretic_Farrell-Jones_Conjecture}
%which often suffices for the geometric or algebraic applications.

%\begin{definition}[$K$-theoretic Farrell-Jones Conjecture
% with additive categories as coefficients up to dimension one]
%  \label{def:K-theoretic_Farrell-Jones_Conjecture_up_to_dimension_one}
%  Let $G$ be a group and let $\calf$ be a family of subgroups.  Then
%  $G$ satisfies the \emph{$K$-theoretic Farrell-Jones Conjecture with
%    additive categories as coefficients with respect to $\calf$ up to 
%dimension one} if
%  for any additive $G$-category $\cala$   the assembly map
%  \begin{eqnarray*}
%    & \asmb^{G,\cala}_n \colon
%    H_n^G\bigl(\EGF{G}{\calf};
%  \bfK_{\cala}^{\langle - \infty\rangle}\bigr) \to
%    H_n^G\bigl(\pt;\bfK_{\cala}^{\langle - \infty\rangle}\bigr)
%    = K_n^{\langle - \infty\rangle}\left(\intgf{G}{\cala}\right)
%    &
%  \end{eqnarray*}
%  induced by the projection $\EGF{G}{\calf} \to \pt$ is bijective for
%  all $n \le 0$ and surjective for $n = 1$.
%\end{definition}

The original source for the (Fibered) Farrell-Jones Conjecture is the
paper by Farrell-Jones~\cite[1.6 on page~257 and~1.7 on
page~262]{Farrell-Jones(1993a)}.

%%%%%%%%%%%%%%%%%%%%%%%%%%%%%%%%%%%%%%%%%%%%%%%%%%%%%%%%%%%%%%%%%%%%%

%%%%%%%%%%%%%%%%%%%%%%%%%%%%%%%%%%%%%%%%%%%%%%%%%%%%%%%%%%%%%%%%%%%%%

\subsection{Applications}

\label{subsec:applications}

As remarked in the introduction, the Farrell-Jones Conjecture 
implies a number of other conjectures.
For a detailed discussion of these applications we refer 
to~\cite{Bartels-Lueck-Reich(2008appl)} and the survey
article~\cite{Lueck-Reich(2005)}.  
Here we summarize these applications as follows. 

\begin{itemize}

\item \emph{Bass Conjecture}\\
  One version of the Bass Conjecture predicts the possible values of the
  Hattori-Stallings rank of a finitely generated $RG$-module extending
  well-known results for finite groups to infinite groups.  
  If $R$ is a field of
  characteristic zero, it follows from the $K$-theoretic FJC.

\item \emph{Borel Conjecture}\\
 The Borel Conjecture says that a closed aspherical 
 topological manifold $N$ is
 topologically rigid, i.e, any homotopy equivalence 
 $M \to N$ with a closed
 topological manifold as source and $N$ as target is homotopic to a
 homeomorphism. The Borel Conjecture is known to be true in 
 dimensions $\le 3$.
 It holds in dimension $\ge 5$ if the fundamental group satisfies the
 $K$-theoretic FJC  and the $L$-theoretic {FJC}.

\item \emph{Homotopy invariance of $L^2$-torsion}\\
  There is the conjecture that for two homotopy equivalent finite connected
  $CW$-complexes whose universal coverings are $\det$-$L^2$-acyclic the
  $L^2$-torsion of their universal coverings agree.  This follows from the
  $K$-theoretic FJC.

\item \emph{Kaplansky Conjecture}\\
  The Kaplansky Conjecture predicts for an integral domain $R$ and a torsionfree
  group $G$ that $0$ and $1$ are the only idempotents in $RG$. If $R$ is a field
  of characteristic zero or if $R$ is a skewfield and $G$ is sofic, it follows
  from the $K$-theoretic FJC.

\item \emph{Moody's Induction Conjecture}\\
  If $R$ is a regular ring with $\IQ \subseteq R$, e.g., a skewfield of
  characteristic zero, then Moody's Induction Conjecture predicts that the map
$$\colim_{\OrGF{G}{\Fin}} K_0(RH) \xrightarrow{\cong} K_0(RG)$$
is bijective. Here the colimit is taken over the full subcategory of the orbit
category whose objects are homogeneous spaces $G/H$ with finite $H$.  It follows
from the $K$-theoretic FJC.

If $F$ is a skewfield of prime characteristic $p$, then Moody's Induction
Conjecture predicts that the map
$$\colim_{\OrGF{G}{\Fin}} K_0(FH)[1/p] \xrightarrow{\cong} K_0(RG)[1/p]$$
is bijective. This also follows from the $K$-theoretic FJC.

\item \emph{Poincar\'e duality groups}\\
  Let $G$ be a finitely presented Poincar\'e duality group of dimension $n$.
  Then there is the conjecture that $G$ is the fundamental group of a compact
  ANR-homology manifold. This follows in dimension $n \ge 6$ if the fundamental
  group satisfies the $K$-theoretic FJC  and the
  $L$-theoretic FJC (see~\cite[Theorem~1.2]{Bartels-Lueck-Weinberger(2009)}).
  In order to replace ANR-homology manifolds by topological
  manifold, one has to deal with Quinn's resolution obstruction
  (see~\cite{Bryant-Ferry-Mio-Weinberger(1996)}, \cite{Quinn(1987_resolution)}).

\item \emph{Novikov Conjecture}\\
  The Novikov Conjecture predicts for a group $G$ that the higher $G$-signatures
  are homotopy invariants and follows from the $L$-theoretic {FJC}.

\item \emph{Serre Conjecture}\\
  The Serre Conjecture predicts that a group $G$ of type {FP} is of type {FF}.  It
  follows from the $K$-theoretic FJC.

\item \emph{Vanishing of the reduced projective class group}\\
  Let $G$ be a torsionfree group and $R$ a regular ring. Then there is the
  conjecture that the change of rings map $K_0(R) \to K_0(RG)$ is bijective.  In
  particular the reduced projective class group $\widetilde{K}_0(RG)$ vanishes
  if $R$ is a principal ideal domain.  This follows from the $K$-theoretic FJC.

\item \emph{Vanishing of the Whitehead group}\\
  There is the conjecture that the Whitehead group $\Wh(G)$ of a torsionfree
  group $G$ vanishes.  This follows from the $K$-theoretic FJC.

\end{itemize}
%%%%%%%%%%%%%%%%%%%%%%%%%%%%%%%%%%%%%%%%%%%%%%%%%%%%%%%%%%%%%%%%%%%%%

\subsection{Inheritance properties}

\label{subsec:Inheritance_properties}

The formulation of the Farrell-Jones Conjecture with
additive categories as coefficients has the advantage that the
various inheritance properties which led to and are guaranteed by the
so called fibered versions are automatically built in
(see~\cite[Theorem~0.7]{Bartels-Lueck(2009coeff)}). 
This implies the following results.
(see~\cite[Corollary~0.9, Corollary~0.10 and
Corollary~0.11]{Bartels-Lueck(2009coeff)}
and~\cite[Lemma~2.3]{Bartels-Lueck(2012annals)}).

\begin{theorem}[Directed colimits]\label{cor:directed_colimits}
  Let $\{G_i \mid i \in I\}$ be a directed system (with not
  necessarily injective structure maps) and let $G$ be its colimit
  $\colim_{i \in I} G_i$. Suppose that $G_i$ satisfy the $K$-theoretic
  FJC for every $i \in I$. Then $G$ satisfies the $K$-theoretic {FJC}.

  The same is true for the  $L$-theoretic {FJC}.
\end{theorem}

\begin{theorem}[Extensions] \label{the:extensions} Let $1 \to K \to G
  \xrightarrow{p} Q \to 1$ be an extension of groups.  Suppose that
  the group $Q$ and for any virtually cyclic subgroup $V \subseteq Q$
  the group $p^{-1}(V)$ satisfy the $K$-theoretic {FJC}. Then the
  group $G$ satisfies the $K$-theoretic {FJC}.

  The same is true for the  $L$-theoretic {FJC}.  
\end{theorem}

\begin{theorem}[Subgroups]\label{the:subgroups}
  If $G$ satisfies the $K$-theoretic FJC, then any
  subgroup $H \subseteq G$ satisfies the $K$-theoretic {FJC}.

  The same is true for  the $L$-theoretic {FJC}.
\end{theorem}

\begin{theorem}[Free and direct products]
 \label{the:free_products_and_direct_products}
If the groups $G_1 $ and $G_2$ satisfy the $K$-theoretic FJC,
then their free amalgamated product $G_1 \ast G_2$ and their
direct product $G_1 \times G_2$
satisfy the $K$-theoretic {FJC}.

The same is true for  the $L$-theoretic {FJC}.
\end{theorem}

Theorem~\ref{the:extensions} and Theorem~\ref{the:subgroups} have also
been proved in~\cite{Hambleton-Pedersen-Rosenthal(2007)}.

\begin{theorem}[Transitivity Principle]
  \label{the:transitivity}
  Let $\calf \subseteq \calg$ be two families of subgroups of $G$.
  Assume that for every element $H \in \calg$ the group $H$ satisfies
  the $K$-theoretic Farrell-Jones Conjecture with additive categories
  as coefficients for the family $\calf|_H
  = \{K \subseteq H\mid K \in \calg\}$.

  Then the relative assembly map
  \[\asmb^{G,\calf,\calg}_n \colon H_n^G\bigl(\EGF{G}{\calf};\bfK_{\cala}\bigr)
 \to H_n^G\bigl(\EGF{G}{\calg};\bfK_{\cala}\bigr)\]
  induced by the up to $G$-homotopy unique $G$-map
  $\EGF{G}{\calf} \to \EGF{G}{\calg}$
  is an isomorphism for any additive $G$-category $\cala$ 
  and all $n \in \IZ$.

  In particular, $G$ satisfies the $K$-theoretic Farrell-Jones Conjecture
  with additive categories as coefficients for the family $\calg$
  if and only if $G$ satisfies the $K$-theoretic Farrell-Jones Conjecture
  with additive categories as coefficients  for the family $\calf$

  The same is true for  the $L$-theoretic {FJC}.
\end{theorem}
\begin{proof}
  Given an additive $G$-category $\cala$ with involution, one obtains
  in the obvious way a homology theory over the group $G$ in the sense of
  \cite[Definition~1.3]{Bartels-Echterhoff-Lueck(2008colim)}
  using~\cite[Lemma~9.5]{Bartels-Lueck(2009coeff)}. In
  Bartels-Echterhoff-L\"uck~\cite[Theorem~3.3]{Bartels-Echterhoff-Lueck(2008colim)}
  the Transitivity Principle is formulated for homology theories over
  a given group $G$.  Its proof is a slight variation of the proof for
  an equivariant homology theory in Bartels-L\"uck~\cite[Theorem~2.4,
  Lemma~2.2]{Bartels-Lueck(2006)} and it yields the claim.
\end{proof}

\begin{corollary} \label{cor:from_Transitivity_Principle}
Let $1 \to K \to G \to Q \to 1$ be an exact sequence of groups.
Suppose that $Q$  satisfies the $K$-theoretic FJC
and that $K$ is finite.  Then $G$  satisfies the $K$-theoretic {FJC}.

The same is true for  the $L$-theoretic {FJC}.
\end{corollary}

We mention already here the following corollary of
the Transitivity Principle~\ref{the:extensions},
Theorem~\ref{the:FJC_virtually_poly_Z-groups}
and Lemma~\ref{lem:virt_dim_of_virt_poly_Z-groups}%
~\ref{lem:virt_dim_of_virt_poly_Z-groups:extensions}.

\begin{corollary} \label{cor:from_Transitivity_Principle_and_poly_Z}
Let $1 \to K \to G \to Q \to 1$ be an exact sequence of groups.
Suppose that $Q$  satisfies the $K$-theoretic FJC
and that $K$ is virtually poly-$\IZ$.  Then $G$  satisfies the $K$-theoretic {FJC}.

The same is true for  the $L$-theoretic {FJC}.
\end{corollary}

\begin{remark}[Virtually nilpotent groups]
  \label{rem:virtually_nilpotent_groups}
  The inheritance properties allows sometimes to prove the FJC
  for other interesting groups. For instance, we can show that every
  virtually nilpotent group satisfies both the $K$-theoretic and the
  $L$-theoretic {FJC}.  This follows from the argument appearing in the proof
  of~\cite[Lemma~2.13] {Bartels-Lueck-Reich(2008appl)} together with
  Theorem~\ref{the:FJC_virtually_poly_Z-groups},
  Theorem~\ref{cor:directed_colimits} and Theorem~\ref{the:extensions}.
\end{remark}

%%%%%%%%%%%%%%%%%%%%%%%%%%%%%%%%%%%%%%%%%%%%%%%%%%%%%%%%%%%%%%%%%%%%%

\typeout{------------------ A strategy ---------------------}

\subsection{A strategy}
\label{subsec:strategy}

In this subsection we present a general strategy to prove the {FJC}.
It is motivated by the paper of
Farrell-Hsiang~\cite{Farrell-Hsiang(1981b)}.

We call a simplicial $G$-action on a simplicial $X$ \emph{cell
preserving} if the following holds: If $\sigma$ is a simplex with
interior $\sigma^{\circ}$ and $g \in G$ satisfy $g\cdot \sigma^{\circ}
\cap \sigma^{\circ} \not= \emptyset$, then we get $g\cdot x = x$ for
all $x \in \sigma$. If $G$ acts simplicially on $X$, then the induced
simplicial $G$-action on the barycentric subdivision $X'$ is always
cell preserving.  The condition cell preserving guarantees that $X$
with the filtration by its skeletons coming from the simplicial
structure on $X$ is a $G$-$CW$-complex structure on $X$.

Recall that a finite group $H$ is called \emph{$p$-hyperelementary}
for a prime $p$, if there is a short exact sequence
\begin{equation*}
  0 \to C \to H \to P \to 0
\end{equation*}
with $P$ a $p$-group and $C$ a cyclic group of order
prime to $p$. It is called \emph{hyperelementary} if it
is \emph{hyperelementary} for some prime $p$.

We recall the following definition 
from~\cite{Bartels-Lueck(2012Farrell-Hsiang)}.

\begin{definition}[Farrell-Hsiang group]
    \label{def:Farrell-Hsiang}
    Let $\calf$ be a family of subgroups of the finitely generated group $G$.
    We call $G$ a \emph{Farrell-Hsiang group} with respect to the family $\calf$
    if the following holds for a fixed word metric $d_G$:

    There exists a natural  number $N$ such that such that for
    any $R > 0$, $\epsilon > 0$ there is a surjective homomorphism 
    $\alpha_{R,\epsilon}
    \colon G \to F_{R,\epsilon}$ with $F_{R,\epsilon}$ a finite group
     such that the following condition is
    satisfied. For any hyperelementary subgroup $H$ of $F_{R,\epsilon}$ we set
    $\overline{H} := \alpha_{R,\epsilon}^{-1}(H)$
    and require that there exists a simplicial complex $E_H$ of
    dimension at most $N$ with a cell preserving simplicial $\overline{H}$-action
    whose stabilizers belong to $\calf$, and an $\overline{H}$-equivariant 
    map $f_H \colon G \to E_H$ such that $d_G(g,h) < R$ 
    implies $d_{E_H}^1 (f(g),f(h)) <
    \epsilon$ for all $g,h \in G$, where $d^1_{E_H}$ is the $l^1$-metric on $E_H$.
  \end{definition}

  \begin{remark} \label{rem:Farrell-Hisiang_plus_conjugation}
    We point out that the existence of $\overline{H}$-equivariant 
    maps $f_H \colon G \to E_H$ as in the above definition is 
    invariant under
    conjugation: If $K$ is conjugated to $H$ in $F_{R,\e}$ then there is
    $\gamma \in G$ such that 
    $K = \alpha_{R,\e}(\gamma^{-1}) H \alpha_{R,\e} (\gamma)$.
    Set $E_K := E_H$. 
    There is an action of $\overline{K} = \gamma^{-1} \overline{H} \gamma$  
    on this simplicial complex where $k \in \overline{K}$ acts as
    $\gamma k \gamma^{-1}$.
    Finally, define $f_K \colon G \to E_K$ by $f_K (g) := f_H (\gamma g)$ 
    this map is $\overline{K}$-equivariant and has the appropriate 
    contracting property
    because $g \mapsto \gamma g$ is an isometry of $G$. 
  \end{remark}

  The next result is proved in~\cite{Bartels-Lueck(2012Farrell-Hsiang)}.

  \begin{theorem}[Farrell-Hsiang groups and the
    Farrell-Jones-Conjecture]
    \label{the:Farrell-Jones_Conjecture_for_Farrell-Hsiang_groups}
    Let $G$ be a Farrell-Hsiang group with respect to the
    family $\calf$ in the sense of Definition~\ref{def:Farrell-Hsiang}.

    Then $G$ satisfies the $K$-theoretic and the $L$-theoretic
    Farrell-Jones Conjecture with additive categories as coefficients with
    respect to the family $\calf$ (see
    Definition~\ref{def:K-theoretic_Farrell-Jones_Conjecture} and
    Definition~\ref{def:L-theoretic_Farrell-Jones_Conjecture}).
  \end{theorem}

%%%%%%%%%%%%%%%%%%%%%%%%%%%%%%%%%%%%%%%%%%%%%%%%%%%%%%%%%%%%%%%%%%%%%
  %%%%%%%%%%%% Section 3: Virtually finitely generated abelian
  %%%%%%%%%%%% groups %%%%%%%%%%%%%%%%%%%
%%%%%%%%%%%%%%%%%%%%%%%%%%%%%%%%%%%%%%%%%%%%%%%%%%%%%%%%%%%%%%%%%%%%%

  \typeout{----------------- Virtually finitely generated abelian
    groups ------------------}

  \section{Virtually finitely generated abelian groups}
  \label{sec:Virtually_finitely_generated_abelian_groups}

  In this section we prove the $K$-theoretic and the $L$-theoretic
  Farrell-Jones Conjecture with additive categories as coefficients
  with respect to  the family $\VCyc$ (see
  Definition~\ref{def:K-theoretic_Farrell-Jones_Conjecture} and
  Definition~\ref{def:L-theoretic_Farrell-Jones_Conjecture}) for
  virtually finitely generated abelian groups. This will be one
  ingredient  in proving the $K$-theoretic and the
  $L$-theoretic FJC for virtually poly-$\IZ$ groups.

\begin{theorem}[Virtually  finitely generated abelian groups]
  \label{the:The_Farrell-Jones_Conjecture_for_virtually_finitely_generated_abelian_groups}
  Both the $K$-theoretic and the $L$-theoretic FJC hold for
  virtually finitely generated abelian groups.
\end{theorem}

\begin{remark}
\label{rem:crystallographic_and_CAT(0)}
Since a virtually finitely generated abelian group
possesses an epimorphism with finite kernel onto a crystallographic group
(see for  instance~\cite[Lemma~4.2.1]{Quinn(2012virtab)}), 
it suffices to prove
Theorem~\ref{the:The_Farrell-Jones_Conjecture_for_virtually_finitely_generated_abelian_groups}
for crystallographic groups because of 
Corollary~\ref{cor:from_Transitivity_Principle}.

A crystallographic group is obviously a $\CAT(0)$-group. Hence it
satisfies the $K$- and $L$-theoretic Farrell Jones Conjecture  
with additive categories as coefficients with respect to
$\VCyc$ by~\cite[Theorem~B]{Bartels-Lueck(2012annals)}
and~\cite{Wegner(2012higher-cat0)}.
Nevertheless we give a proof using different methods in this
Section, because this proof for virtually finitely 
generated abelian groups is a good model for the proof
for virtually poly-$\IZ$-groups in 
Sections~\ref{sec:Irreducible_special_affine_groups} 
and~\ref{sec:Virtually_poly-Z-groups}.

In this section we follow Quinn's proof of the 
Farrell-Jones Conjecture for virtually finitely generated abelian groups
and untwisted group rings $RG$ over commutative 
rings~\cite[Theorem~1.2.2
and Corollary~1.2.3]{Quinn(2012virtab)}.
\end{remark}

%%%%%%%%%%%%%%%%%%%%%%%%%%%%%%%%%%%%%%%%%%%%%%%%%%%%%%%%%%%%%%%%%%%%%

  \subsection{Review of crystallographic groups}
  \label{subsec:Review_of_crystallographic_groups}

  In this subsection we briefly collect some basic facts about
  crystallographic groups.

  A \emph{crystallographic group $\Delta$ of rank $n$} is a discrete
  subgroup of the group of isometries of $\IR^n$ such that the induced
  isometric group action $\Delta \times \IR^n \to \IR^n$ is proper and
  cocompact.  The translations in $\Delta$ form a normal subgroup
  isomorphic to $\IZ^n$ which is called the \emph{translation
    subgroup} and will be denoted by $A = A_{\Delta}$.  It is equal to
  its own centralizer.  The quotient $F_{\Delta} := \Delta/A_{\Delta}$
  is called the \emph{holonomy group} and is a finite group.

  A group $G$ is called an \emph{abstract crystallographic group of rank $n$} if it
  contains a normal subgroup $A$ which is isomorphic to $\IZ^n$, has finite index and is
  equal to its own centralizer in $G$. Such a subgroup $A$ is unique by the following
  argument.  The centralizer in $G$ of any subgroup $B$ of $A$, which has finite index in
  $A$, is $A$, since any automorphism of $A$ which induces the identity on $B$ is itself
  the identity.  Suppose that $A'$ is another normal subgroup which is isomorphic to
  $\IZ^n$, has finite index and is equal to its own centralizer in $G$. 
  Then $A\cap A'$ is
  a normal subgroup of $A$ and of $A'$ of finite index. 
  Hence $A= A'$, as both $A$ and $A'$ coincide with the centralizer of 
  $A \cap A'$. 
  In
  particular $A$ is a characteristic subgroup of $G$, i.e., any group automorphism of $G$
  sends $A$ to $A$.

  Every abstract crystallographic group $G$ of rank $n$ is a
  crystallographic group of rank $n$ whose group of translations is
  $A$ and vice versa (see~\cite[Definition~1.9 and
  Proposition~1.12]{Connolly-Kozniewski(1990)}). The rank of a crystallographic group
  is equal to its virtual cohomological dimension.

  \begin{notation} \label{not:sA_and_S_s}
  Let $A$ be an abelian group and $s$ be an integer. We denote by
  $sA$ or $s \cdot A$ the subgroup of $A$ given the image of the
  map $s \cdot \id_A \colon A \to A$
  and by $A_s$ the quotient $A/sA$.
  \end{notation}

  \begin{definition}[Expansive map]\label{def:expansive_map}
  Let $\Delta$ be a crystallographic group and $s$ be an integer different from zero.
  A group homomorphism $\phi \colon \Delta \to \Delta$ is called \emph{$s$-expansive}
  if it fits into the following commutative diagram
  $$
  \xymatrix{1 \ar[r]
  & A_{\Delta} \ar[r]^{i} \ar[d]^{s \cdot \id}
  & \Delta \ar[r]^{\pr} \ar[d]^{\phi}
  & F_{\Delta} \ar[r] \ar[d]^{\id}
  & 1
  \\
  1 \ar[r]
  & A_{\Delta} \ar[r]^{i}
  & \Delta \ar[r]^{\pr}
  & F_{\Delta} \ar[r]
  & 1}
  $$
\end{definition}

Given an abelian group $A$, let $A\rtimes_{-\id} \IZ/2$ be the semidirect product with
respect to the automorphism $-\id \colon A \to A$.

\begin{lemma}\label{lem:expansive_maps}
  Let $\Delta$ be a crystallographic group. Let $s \not= 0$ be an integer.
  \begin{enumerate}

  \item \label{lem:expansive_maps:existence} There exists an $s$-expansive map $\phi
    \colon \Delta \to \Delta$ provided that $s \equiv 1 \mod |F_{\Delta}|$;

  \item \label{lem:expansive_maps:affine_map} 
    For every $s$-expansive map $\phi \colon
    \Delta \to \Delta$ there  exists $u \in \IR^n$ such that the affine map
    $$a_{s,u} \colon \IR^n \to \IR^n, \quad x \mapsto s\cdot x + u$$
    is $\phi$-equivariant;

  \item \label{lem:expansive_maps:subgroup_in_the_image} 
    Suppose that $\Delta$ is $\IZ^n$
    or the semi-direct product $\IZ^n \rtimes_{-\id} \IZ/2$.
    Let $\overline{H} \subseteq \Delta $ be a subgroup with
    $\overline{H} \cap \IZ^n \subseteq  s\IZ^n$.

    Then there exists an $s$-expansive map $\phi \colon \Delta \to \Delta$
    and an element $v \in \IR$ such that
    $\overline{H} \subseteq \im(\phi)$  and the map
    $a \colon \IR^n \to \IR^n,\; x \mapsto s \cdot x + v$ 
    is $\phi$-equivariant.

  \end{enumerate}

\end{lemma}
\begin{proof}~\ref{lem:expansive_maps:existence} Consider $A_{\Delta}$
  as $\IZ[F_{\Delta}]$-module by the conjugation action of
  $F_{\Delta}$ on $A_{\Delta}$. Since $A_{\Delta}$ is abelian,
  isomorphism classes of extensions with $A_{\Delta}$ as normal
  subgroup and $F_{\Delta}$ as quotient are in one to one
  correspondence with elements in $H^2(F_{\Delta};A_{\Delta})$
  (see~\cite[Theorem~3.12 in Chapter~IV on page~93]{Brown(1982)}).
  Let $\Theta$ be the class associated to the extension
 $1 \to   A_{\Delta} \to \Delta \to F_{\Delta} \to 1$.  Since $F_{\Delta}$ is
  finite, $H^2(F_{\Delta};A_{\Delta})$ is annihilated by
  multiplication with $|F_{\Delta}|$
  (see~\cite[Corollary~10.2 in Chapter~III on page~84]{Brown(1982)}).
  Hence multiplication with $s$ induces the
  identity on $H^2(F_{\Delta};A_{\Delta})$ because of 
  $s \equiv 1 \mod |F_{\Delta}|$.
Therefore $H^2(F_{\Delta};s \cdot \id_{A_{\Delta}}) = s \cdot
  \id_{H^2(F_{\Delta};A_{\Delta})}$ sends $\Theta$ to $\Theta$, and
  the claim follows.
  \\[1mm]~\ref{lem:expansive_maps:affine_map} Since $F_{\Delta}$ is finite,
  $H^1\bigl(F_{\Delta};A_{\Delta} \otimes_{\IZ} \IR\bigr)$ is trivial. Now one proceeds as
  in the (more difficult) proof of Lemma~\ref{lem:equivariant_affine_diffeo}.
  \\[1mm]~\ref{lem:expansive_maps:subgroup_in_the_image} In the case
  $\Delta = \IZ^n$ just take $\phi = s \cdot \id_{\IZ^n}$ and $v=0$. 
  It remains
  to treat the case $\Delta = \IZ^n\rtimes_{-\id} \IZ/2$.

  Let $t$ be the generator of $\IZ/2$. We write the multiplication in $\IZ/2$
  multiplicatively and in $\IZ^n$ additively.  For an element $u\in \IZ^n$ we define an
  injective group homomorphism
  \begin{eqnarray}
    &\phi_u \colon \Delta  \to \Delta &
    \label{map_phi_u_colon_Delta_to_Delta}
  \end{eqnarray}
  by $\phi_u(t) = ut$ and $\phi_u(x) = s \cdot x$ for $x \in \IZ^n$.  This is well defined
  as the following calculation shows for $x \in \IZ^n$
  \begin{eqnarray*}
    &\phi_u(t)^2 = utut = utut^{-1}  = u+ (-u) = 0;&
  \end{eqnarray*}
  and
  \begin{multline*}
  \phi_u(t)\phi_u(x)\phi_u(t)^{-1} = ut (s \cdot x) (ut)^{-1} = ut (s \cdot x) t^{-1} (-u)
   \\
   = u +(-s \cdot x)  + (-u)
   = -s \cdot x = \phi_u(-x) = \phi_u(txt^{-1}).
 \end{multline*}
 Obviously $\phi_u$ is $s$-expansive.

  Let $\pr \colon \Delta \to \IZ/2$ be the projection.
  If $\pr(\overline{H})$ is trivial, we can choose $\phi_0$.  Suppose that
  $\pr(\overline{H})$ is non-trivial. 
  Then there is $u \in \IZ^n$ with $ut \in \overline{H}$.
  Consider any element $x \in \overline{H} \cap \IZ^n$.
  Then by assumption we can find $y \in \IZ^n$ with $x = s \cdot y$
  and hence $\phi_u(y) = x$.
  Consider any element of the form $xt$ which lies in $\overline{H}$.  Then
   $(xt) \cdot (ut) = x - u$ lies in $\overline{H} \cap \IZ^n$ and hence in
  $\im(\phi_u)$. Since $ut$ and $(xt) \cdot (ut)$ lie in the
  image of $\phi_u$, the same is true for $xt$.
  We have shown $\overline{H} \subseteq \im(\phi_u)$.

  One easily checks that the map
  $a \colon \IR^n \to \IR^n,\; x \mapsto s \cdot x + u/2$ is 
  $\phi_u$-equivariant.
  This finishes the proof of Lemma~\ref{lem:expansive_maps}.
\end{proof}
%%%%%%%%%%%%%%%%%%%%%%%%%%%%%%%%%%%%%%%%%%%%%%%%%%%%%%%%%%%%%%%%%%%%%

\subsection{The Farrell-Jones Conjecture for certain crystallographic groups of rank two}
\label{subsec:The_Farrell-Jones_Conjecture_for_certain_crystallographic_groups_of_rank_two}

We will handle the general case of a virtually finitely generated abelian group
by induction over its virtual cohomological dimension. For this purpose we have
to handle in Lemmas~\ref{pro:FJC_for_Z2} 
and~\ref{lem:FJC_for_certain_crystallographic_groups_of_rank_two} 
two special low-dimensional cases first.

The following elementary lemma is taken
from~\cite[Lemma~3.3.2]{Quinn(2012virtab)} (see also~\cite[Lemma~4.3]{Farrell-Hsiang(1983)}).
Denote by $d^{\euc}$ the Euclidean metric on $\IR^n$.

\begin{lemma} \label{lem:map_from_Z2_to_Z} Let $p$ be a prime and $C \subseteq (\IZ/p)^2$
  be a non-trivial cyclic subgroup.  Then there is a homomorphism
$$r  \colon \IZ^2 \to \IZ$$
such that the kernel of the map $(\IZ/p)^2 \to \IZ/p$ given by its reduction modulo $p$ is
$C$ and the induced map
$$r_{\IR} = r \otimes_{\IZ}\IR \colon \IR^2 =
\IZ^2 \otimes_{\IZ}\IR \to \IR = \IZ \otimes_{\IZ}\IR$$
satisfies
$$d^{\euc}\big(r_{\IR}(x_1),r_{\IR}(x_2)\bigr) \le \sqrt{2p} \cdot d^{\euc}(x_1,x_2)$$
for all $x_1,x_2 \in \IR^2$.
\end{lemma}

\begin{lemma}
\label{pro:FJC_for_Z2}
Both the $K$-theoretic and the $L$-theoretic FJC
hold for $\IZ^2$ and $\IZ^2 \rtimes_{-\id} \IZ/2$.
\end{lemma}
\begin{proof}
  Because of Theorem~\ref{the:subgroups} applied to
  $\IZ^2 \subseteq \IZ^2 \rtimes_{-\id} \IZ/2$
  it suffices to prove the claim for $\IZ^2 \rtimes_{-\id} \IZ/2$.  Because of
  Theorem~\ref{the:Farrell-Jones_Conjecture_for_Farrell-Hsiang_groups} it suffices to
  show that $\IZ^2 \rtimes_{\id}\IZ$ is a Farrell-Hsiang group with respect to the family
  $\VCyc$ in the sense of Definition~\ref{def:Farrell-Hsiang}.

  In the sequel we abbreviate $\Delta :=\IZ^2 \rtimes_{-\id} \IZ/2$. We have the obvious
  short exact sequence
$$1 \to \IZ^2 \xrightarrow{i} \Delta \xrightarrow{\pr} \IZ/2\to 1.$$
Fix a word metric $d_{\Delta}$ on $\Delta$. The map
$$\ev \colon \Delta \to \IR^2$$
given by the evaluation of the obvious isometric proper cocompact $\Delta$-action
on $\IR^2$ is by the \v{S}varc-Milnor Lemma (see~\cite[Proposition~8.19 in Chapter~I.8
on page~140]{Bridson-Haefliger(1999)}) a quasi-isometry if we equip $\IR^n$ with the
Euclidean metric $d^{\euc}$.  Hence we can find constants $C_1$ and $C_2$ such that for
all $g_1,g_2 \in \Delta$ we have
\begin{eqnarray}
d^{\euc}\bigl(\ev(g_1),\ev(g_2)\bigr) &\le & C_1 \cdot d_{\Delta}(g_1,g_2)  +  C_2.
\label{deuc_byC_1_cdot_d_Delta_plus_c_2}
\end{eqnarray}
Consider positive real numbers $R$ and $\epsilon$. Choose two different odd
prime numbers $p$ and $q$ satisfying
\begin{eqnarray}
  \frac{8 \cdot (C_1 \cdot R + C_2)^2}{\epsilon^2} \le p,q.
  \label{proof_for_Zn:choice_of_p_q}
\end{eqnarray}
For a natural number $s$ define $\Delta_{s}$ to be $\Delta/s\IZ^2$. We
have the obvious exact sequence
$$1 \to \IZ^2/s\IZ^2 = \big(\IZ/s\bigr)^2 \to  \Delta_{s} \xrightarrow{\pr_s} \IZ/2
\to 1.$$
The canonical projection $\alpha_{pq} \colon \Delta \to \Delta_{pq}$
will play the role of the map $\alpha_{R,\epsilon}$ appearing in
Definition~\ref{def:Farrell-Hsiang}.

Let $H \subseteq \Delta_{pq}$ be a $l$-hyperelementary subgroup for some prime
$l$. Since $p$ and $q$ are different, we can assume without loss of generality
$l \not= p$.  Then the canonical projection $\pi \colon \Delta_{pq} \to
\Delta_p$ sends $H \cap \IZ^2/pq\IZ^2$ to a cyclic subgroup $C$ of
$\IZ^2/p\IZ^2$.  Let $r \colon \IZ^2 \to \IZ$ be the homomorphism appearing in
Lemma~\ref{lem:map_from_Z2_to_Z} if $C$ is non-trivial and to be the projection
on the first factor if $C$ is trivial.  Let $\overline{H}$ be the preimage of
$H$ under the projection $\alpha_{pq} \colon \Delta \to \Delta_{pq}$.  In all
cases we get for $x_1,x_2 \in \IR^2$
\begin{eqnarray}
d^{\euc}\big(r_{\IR}(x_1),r_{\IR}(x_2)\bigr) & \le &\sqrt{2p} \cdot d^{\euc}(x_1,x_2)
\label{r_R-estimate}
\end{eqnarray}
and
\begin{eqnarray}
r\bigl(\overline{H} \cap \IZ^2\bigr)& \subseteq & p\IZ.
\label{r(H_cap_Z2/pqZ)_subset_pZ}
\end{eqnarray}

The homomorphism $r$ extends to a group homomorphism
$$\overline{r} := r \rtimes_{\-\id} \id_{\IZ/2} \colon \Delta
= \IZ^2 \rtimes_{- \id} \IZ/2 \to D_{\infty} = \IZ \rtimes_{-\id} \IZ/2.$$
 We conclude $\overline{r}(\overline{H}) \cap \IZ =
r\bigl(\overline{H} \cap \IZ^2\bigr) \subseteq p\IZ$
from~\eqref{r(H_cap_Z2/pqZ)_subset_pZ}.  Because of
Lemma~\ref{lem:expansive_maps}~\ref{lem:expansive_maps:subgroup_in_the_image}
we can find a $p$-expansive map
$$\phi \colon D_{\infty} \to D_{\infty}$$
and an affine map
$$a_{p,u} \colon \IR \to \IR, \quad x \mapsto p \cdot x +u$$
such that $a_{p,u}$ is $\phi$-equivariant and
\begin{eqnarray}
  \overline{r}(\overline{H}) & \subseteq & \im(\phi).
  \label{overline(r)(overline(H))_subseteq_im(phi)}
\end{eqnarray}

Let $E_H$ be the simplicial complex with underlying space $\IR$ whose set of
zero-simplices is $\{n/2 \mid n \in \IZ\}$. Equip $\IR$ with the standard
$D_{\infty}$-action given by translation with integers and
$-\id_{\IR}$.  Then the $D_{\infty}$-action on $E_H = \IR$ is a cell preserving
simplicial action. If $d^{l^1}$ is the $l^1$-metric on $E_H$, we get for all
$y_1,y_2$ in $E_H$
\begin{eqnarray}
  d^{l^1}(y_1,y_2) & \le & 2 \cdot d^{\euc}(y_1,y_2).
  \label{dl1_compared_with_deuc_on_R}
\end{eqnarray}

Define a map
$$f_H \colon \Delta \xrightarrow{\ev} \IR^2 \xrightarrow{r_{\IR}} \IR
\xrightarrow{(a_{p,u})^{-1}} E = \IR.$$ The map $\ev \colon \Delta \to \IR^2$ is
$\Delta$-equivariant. The map $r_{\IR} \colon \IR^2 \to \IR$ is $\overline{r}
\colon \IZ^2 \rtimes_{-\id} \IZ/2 \to D_{\infty} = \IZ \rtimes_{-\id}
\IZ/2$-equivariant.  Because
of~\eqref{overline(r)(overline(H))_subseteq_im(phi)} we can define an
$\overline{H}$-action on $\IR$ by requiring that $\overline{h} \in \overline{H}$
acts by multiplication with the element $u \in D_{\infty}$ which is uniquely
determined by $\overline{r}(\overline{h}) = \phi(u)$.  With respect to this
$\overline{H}$-action and the obvious $\overline{H}$-action on $\Delta$ the map
$f_H$ is $\overline{H}$-equivariant. All isotropy groups of the
$\overline{H}$-action on $E$ are virtually cyclic. We estimate for
$g_1$, $g_2 \in \Delta$ with $d_{\Delta}(g_1,g_2) \le R$
using~\eqref{deuc_byC_1_cdot_d_Delta_plus_c_2},~\eqref{proof_for_Zn:choice_of_p_q},%
~\eqref{r_R-estimate} and~\eqref{dl1_compared_with_deuc_on_R}
\begin{eqnarray*}
  d^{l^1}\big(f_H(g_1),f_H(g_2)\bigr)
  & \le &
  2 \cdot d^{\euc}\big(f_H(g_1),f_H(g_2)\bigr)
  \\
  & = &
  2 \cdot  d^{\euc}\big(a_{p,u}^{-1} \circ  r_{\IR} \circ \ev(g_1), a_{p,u}^{-1}
  \circ r_{\IR} \circ \ev(g_2)\bigr)
  \\
  & = & \frac{2}{p} \cdot d^{\euc}\big(r_{\IR} \circ \ev(g_1),r_{\IR}\circ \ev(g_2)\bigr)
  \\
   & \le &
  \frac{2}{p} \cdot  \sqrt{2p} \cdot d^{\euc}\big(\ev(g_1),\ev(g_2)\bigr)
  \\
  & \le &
  \frac{2 \cdot \sqrt{2}}{\sqrt{p}} \cdot
  \left(C_1 \cdot  d_{\Delta}\big(g_1,g_2\bigr) +C_2 \right)
  \\
  & \le &
  \frac{2 \cdot \sqrt{2}}{\sqrt{p}} \cdot  \left(C_1 \cdot  R +C_2 \right)
  \\
  & \le &
  \epsilon.
\end{eqnarray*}

We conclude that $\Delta$ is a Farrell-Hsiang group in the sense of
Definition~\ref{def:Farrell-Hsiang} with respect to the family $\VCyc$.
Hence Lemma~\ref{pro:FJC_for_Z2} follows from
Theorem~\ref{the:Farrell-Jones_Conjecture_for_Farrell-Hsiang_groups}.
\end{proof}

\begin{lemma}
  \label{lem:FJC_for_certain_crystallographic_groups_of_rank_two}
  Let $\Delta$ be a crystallographic group of rank two which possesses
  a normal  infinite cyclic subgroup.
  Then both the $K$-theoretic and the $L$-theoretic FJC hold for $\Delta$.
\end{lemma}
\begin{proof}
We will use induction over the order of $F = F_{\Delta}$.
If $F$ is trivial, then $\Delta = \IZ^2$
and the claim follows from Lemma~\ref{pro:FJC_for_Z2}.
The induction step for $|F| \ge 2$
is done as follows.

Because of Lemma~\ref{pro:FJC_for_Z2} we can assume in the sequel that $\Delta$
is different from $\IZ^2 \rtimes_{-\id} \IZ/2$.  Let $\calf$ be the family of
subgroups $K \subseteq \Delta$ which are virtually cyclic or satisfy
$\pr_{\Delta}(K) \not= F_{\Delta}$ for the projection $\pr_{\Delta} \colon
\Delta \to F_{\Delta}$.  Because of the induction hypothesis, the Transitivity
Principle~\ref{the:transitivity} and
Theorem~\ref{the:Farrell-Jones_Conjecture_for_Farrell-Hsiang_groups} it suffices
to show that $\Delta$ is a Farrell-Hsiang group with respect to the family
$\calf$ in the sense of Definition~\ref{def:Farrell-Hsiang}.

We have the canonical exact sequence associated to a crystallographic
group
$$1 \to A = A_{\Delta} \xrightarrow{i} \Delta \xrightarrow{\pr} F = F_{\Delta} \to 1.$$
Next we analyze the conjugation action $\rho \colon F \to \aut(A)$.
Since $\Delta$ is crystallographic,  $\rho \colon F \to \aut(A)$  is injective.
By assumption $A \cong \IZ^2$ and we can find a normal infinite cyclic
subgroup $C\subset \Delta$.

Next we show that $A$ contains precisely two maximal infinite cyclic subgroups
which are $F$-invariant.

By rationalizing we obtain a $2$-dimensional
rational representation $A_{\IQ} := A \otimes_{\IZ} \IQ$ of $F$.
It contains a one-dimensional $F$-invariant $\IQ$-subspace,
namely $C_{\IQ} := (C \cap A) \otimes_{\IZ} \IQ$.
Hence $A_{\IQ}$ is a direct summand
of two one-dimensional rational representations $V_1 \oplus V_2$. For each
$V_i$ there must be a homomorphism $\sigma_i \colon F \to \{\pm 1\}$
such that $f \in F$ acts on $V_i$ by multiplication with $\sigma_i(f)$.
Hence we can find two elements $x_1$ and $x_2 \in A$ such that
$x_1$ and $x_2$ are $\IZ$-linearly independent and the cyclic subgroups generated by them
are $F$-invariant. Let $C_i$ be the unique maximal infinite cyclic subgroups of
$A$ which contains $x_i$.  Then $C_1$ and $C_2$ are $F$-invariant and
$$A = C_1 \oplus C_2.$$
The $F$-action on $C_i$ is given by the homomorphism
$\sigma_i \colon F \to \{\pm 1\}$. Since $\rho \colon F \to \aut(A)$ is injective
and $\Delta$ is not isomorphic to $\IZ^2\rtimes_{-\id} \IZ/2$,
the homomorphisms $\sigma_1$ and $\sigma_2$  from $F$ to $\{\pm 1\}$ must be different
and $F$ is isomorphic to $\IZ/2$ or $\IZ/2 \oplus \IZ/2$.

It remains to show that any maximal infinite cyclic subgroup $D$ which is $F$-invariant is
equal to $C_1$ or $C_2$. Given such $D$, we obtain an $F$-invariant $\IQ$-subspace
$D_{\IQ} \subseteq A_{\IQ}$.  Since $C_1$, $C_2$ and $D$
are maximal infinite cyclic subgroups of
$A$, it suffices to show $D_{\IQ} =(C_i)_{\IQ}$ for some $i \in \{1,2\}$. Suppose the
contrary.  Then for $i = 1,2$ the projection $A_{\IQ} \to  (C_i)_{\IQ}$
induces an isomorphism
$D_{\IQ} \to (C_i)_{\IQ}$. Hence $(C_1)_{\IQ}$ and $(C_2)_{\IQ}$ are isomorphic.
This implies $\sigma_1 = \sigma_2$,
a contradiction.  Hence we have shown that $A$ contains precisely two maximal infinite
cyclic subgroups which are $F$-invariant.

If $C \subseteq A$ is a maximal  infinite cyclic subgroup which
is invariant under the $F$-action,
then it is normal in $\Delta$ and we can consider the projection
$$\widehat{\xi}_C \colon \Delta \to \Delta/C.$$
We obtain a commutative diagram
$$\xymatrix{
1 \ar[r] & A \ar[d]^{\xi_C} \ar[r]^{i} & \Delta \ar[r]^{\pr} \ar[d]^{\widehat{\xi}_C}
& F \ar[r] \ar[d]^{\id} & 1
\\
1 \ar[r] & A/C \ar[r]^{\overline{i}}
& \Delta/C \ar[r]^{\overline{\pr_C}}  & F \ar[r]  & 1}
$$
where the vertical maps are the obvious projections.

Since $\Delta/C$ is virtually abelian with virtual cohomological dimension
one, we can find an epimorphism $\overline{\mu_C} \colon \Delta/C \to \Delta'_C$
to a crystallographic group of rank one
whose kernel is finite. We obtain a commutative diagram
$$\xymatrix{
1 \ar[r] & A/C \ar[r]^{\overline{i}} \ar[d]^{\mu_C}
& \Delta/C \ar[d]^{\widehat{\mu}_C}\ar[r]^{\overline{\pr_C}}  & F \ar[r]  \ar[d]& 1
\\
1 \ar[r] & A_{\Delta'_C} \ar[r] & \Delta'_C \ar[r] & F_{\Delta'_C} \ar[r]  & 1}
$$
The map $\mu_C$ is injective and $\Delta'_C$ is either
$\IZ$ or $D_{\infty} = \IZ \rtimes_{-\id} \IZ/2$.
Define homomorphisms
\begin{eqnarray*}
& \widehat{\nu}_C := \widehat{\mu}_C \circ \widehat{\xi}_C \colon \Delta \to \Delta'_C;&
\\
&\nu_C := \mu_C \circ \xi_C \colon A \to A_{\Delta'_C}.&
\end{eqnarray*}

Consider word metrics $d_{\Delta}$ and $d_{\Delta'_C}$.  Recall that $\widehat{\nu}_C$ is
a surjective group homomorphism and the quasi-isometry type of a word metric is
independent of the choice of a finite set of generators.  Hence we can find constants
$C_1$ and $C_2$ such that for every (of the finitely many) maximal infinite cyclic
subgroups $C \subseteq A$ which are invariant under the $F$-action and for all
$g_1, g_2 \in \Delta$ we get
\begin{eqnarray}
d_{\Delta'_C}\big(\widehat{\nu}_C(g_1),\widehat{\nu}_C(g_1)\bigr)
& \le &
C_1 \cdot d_{\Delta}\bigl(g_1,g_2) + C_2.
\label{d_Delta_versus_d_Delta_prime}
\end{eqnarray}
Equip $\IR$ with the standard action of $\Delta'_C$.  
Let $E$ be the simplicial complex
whose underlying space is $\IR$ and whose set of $0$-simplices is
$\{n/2 \mid n \in \IZ\}$.
The $\Delta'_C$-action above is a cell preserving simplicial action on $E$.
If $d^{l^1}$ is the $l^1$-metric on $E$, we get for $y_1,y_2 \in \IR$
\begin{eqnarray}
d^{l^1}(y_1,y_2) & \le & 2 \cdot d^{\euc}(y_1,y_2).
\label{estimate_dl1_versus_deuc_for_E_is_R}
\end{eqnarray}
Let the  map
$$\ev_C \colon \Delta_C' \to \IR$$
be given by the evaluation of the isometric proper cocompact $\Delta'_C$-action
on $\IR$. By the \v{S}varc-Milnor Lemma (see~\cite[Proposition~8.19 in Chapter~I.8
on page~140]{Bridson-Haefliger(1999)})  we can find constants $C_3$ and $C_4$
such that for every  (of the finitely many) maximal infinite cyclic
subgroups $C \subseteq A$ which are invariant under the $F$-action and
all $g_1,g_2 \in \Delta'_C$ we have
\begin{eqnarray}
d^{\euc}\bigl(\ev_C(g_1),\ev_C(g_2)\bigr) & \le & C_3 \cdot d_{\Delta'_C}(g_1,g_2)  +  C_4.
\label{estimate_for_ev_in_rank_two_case}
\end{eqnarray}
We conclude from~\eqref{d_Delta_versus_d_Delta_prime},%
~\eqref{estimate_dl1_versus_deuc_for_E_is_R}
and~\eqref{estimate_for_ev_in_rank_two_case}
that we can  find constants $D_1> 0$ and $D_2> 0$ such that
for every  maximal infinite cyclic
subgroup $C \subseteq A$ which is invariant under the $F$-action and
all $g_1,g_2 \in \Delta$ we have
\begin{eqnarray}
d^{l^1}\big(\ev_C \circ \widehat{\nu}_C(g_1),\ev_C  \circ \widehat{\nu}_C(g_1)\bigr)
& \le &
D_1 \cdot d_{\Delta}(g_1,g_2) + D_2
\label{d_euc_versus_d_Delta}
\end{eqnarray}

Consider positive real numbers $R$ and $\epsilon$.
We can choose an odd  prime $p$ satisfying
\begin{eqnarray}
p & \ge & \frac{2 \cdot (D_1 \cdot R +D_2)}{\epsilon}. \label{choice_for_p_rank_two}
\end{eqnarray}

Put $A_p = A/pA$ and $\Delta_p = \Delta/pA$. We obtain an exact sequence
$$1 \to A_p \to \Delta_p \xrightarrow{\pr_p} F \to 1.$$
The projection $\alpha_p \colon \Delta \to \Delta_p$ will play the role of the map
$\alpha_{R,\epsilon}$ appearing in Definition~\ref{def:Farrell-Hsiang}.

Let $H \subseteq \Delta_p$ be a hyperelementary subgroup.
If $\pr_p(H)$ is not $F$, then $H$ belongs to $\calf$ and we can take for
$f_H$ the map $\Delta \to \{\bullet\}$.  Hence it remains to treat the case
$\pr_p(H) = F$.

Next show that $A_p \cap H$ is cyclic.  Choose a prime $q$ and an exact sequence
$1 \to D \to H \to P \to 1$ for a $q$-group $P$ and a cyclic group of order
prime to $q$. If $p$ and $q$ are different, $A_p \cap H$ embeds into $D$ and is
hence cyclic.  Suppose that $p = q$.  It suffices to show that $A_p \cap H$ is
different from $A_p$, or, equivalently, $H \not= \Delta_p$.  Suppose the
contrary, i.e., $H = \Delta_p$. 
Because the order of $F$ is $2$ or $4$ and $p$ is odd this implies that 
the composite $A_p \to H \to P$ is an isomorphism. 
Hence there is a retraction for the inclusion $A_p \to \Delta_p$.
This implies that the conjugation action of $F$ on $A_p$ is trivial.  
We have
already explained that there are homomorphisms 
$\sigma_i \colon F \to \{\pm 1\}$
such that $f \in F$ acts on $C_i$ by multiplication with 
$\sigma_i(f)$ and that
these two homomorphisms must be different.  The induced $F$-action on $A_p =
(C_1)_p \oplus (C_2)_p$ is analogous.  Since $p$ is odd, this leads to a
contradiction.  Hence $H \cap A_p$ is cyclic.

Since $\pr_p(H) = F$, the cyclic subgroup $H \cap A_p$ is invariant under the
$F$-action on $A_p = (C_1)_p \oplus (C_2)_p$. Hence $A_p \cap H$ must be
contained in $(C_i)_p = \alpha_p(C_i)$ for some $i \in \{1,2\}$. We put
$C = C_i$ and $\Delta' = \Delta_{C_i}'$ in the sequel.

Let $\overline {H}$ be the preimage of $H$ under the projection
$\alpha_p \colon \Delta \to \Delta_p$.
Then we get for the homomorphism $\xi_{C} \colon A \to A/C$
\begin{eqnarray*}
\xi_C(\overline{H} \cap A) & \subseteq & p (A/C).
\end{eqnarray*}
Since the map $\mu_C \colon A/C \to A_{\Delta'}$ is injective,
we conclude for the homomorphism $\nu_{C} \colon A \to A_{\Delta'}$
\begin{eqnarray*}
\nu_C(\overline{H} \cap A) \cap A_{\Delta'} & \subseteq & p A_{\Delta'}.
\end{eqnarray*}
Because $p$ is odd and $|F|$ is $2$ or $4$ this implies that
\begin{eqnarray*}
\widehat{\nu}_C(\overline{H}) \cap A_{\Delta'} & \subseteq & p A_{\Delta'}.
\end{eqnarray*}
Because of Lemma~\ref{lem:expansive_maps}~\ref{lem:expansive_maps:affine_map}
and~\ref{lem:expansive_maps:subgroup_in_the_image}
we can find a $p$-expansive map
$$\phi \colon \Delta' \to \Delta'$$
and an affine map
$$a_{p,u} \colon \IR \to \IR, \quad x \mapsto p \cdot x  + u,$$
such that $a_{p,u}$ is $\phi$-equivariant and
\begin{eqnarray}
\nu_C(\overline{H}) & \subseteq & \im(\phi).
\label{mu_C(overline(H))_subseteq_im(phi)}
\end{eqnarray}

Let $E_H$ be the simplicial complex whose underlying space is $\IR$ and whose set of
$0$-simplices is $\{n/2 \mid n \in \IZ\}$. The standard $\Delta'_C$-action is
a cell preserving simplicial action on $E_H$. We define the map
$$f_H \colon \Delta \xrightarrow{\widehat{\nu}_C} \Delta'
\xrightarrow{\ev} \IR \xrightarrow{a_{p,u}^{-1}} E_H = \IR$$
{}From~\eqref{estimate_dl1_versus_deuc_for_E_is_R},~\eqref{d_euc_versus_d_Delta}
and~\eqref{choice_for_p_rank_two} and we conclude for
$g_1,g_2 \in \Delta$ satisfying $d_{\Delta}(g_1,g_2) \le R$
\begin{eqnarray*}
d^{l^1}\big(f_H(g_1),f_H(g_1)\bigr)
& = &
2 \cdot d^{\euc}\big(f_H(g_1),f_H(g_1)\bigr)
\\
& = &
2 \cdot d^{\euc} \big(a_{p,u}^{-1}\circ \ev_C \circ \widehat{\nu}_C(g_1),a_{p,u}^{-1}
\circ  \ev_C  \circ \widehat{\nu}_C(g_1)\bigr)
\\
& = &
\frac{2}{p} \cdot d^{\euc} \big(\ev_C \circ \widehat{\nu}_C(g_1),\ev_C
\circ \widehat{\nu}_C(g_1)\bigr)
\\
&\le &
 \frac{2}{p} \cdot \left(D_1 \cdot d_{\Delta}(g_1,g_2) + D_2\right)
\\
& \le &
\frac{2 \cdot (D_1 \cdot R + D_2}{p}
\\
& \le & \epsilon.
\end{eqnarray*}

Because of~\eqref{mu_C(overline(H))_subseteq_im(phi)} we can define
a $\overline{H}$-action on $E_H$ by requiring
that $\overline{h} \in \overline{H}$ acts by the unique element $g \in \Delta'$
which is mapped under the injective homomorphism $\phi \colon \Delta' \to \Delta'$ to
$\nu_C(\overline{h})$. Then the map $f_H \colon \Delta \to E$ is
$\overline{H}$-equivariant and all isotropy groups of the $\overline{H}$-action on
$E$ are virtually cyclic.

We conclude that $\Delta$ is a Farrell-Hsiang group in the sense of
Definition~\ref{def:Farrell-Hsiang} with respect to the family $\VCyc$.
Hence Lemma~\ref{lem:FJC_for_certain_crystallographic_groups_of_rank_two}
follows from  Theorem~\ref{the:Farrell-Jones_Conjecture_for_Farrell-Hsiang_groups}.
\end{proof}

\begin{lemma}\label{lem:extensions_with_virtually_cyclic_as_kernel}
Let $1 \to V \to G \to Q \to 1$ be an exact sequence of groups.
Suppose that $Q$  satisfies the $K$-theoretic FJC
and that $V$ is virtually cyclic.  Then $G$  satisfies the $K$-theoretic FJC

The same is true for the for the $L$-theoretic {FJC}.
\end{lemma}
\begin{proof}
  By the Transitivity Principle~\ref{the:transitivity} it suffices to show that $G$
  satisfies the $K$-theoretic FJC in the special case that $Q$ is virtually cyclic.
  Since this is obvious for finite $V$, we can assume that $V$ is infinite.
  Let $C'$ be an infinite cyclic subgroup of
  $V$. Let $C$ be the intersection $\bigcap_{\phi \in \aut(V)} \phi(V)$.  Since $V$
  contains only finitely many subgroups of a given index, this is a finite intersection
  and hence $C$ is a characteristic subgroup of $V$ which is infinite cyclic and has
  finite index. Hence $C$ is a normal infinite cyclic subgroup of the virtually finitely
  generated abelian group $G$ and the  virtual cohomological dimension of $G$ is
  two.  There exists a homomorphism
  with finite kernel onto a crystallographic group $G \to G'$ (see for
  instance~\cite[Lemma~4.2.1]{Quinn(2012virtab)}).  The rank of $G'$ is two and $G'$ contains a
  normal infinite cyclic subgroup.  Hence $G$ satisfies the $K$-theoretic FJC because of
  Corollary~\ref{cor:from_Transitivity_Principle} and
  Lemma~\ref{lem:FJC_for_certain_crystallographic_groups_of_rank_two}.
\end{proof}

%%%%%%%%%%%%%%%%%%%%%%%%%%%%%%%%%%%%%%%%%%%%%%%%%%%%%%%%%%%%%%%%%%%%%

\subsection{The Farrell-Jones Conjecture for virtually finitely generated abelian groups}
\label{subsec:Virtually_finitely_generated_abelian_groups}

In this subsection we finish the proof of
Theorem~\ref{the:The_Farrell-Jones_Conjecture_for_virtually_finitely_generated_abelian_groups}.
\begin{proof}[Proof of
Theorem~\ref{the:The_Farrell-Jones_Conjecture_for_virtually_finitely_generated_abelian_groups}]
We use induction over the virtual cohomological dimension $n$ of  the virtually
finitely generated abelian group $\Delta$ and subinduction over the minimum of the orders
of finite groups $F$ for which there exists
an exact sequence $1 \to \IZ^n \to \Delta \to F \to 1$.
The induction beginning $n \le 1$ is trivial since then $\Delta$ is virtually cyclic.

In the induction step we can assume that $\Delta$ is a crystallographic group of rank $n$
because of Corollary~\ref{cor:from_Transitivity_Principle}
since a virtually finitely generated abelian group
possesses an epimorphism with finite kernel onto a crystallographic group
(see for instance~\cite[Lemma~4.2.1]{Quinn(2012virtab)}).
Hence we have to prove that a crystallographic group $\Delta$ of rank $n \ge 2$
satisfies both the $K$- and $L$-theoretic FJC provided that every
virtually finitely generated abelian  group  $\Delta'$
satisfies both the $K$- and $L$-theoretic FJC if $\vcd(\Delta') < n$ 
or if there exists an extension $1 \to \IZ^n  \to \Delta' \to F \to 1$ 
for a finite group $F$ with $|F| < |F_{\Delta}|$.

Because of the induction hypothesis and
Lemma~\ref{lem:extensions_with_virtually_cyclic_as_kernel}
we can assume from now on that $\Delta$ does not contain
a normal infinite cyclic subgroup $C$.

Let $\calf$ be the family of subgroups of $\Delta$ which contains all subgroups
$\Delta' \subseteq \Delta$ such that $\vcd(\Delta') <\vcd(\Delta)$ holds or
that both $\vcd(\Delta') = \vcd(\Delta)$ and $|F_{\Delta'}| < |F_{\Delta}|$ hold.
By the induction hypothesis, the Transitivity Principle~\ref{the:transitivity}
and Theorem~\ref{the:Farrell-Jones_Conjecture_for_Farrell-Hsiang_groups}
it suffices to show that $\Delta$ is a Farrell-Hsiang group in the sense of
Definition~\ref{the:Farrell-Jones_Conjecture_for_Farrell-Hsiang_groups}
for the family $\calf$.

Fix a word metric $d_{\Delta}$ on $\Delta$. Let the map
$$\ev \colon \Delta \to \IR^n$$
be given by the evaluation of the cocompact proper
isometric $\Delta$-operation on $\IR^n$. It is by the \v{S}varc-Milnor Lemma
(see~\cite[Proposition~8.19 in Chapter~I.8 on page~140]{Bridson-Haefliger(1999)}) a
quasi-isometry if we equip $\IR^n$ with the Euclidean metric $d^{\euc}$.  Hence we can
find constants $C_1$ and $C_2$ such that for all $g_1,g_2 \in \IZ^2$ we have
\begin{eqnarray}
d^{\euc}\bigl(\ev(g_1),\ev(g_2)\bigr)  &\le & C_1 \cdot d_{\Delta}(g_1,g_2)  +  C_2.
\label{d_euc(ev(g_1),ev(g_2)_compared_to_word_metric}
\end{eqnarray}

Consider real numbers $R > 0$ and $\epsilon >0$.
Since $\Delta$ acts properly, smoothly and cocompactly  on $\IR^n$, we can equip $\IR^n$
with the structure of a simplicial complex such that
the $\Delta$-action is cell-preserving and simplicial.
Denote this simplicial complex by $E$.
The induced $l^1$-metric $d^{l^1}$ and the Euclidean metric
$d^{\euc}$ induce the same topology since $E$ is bounded locally finite.
Hence we can find $\delta > 0$ such that
\begin{eqnarray}
d^{\euc}\bigl(y_1,y_2) \le \delta & \implies & d^{l^1}\bigl(y_1,y_2)\bigr) \le \epsilon
\label{comparing_d_euc_and_d_l1}
\end{eqnarray}
holds for all $y_1,y_2$.

We can write $|F_{\Delta}| = 2^k \cdot l$ for some odd natural number $l$
and non-negative integer $k$.
By Dirichlet's Theorem
(see~\cite[Lemma~3 in III.2.2 on page~25]{Serre(1993)})
there exist infinitely many primes which are congruent to
$-1$ modulo $4l$.
Hence we can choose a prime number $p$ satisfying
\begin{equation*}
p  \ge  \frac{C_1 \cdot R + C_2}{\delta}
\qquad \text{and} \qquad p  \equiv  -1 \mod 4l.
\end{equation*}
Now choose $r$ such that
\begin{eqnarray*}
p^r & \equiv & 1 \mod |F_{\Delta}|.
\end{eqnarray*}
Since $A = A_{\Delta}$ is a characteristic subgroup of $\Delta$,
also $p^r \cdot A$ is  characteristic and hence a normal subgroup 
of $\Delta$.
Define groups
\begin{eqnarray*}
A_{p^r} & :=& A/(p^r\cdot A);
\\
\Delta_{p^r} & := & \Delta/(p^r \cdot A).
\end{eqnarray*}
Let $\pr \colon \Delta \to F_{\Delta}$ and $\pr_{p^r} \colon \Delta_{p^r} \to F_{\Delta}$
be the canonical projections.
Let the epimorphism
$$\alpha_{p^r} \colon \Delta \to \Delta_{p^r}$$
be the canonical projection. It will play the role of the map $\alpha_{R,\epsilon}$
appearing in Definition~\ref{def:Farrell-Hsiang}.

Consider a hyperelementary subgroup $H \subseteq \Delta_{p^r}$.
Let $\overline{H}$ be the preimage of $H$ under
$\alpha_{p^r} \colon \Delta\to  \Delta_{p^r}$.
Suppose that $\pr_{p^r}(H) \not= F_{\Delta}$.
Then $\overline{H}$ is a crystallographic group with
$\vcd(\overline{H}) = \vcd(\Delta)$ and $|F_{\overline{H}}| < F_{\Delta}$.
 By induction hypothesis
$\overline{H}$ satisfies both the $K$- and $L$-theoretic FJC and hence
belongs to $\calf$. Therefore  we can  take as the desired $\overline{H}$-map
in this case the projection to the one-point-space
$$f_H \colon \overline{H} \to E_H:= \{\bullet\}.$$
It remains to consider the case, where $\pr_{p^r}(H) = F_{\Delta}$. 
We conclude
from~\cite[Proposition~2.4.2]{Quinn(2012virtab)} that 
$H \cap A_{p^r} = \{0\}$ since
$\Delta$ contains no infinite normal cyclic subgroup and the prime 
number $p$ satisfies $p \equiv  3 \mod 4$
and $p  \not \equiv  1 \mod q$ for any odd prime $q$ dividing $|F_\Delta|$
and hence $l$.

Since $p^r \equiv  1 \mod |F_{\Delta}|$ we can choose by
Lemma~\ref{lem:expansive_maps}~\ref{lem:expansive_maps:existence}
a $p^r$-expansive homomorphism $\phi \colon \Delta \to \Delta$.
Consider the composite  
$\alpha_{p^r} \circ \phi \colon \Delta \to \Delta_{p^r}$.
Its restriction to $A = A_{\Delta}$ is trivial.
Hence there is a map $\overline{\phi }\colon F_{\Delta}  \to \Delta_{p^r}$ 
satisfying
\begin{eqnarray*}
\alpha_{p^r} \circ \phi & = & \overline{\phi} \circ \pr;
\\
\pr_{p^r} \circ \overline{\phi} & = & \id_{F_{\Delta}}.
\end{eqnarray*}
Hence $\overline{\phi}$ is a splitting of the exact
sequence $1 \to A_{p^r} \to \Delta_{p^r} \to F_{\Delta} \to 1$.
The homomorphism $\pr_{p^r}|_H \colon H \to F_{\Delta}$ is an isomorphism
and hence defines a second splitting. Since the order of the finite group
$A_{p^r}$ and the order of $F_{\Delta}$ are prime,
$H^1\bigl(F_{\Delta};A_{p^r}\bigr)$ vanishes
(see~\cite[Corollary~10.2 in Chapter~III on page~84]{Brown(1982)}).  Hence
the subgroups $H$ and $\im(\alpha_{p^r} \circ \phi) = \im(\overline{\phi})$
are conjugated in $\Delta_{p^r}$
(see~\cite[Corollary~3.13 in Chapter~IV on page~93]{Brown(1982)}).
Because of Remark~\ref{rem:Farrell-Hisiang_plus_conjugation} 
we can assume without loss of generality that $H = \im(\alpha_{p^r} \circ \phi)$.
Next we show for $\overline{H} := \alpha_{p^r}^{-1}(H)$
\begin{eqnarray}
\overline{H} & = & \im(\phi).
\label{overlineH_is_im(phi)}
\end{eqnarray}
Because of $H = \im(\alpha_{p^r} \circ \phi)$ it suffices to prove
$\alpha_{p^r}^{-1}\bigl(\im(\alpha_{p^r} \circ \phi)\bigr) \subseteq \im(\phi)$.
Consider $g \in \alpha_{p^r}^{-1}\bigl(\im(\alpha_{p^r} \circ \phi)\bigr)$.
Choose $g_0 \in \Delta$ with $\alpha_{p^r}(g) = \alpha_{p^r}\bigl(\phi(g_0)\bigr)$.
We conclude that $\alpha_{p^r}\bigl(g \cdot \phi(g_0)^{-1}\bigr)$ is trivial.
Hence we can find $a \in A$ with $\phi(a) = p^r \cdot a = g \cdot \phi(g_0)^{-1}$.
This implies $g = \phi(a\cdot g_0)$. Hence~\eqref{overlineH_is_im(phi)} is true.

By Lemma~\ref{lem:expansive_maps}~\ref{lem:expansive_maps:affine_map}
there exists an element $u \in \IR$ such that the affine map
$$a_{p^r,u} \colon \IR^n \to \IR^n, \quad x \mapsto p^r \cdot y + u$$
is $\phi$-linear.
Consider the composite
$$f_H \colon \Delta \xrightarrow{\ev} \IR^n \xrightarrow{(a_{p^r,u})^{-1}}
E_H := \IR^n.$$
We get from~\eqref{d_euc(ev(g_1),ev(g_2)_compared_to_word_metric}
for $g_1,g_2 \in \Delta$ with $d_{\Delta}(g_1,g_2) \le R$
\begin{eqnarray*}
d^{\euc}\bigl(f_H(g_1), f_H(g_2)\bigr)
& = &
d^{\euc}\bigl((a_{p^r,u})^{-1} \circ \ev(g_1), (a_{p^r,u})^{-1} \circ \ev(g_2)\bigr)
\\
& = &
\frac{1}{p^r} \cdot d^{\euc}\bigl(\ev(g_1), \ev(g_2)\bigr)
\\
& \le &
\frac{1}{p^r} \cdot \left(C_1 \cdot d_{\Delta}(g_1,g_2) + C_2\right)
\\
& \le &
\frac{1}{p^r} \cdot \left(C_1 \cdot R+ C_2\right).
\end{eqnarray*}
Our choice of $p$ guarantees
$$\frac{1}{p^r} \cdot \left(C_1 \cdot R+ C_2\right) \le \delta,$$
where $\delta$ is the number appearing in~\eqref{comparing_d_euc_and_d_l1}.
We conclude from~\eqref{comparing_d_euc_and_d_l1} for all 
\mbox{$g_1,g_2 \in \Delta$}
\begin{eqnarray*}
d_{\Delta}(g_1,g_2)  \le R
& \implies &
d^{l^1}\bigl(f_H(g_1),f_H(g_2)\bigr) \le \epsilon.
\end{eqnarray*}
Because of~\eqref{overlineH_is_im(phi)} we can define a cell preserving simplicial
$\overline{H}$-action on the simplicial complex $E_H$
 by requiring that the action of $h \in \overline{H}$ is
given by the action of $g \in \Delta$ for the element uniquely determined by
$\phi(g) = h$. The isotropy groups of this $H$-action on $E$
are all finite and hence belong to
$\calf$.  The map $f_H\colon \Delta \to E_H$ is $H$-equivariant.

We conclude that $\Delta$ is a Farrell-Hsiang group in the sense of
Definition~\ref{def:Farrell-Hsiang} with respect to the family $\calf$. 
Now
Theorem~\ref{the:The_Farrell-Jones_Conjecture_for_virtually_finitely_generated_abelian_groups} follows from Theorem~\ref{the:Farrell-Jones_Conjecture_for_Farrell-Hsiang_groups}.
\end{proof}

%%%%%%%%%%%%%%%%%%%%%%%%%%%%%%%%%%%%%%%%%%%%%%%%%%%%%%%%%%%%%%%%%%%%%
%%%%%%%%%%%%%%% Section 4: Special affine groups %%%%%%%%%%%%%%%%%%%%
%%%%%%%%%%%%%%%%%%%%%%%%%%%%%%%%%%%%%%%%%%%%%%%%%%%%%%%%%%%%%%%%%%%%%

\typeout{---------- Special affine groups  --------------------------}

\section{Irreducible special affine groups}
\label{sec:Irreducible_special_affine_groups}

In this section we prove the $K$-theoretic and the $L$-theoretic
Farrell-Jones Conjecture with additive categories as coefficients with
respect to $\VCyc$ for irreducible special affine groups. This will be
the key ingredient and step in proving the $K$-theoretic and the
$L$-theoretic Farrell-Jones Conjecture with coefficients in an
additive category with respect to $\VCyc$ for virtually poly-$\IZ$
groups.

The irreducible special affine groups will play in the proof for virtually
poly-$\IZ$-groups the analogous role as the crystallographic groups of rank two
which contain a normal infinite cyclic subgroup played in the proof for
virtually finitely generated abelian groups. The general structure of the proof
for virtual poly-$\IZ$ group is similar but technically much more sophisticated
and complicated than in the case of virtually finitely generated abelian
groups. It relies on the fact that we do know the claim already for virtually
finitely generated abelian groups.  
Our proof is inspired by the one appearing
in Farrell-Hsiang~\cite{Farrell-Hsiang(1981b)} and
Farrell-Jones~\cite{Farrell-Jones(1988b)}.

%%%%%%%%%%%%%%%%%%%%%%%%%%%%%%%%%%%%%%%%%%%%%%%%%%%%%%%%%%%%%%%%%%%%%

\subsection{Review of (irreducible) special affine groups}
\label{subsec:Review_of_(irreducible)_special_affine_groups}

In this subsection we briefly collect some basic facts about
(irreducible) special affine groups. We will denote by $\vcd$ the
virtual cohomological dimension of a group (see~\cite[Section~11 in
Chapter VIII]{Brown(1982)}).  The following definition is equivalent
to Definition~4.7 in~\cite{Farrell-Jones(1993a)}.

\begin{definition}[(Irreducible) special affine group]
  \label{def:special_affine_group}
  A group $\Gamma$ is called a \emph{special affine group} of rank
  $(n+1)$ if there exists a short exact sequence
$$1 \to \Delta \to  \Gamma \to D\to 1 $$
and an action $\rho' \colon \Gamma \times \IR^n \to \IR^n$ by affine
motions of $\IR^n$ satisfying:

\begin{enumerate}

\item $D$ is either the infinite cyclic group $\IZ$ or the infinite
  dihedral group $D_{\infty}$;

\item The restriction of $\rho'$ to $\Delta$ is a cocompact isometric
  proper action of $\Delta$.

\end{enumerate}

We call a special affine group \emph{irreducible} if for any
epimorphism $\Gamma \to \Gamma'$ onto a virtually finitely generated
abelian group $\Gamma'$ we have $\vcd(\Gamma') \le 1$.

\end{definition}

Notice that the group $\Delta$ appearing in
Definition~\ref{def:special_affine_group} is a crystallographic group
of rank $n$.  Let
\[
\rho'' \colon D \times \IR \to \IR
\]
be the isometric cocompact proper standard action which is given by
translations with integers and $-\id$. We will consider the action
\begin{eqnarray}
  & \rho \colon \Gamma \times \IR^{n+1} \to \IR^{n+1} &
  \label{action_of_A_on_R(n_plus_1)}
\end{eqnarray}
given by the diagonal action for the $\Gamma$-action $\rho'$ on
$\IR^n$ and the $\Gamma$-action on $\IR$ coming from the epimorphism
$\Gamma \to D$ and the $D$-action $\rho''$ on
$\IR$. This $\Gamma$-action $\rho$ is a proper cocompact action by
affine motions and is not necessarily an isometric action.

%%%%%%%%%%%%%%%%%%%%%%%%%%%%%%%%%%%%%%%%%%%%%%%%%%%%%%%%%%%%%%%%%%%%%

\subsection{Some homological computations}
\label{subsec:Some_homological_computations}

Let $A =
A_{\Delta}$ be the unique and hence characteristic subgroup of
$\Delta$ which is abelian, normal and equal to its own centralizer in
$\Delta$. Since it is a characteristic subgroup, it is normal in both
$\Delta$ and $\Gamma$. Define
$$Q:= \Gamma/A.$$
Then we obtain exact sequences
\begin{eqnarray}
  & 1 \to A \to \Gamma \xrightarrow{\pr}  Q \to 1; &
  \label{A_to_Gamma_to_Q}
  \\
  & 1 \to F_{\Delta} \to Q \xrightarrow{\pi} D \to 1, &
  \label{F_to_Q_to_D}
\end{eqnarray}
where $F_{\Delta}$ is the finite group $\Delta/A$. In particular $Q$
is an infinite virtually cyclic group.

\begin{lemma} \label{lem:Hast(Q;A)_finite_irreducible_case}
    Let  $\Gamma$ be a special affine group. Consider $A$ as $\IZ
    Q$-module by the conjugation action associated to the exact
    sequence~\eqref{A_to_Gamma_to_Q}. Then:

   \begin{enumerate}

   \item \label{lem:Hast(Q;A)_finite_irreducible_case:characterization}
   $\Gamma$ is irreducible if and only if for any  subgroup
   $\Gamma_0 \subseteq \Gamma$ of finite index $\rk_{\IZ}(H_1(\Gamma_0)) \le 1$ holds;

   \item \label{lem:Hast(Q;A)_finite_irreducible_case:H2}
    The order of $H^2(Q;A)$ is finite;

    \item \label{lem:Hast(Q;A)_finite_irreducible_case:H1}
    If $\Gamma$ is irreducible, then the order of $H^1(Q;A)$ is finite.

    \end{enumerate}
 \end{lemma}
  \begin{proof}~\ref{lem:Hast(Q;A)_finite_irreducible_case:characterization}
  ``$\Longrightarrow$''
  Let $\Gamma_0 \subseteq \Gamma$ be a subgroup of finite index.
  We have to show $\rk_{\IZ}(H_1(\Gamma_0)) \le 1$, provided
  that $\Gamma$ is irreducible.
  We can find a normal subgroup $\Gamma_1 \subseteq \Gamma$ with
  $[\Gamma: \Gamma_1] < \infty$ and $\Gamma_1 \subseteq \Gamma_0$.
  Since the image of the map $H_1(\Gamma_1) \to H_1(\Gamma_0)$ 
  induced by the inclusion 
  has a finite cokernel, it suffices to show  $\rk_{\IZ}(H_1(\Gamma_1)) \le 1$.
  Since $[\Gamma_1,\Gamma_1]$
  is a characteristic subgroup of $\Gamma_1$ and $\Gamma_1$ is a normal
  subgroup of $\Gamma$, the subgroup $[\Gamma_1,\Gamma_1]$ of $\Gamma$ is normal in
  $\Gamma$. Let $f \colon \Gamma \to V := \Gamma /[\Gamma_1,\Gamma_1]$
  be the projection. Its restriction to $\Gamma_1$ factorizes over the projection
  $f_1 \colon \Gamma_1 \to H_1(\Gamma_1)$ to a homomorphism
  $i \colon H_1(\Gamma_1) \to V$. One easily checks that
  $i$ is an injection whose image has finite index in $V$.
  Hence $V$ is virtually finitely generated abelian.
  Since $\Gamma$ is by  assumption irreducible, the virtual cohomological dimension
  of $V$ and hence of $H_1(\Gamma_1)$ is at most one. This implies
   $\rk_{\IZ}(H_1(\Gamma_0)) \le 1$.\\[1mm]
  ``$\Longleftarrow$''
  Consider an epimorphism $f \colon \Gamma \to V$ to a virtually
  finitely generated abelian group $V$. Put $n = \vcd(V)$. Choose a
  subgroup $V_0\subseteq V$ with $V_0 \cong \IZ^n$ and $[V:V_0] < \infty$.
  Let $\Gamma_0 \subseteq \Gamma$ be the preimage of $V_0$ under $f$ and denote by
  $f_0 \colon \Gamma_0 \to V_0$ the restriction of $f$ to $\Gamma_0$.
  Then $f_0$ is an epimorphism and $\Gamma_0$ is a subgroup of $\Gamma$
  with $[\Gamma:\Gamma_0] < \infty$. The map $f_0$ factorizes
  over the projection $\Gamma_0 \to H_1(\Gamma_0)$ to an epimorphism
  $\overline{f_0} \colon H_1(\Gamma_0) \to V_0$. Hence
  $n \le \rk_{\IZ}(H_1(\Gamma_0))$. Since by assumption
  $\rk_{\IZ}(H_1(\Gamma_0)) \le 1$, we conclude $\vcd(V) \le 1$.
  Hence $\Gamma$ is irreducible.
  \\[1mm]~\ref{lem:Hast(Q;A)_finite_irreducible_case:H2}
  Since $Q$ is infinite and virtually cyclic, there exists an exact
  sequence $1 \to \IZ \to Q \to F \to 1$ for some finite group $F$.
  Recall that the cohomology group of finite groups is finite for any
  coefficient module in dimensions $\ge 1$ (see~\cite[Corollary~10.2
  in Chapter~III on page~84]{Brown(1982)}).  Obviously the cohomology
  of $\IZ$ vanishes for any coefficient module in dimensions $\ge
  2$. Now the claim follows from the Hochschild-Serre spectral
  sequence (see~\cite[Section~6 in Chapter~VII]{Brown(1982)}) applied
  to the exact sequence above.
  \\[1mm]~\ref{lem:Hast(Q;A)_finite_irreducible_case:H1}
  Because of the Hochschild-Serre spectral sequence
  (see~\cite[Section~6 in Chapter~VII]{Brown(1982)}) applied to the
  exact sequence above it suffices to prove that $H^1(\IZ;A)$ is
  finite. This is equivalent to the statement that $H_0(\IZ;A)$ is
  finite since $H^1(\IZ;A) \cong H_0(\IZ;A)$.
  Let $\Gamma_0$ be the preimage of $\IZ \subseteq Q$ under the
  projection $\pr \colon \Gamma \to Q$. It is a normal subgroup of
  finite index in $\Gamma$ and fits into an exact sequence $1 \to A
  \to \Gamma_0 \to \IZ\to 0$. {}From the Hochschild-Serre spectral sequence we
  obtain a short exact sequence
  $$0 \to H_0(\IZ;A) \to H_1(\Gamma_0) \to H_1(\IZ) \to 0.$$

 Hence it remains to show that the rank of the finitely generated
 abelian group $H_1(\Gamma_0)$ is at most one. This follows from
 assertion~\ref{lem:Hast(Q;A)_finite_irreducible_case:characterization}.
\end{proof}

%%%%%%%%%%%%%%%%%%%%%%%%%%%%%%%%%%%%%%%%%%%%%%%%%%%%%%%%%%%%%%%%%%%%%

\subsection{Finding the appropriate finite quotient groups}
\label{subsec:Finding_the_approriate_finite_quotient_groups}

Fix a special affine group $\Gamma$ of rank $(n+1)$.
For any positive integer $s$
the subgroup $sA \subseteq A$ is characteristic
and hence is normal in both $A$ and $\Gamma$. Put
\begin{eqnarray*}
  A_s & := & A/sA;
  \\
  \Gamma_s & := & \Gamma/sA.
\end{eqnarray*}
Then $A_s$ is isomorphic to $(\IZ/s)^n$ and we obtain an exact
sequence
\begin{eqnarray}
  & 1 \to A_s \to \Gamma_s \xrightarrow{\pr_s}  Q \to 1. &
  \label{A_s_to_Gamma_s_to_Q}
\end{eqnarray}
Let
$$p_s \colon \Gamma \to \Gamma_s$$
be the canonical projection.

\begin{definition}[Pseudo $s$-expansive homomorphism]
\label{def:pseudo_s-expansive_homomorphism}
  Let $s$ be an integer. We call a group homomorphism $\phi \colon
  \Gamma \to \Gamma$ \emph{pseudo $s$-expansive} if it fits into the
  commutative diagram
$$\xymatrix{1 \ar[r]
  & A \ar[r] \ar[d]^{s \cdot \id} & \Gamma \ar[r]^{\pr} \ar[d]^{\phi}
  & Q \ar[r] \ar[d]^{\id} & 1
  \\
  1 \ar[r] & A \ar[r] & \Gamma \ar[r]^{\pr} & Q \ar[r] & 1}
$$
where both the upper and the lower horizontal exact sequence is the
one of~\eqref{A_to_Gamma_to_Q}.
\end{definition}

Recall that $|H^2(Q;A)|$ is finite by
Lemma~\ref{lem:Hast(Q;A)_finite_irreducible_case}%
~\ref{lem:Hast(Q;A)_finite_irreducible_case:H2}.

\begin{lemma}\label{lem:existence_of_pseudo_s-expansive_maps}
  \begin{enumerate}

  \item \label{lem:existence_of_pseudo_s-expansive_maps:existence} For
    any integer $s$ with $s \equiv 1 \mod |H^2(Q;A)|$ there exists a
    pseudo $s$-expansive homomorphism $\phi \colon \Gamma \to \Gamma$;

  \item \label{lem:existence_of_pseudo_s-expansive_maps:splitting}

    For any integer $s$ with $s \equiv 1 \mod |H^2(Q;A)|$ the exact
    sequence $1 \to A_s \to \Gamma_s \xrightarrow{\pr_s} Q \to 1$
    of~\eqref{A_s_to_Gamma_s_to_Q} splits.

  \end{enumerate}
\end{lemma}
\begin{proof}~\ref{lem:existence_of_pseudo_s-expansive_maps:existence}
  Since $A$ is abelian, isomorphism classes of extensions with $A$ as
  normal subgroup and $Q$ as quotient are in one to one correspondence
  with elements in $H^2(Q;A)$ (see~\cite[Theorem~3.12 in Chapter~IV on
  page~93]{Brown(1982)}).  Let $\Theta$ be the class associated to the
  extension~\eqref{A_to_Gamma_to_Q}. Since by assumption $s \equiv 1
  \mod |H^2(Q;A)|$ , the homomorphism
  $$H^2(Q;s \cdot \id_A) = s \cdot \id_{H^2(Q;A)} \colon H^2(Q;A) \to H^2(Q;A)$$
  is the identity and sends $\Theta$ to $\Theta$, and the claim
  follows.
  \\[1mm]~\ref{lem:existence_of_pseudo_s-expansive_maps:splitting} Let
  $\phi \colon \Gamma \to \Gamma$ be a pseudo $s$-expansive map.  The
  composite $p_s \circ \phi \colon \Gamma \to \Gamma_s$ sends $A$ to
  the trivial group and hence factorizes through $\pr\colon \Gamma \to
  Q$ to a homomorphism $\overline{\phi} \colon Q \to \Gamma_s$
  whose composite with $\pr_s \colon \Gamma_s \to Q$ is the  identity.
\end{proof}

The group $Q$ is virtually cyclic. Hence we can choose a normal
infinite cyclic subgroup $C \subseteq Q$.  Fix an integer $s$
satisfying $s \equiv 1 \mod |H^2(Q;A)|$ and a positive integer $r$
such that the order of $\aut(A_s) = \GL_n(\IZ/s)$ divides $r$.  Put
$$Q_r = Q/rC.$$
Let
$$\rho_s \colon Q \to \aut(A_s)$$
be the group homomorphism given by the conjugation action associated
to the exact sequence~\eqref{A_s_to_Gamma_s_to_Q}.  It factorizes
through the projection $Q \to Q_r$ to a homomorphism
$$\rho_{r,s} \colon Q_r \to \aut(A_s).$$
By
Lemma~\ref{lem:existence_of_pseudo_s-expansive_maps}%
~\ref{lem:existence_of_pseudo_s-expansive_maps:splitting}
we can choose a splitting
$$\sigma \colon Q \to \Gamma_s$$
of the projection $\pr_s \colon  \Gamma_s \to Q$. It yields an explicit
isomorphism $\Gamma_s \xrightarrow{\cong} A_s \rtimes_{\rho_s} Q$.
Its composition with the group homomorphism $A_s \rtimes_{\rho_s} Q \to
A_s \rtimes_{\rho_{r,s}} Q_r$, which comes from the identity on $A_s$ and
the projection $Q \to Q_r$, is denoted by
$$q_{r,s} \colon \Gamma_s \to A_s \rtimes_{\rho_{r,s}} Q_r.$$
We obtain a commutative diagram
\begin{eqnarray*}
  &
  \xymatrix{1 \ar[r] &
    A_s \ar[r] \ar[d]_{\id} &
    \Gamma_s \ar[r]^{\pr_s} \ar[d]^{q_{r,s}}
    & Q\ar[d] \ar[r]
    & 1
    \\
    1 \ar[r] &
    A_s \ar[r] ^-{\overline{i_s}} &
    A_s \rtimes_{\rho_{r,s}} Q_r \ar[r]
    & Q_r \ar[r]
    & 1
  }
  &
  \label{diagram_for_q_s}
\end{eqnarray*}
where the upper exact sequence is the one
of~\eqref{A_s_to_Gamma_s_to_Q}, the lower exact sequence is the
obvious one associated to a split extension and the right vertical
arrow is the canonical projection $Q \to Q_r$.

Define an epimorphism of groups by the composite

\begin{eqnarray}
  & \alpha_{r,s} \colon \Gamma \xrightarrow{p_s} \Gamma_s
\xrightarrow{q_{r,s}} A_s \rtimes_{\rho_{r,s}} Q_r. &
\label{the_maps_alpha_rs}
\end{eqnarray}

It will play the role of the map $\alpha_{R,\epsilon}$ appearing
in Definition~\ref{def:Farrell-Hsiang}.

%%%%%%%%%%%%%%%%%%%%%%%%%%%%%%%%%%%%%%%%%%%%%%%%%%%%%%%%%%%%%%%%%%%%%

\subsection{Hyperelementary subgroups and index estimates}
\label{subsec:Hyperelementary_subgroups_and_index_estimates}

This subsection is devoted to the proof of the following proposition.
Recall that $\pr \colon \Gamma \to Q$ and $\pi \colon Q \to D$ are
the canonical projections and that we
consider $A$ as a $\IZ Q$-module by the conjugation action coming from the exact
sequence~\eqref{A_to_Gamma_to_Q}.

\begin{proposition} \label{prop:hyper-good} Let $\Gamma$ be an irreducible special
  affine group.  Consider any natural number $\tau$.
  Then we can find natural numbers $s$ and $r$
  with the following properties:

  \begin{enumerate}

  \item \label{prop:hyper-good:s_equiv_1}
   $s \equiv 1 \mod |H^2(Q;A)|$;

  \item \label{prop:hyper-good:aut_divides_r}
  The order of $\aut(A_s)$ divides $r$;

  \item \label{prop:hyper-good:alternative}For every hyperelementary subgroup $H
  \subseteq A_s \rtimes_{\rho_{r,s}} Q_r$   one of the
  following two statements is true if $\overline{H}$ is the preimage of
  $H$ under the epimorphism $\alpha_{r,s} \colon \Gamma \to A_s
  \rtimes_{\rho_{r,s}} Q_r$:
  \begin{enumerate}

  \item \label{prop:hyper-good:index_in_A}
   The homology groups $H^1(\pr(\overline{H});A)$  and
   $H^2(\pr(\overline{H});A)$ are finite 
   and there exists
   a natural number $k$ satisfying
  $$\begin{array}{l}
    k \; \text{divides} \; s;
    \\
    k \equiv 1 \mod |H^2(Q;A)|;
    \\
    k \equiv 1 \mod \bigl|H^1\bigl(\pr(\overline{H});A\bigr)\bigr|;
    \\
    k \equiv 1 \mod \bigl|H^2\bigl(\pr(\overline{H});A\bigr)\bigr|;
    \\
    k \ge \tau;
    \\
    \overline{H} \cap A \subseteq kA;
  \end{array}$$

\item \label{prop:hyper-good:index_in_D}
$\left[D : \pi\circ \pr(\overline{H})\right] \ge \tau$.

\end{enumerate}
\end{enumerate}
\end{proposition}

It will provide us with the necessary index estimates when we later
show that $\Gamma$ is a Farrell-Hsiang group in the sense of
Definition~\ref{def:Farrell-Hsiang}.

In order to prove Proposition~\ref{prop:hyper-good} we first reduce from
special affine groups of rank $(n+1)$ to
the special case of  semi-direct products 
$\IZ^n\rtimes_{\phi} \IZ$. 
Notice that every special affine group of rank $(n+1)$ contains
a subgroup of finite index which is isomorphic to 
$\IZ^n\rtimes_{\phi} \IZ$.

\begin{definition}
  \label{def:hyper-good}
  Consider $M \in \GL_n(\IZ)$. We will say that $M$ is
  \emph{hyper-good} if the following holds: Given natural numbers $o$
  and $\nu$, there are natural numbers $s$ and $r$ satisfying
  \begin{enumerate}
  \item \label{def:hyper-good:s_equiv_1} $s \equiv 1 \mod o$;
  \item \label{def:hyper-good:order_of_GL_divides_r} The order of
    $\GL_n(\IZ/s)$ divides $r$.  In particular we can consider the
    group $(\IZ/s)^n \rtimes_{M_s} \IZ/r$, where $M_s$ is the
    reduction of $M$ modulo $s$.  (We will consider $(\IZ/s)^n$ as a
    subgroup of this group and denote by $\pr_{r,s}$ the canonical
    projection from this group to $\IZ/r$.)
  \item \label{def:hyper-good:index} If $H$ is a hyperelementary
    subgroup of $(\IZ/s)^n \rtimes_{M_s} \IZ/r$, then at least one
     of the following two statements is true:
    \begin{enumerate}
    \item \label{def:hyper-good:index:cap_Zn} There exists a natural
      number $k$ satisfying
     $$\begin{array}{l}
       k \; \text{divides} \; s;
       \\
       k \equiv 1 \mod o;
       \\
       k \ge \nu;
       \\
       H \cap (\IZ/s)^n \subseteq k(\IZ/s)^n;
       \end{array}$$

   \item \label{def:hyper-good:index:Zr} $\bigl[\IZ/r : \pr_{r,s}(H)\bigr] \geq
     \nu$.
   \end{enumerate}
 \end{enumerate}
\end{definition}

 In the sequel we will use the following elementary facts about indices of subgroups.
Given  a group $G$ with two subgroups $G_0$ and $G_1$ of finite index, we get
$$[G_0 : (G_0 \cap G_1)] \le [G:G_1].$$
If $f \colon G \to G'$ is an epimorphism
with finite kernel $K$ and $G_0 \subseteq G$ is a subgroup, then
$$[G':f(G_0)] \le [G:G_0] \le [G':f(G_0)] \cdot |K|.$$

\begin{lemma} \label{lem:reduction_to_Zn_rtimes_Z} In order to prove
  Proposition~\ref{prop:hyper-good} it suffices to show that any
  matrix $M \in \GL_n(\IZ)$ is hyper-good.
\end{lemma}
\begin{proof}
  Recall that we have already chosen a normal infinite cyclic subgroup
  $C \subseteq Q$.  The index $[Q:C]$ is finite.  Let $\rho \colon Q
  \to \aut(A)$ be the conjugation action associated to the exact
  sequence $1 \to A \to \Gamma \xrightarrow{\pr} Q \to 1$ introduced
  in~\eqref{A_to_Gamma_to_Q}.  Fix a generator $t$ of $C$. Let
  \[\eta \colon A \to A\]
  be the automorphism given by $\rho(t)$.  Put
$$\widehat{\Gamma} := \pr^{-1}(C).$$
 Then $\widehat{\Gamma}$ is a normal subgroup in $\Gamma$ of finite
 index $[\Gamma : \widehat{\Gamma}] = [Q:C]$ and fits into an exact
 sequence
$$1 \to A \to \widehat{\Gamma} \xrightarrow{\widehat{\pr}} C \to 1,$$
  where $\widehat{\pr}$ is the restriction of $\pr$ to
 $\widehat{\Gamma}$.

Let $\tau$ be any natural number.
Let $\Gamma' \subseteq \Gamma$ be
 a subgroup of finite index. The exact sequence
$1 \to \Delta \to \Gamma \xrightarrow{z} D \to 1$ appearing in
Definition~\ref{def:special_affine_group} yields an exact sequence
$$1 \to \Gamma' \cap \Delta \to \Gamma' \to z(\Gamma') \to 1$$
for subgroups $\Gamma' \cap \Delta \subseteq \Delta$
and $z(\Gamma') \subseteq D$ of finite index.
Hence $\Gamma'$ is again an special affine group.  We conclude
from Lemma~\ref{lem:Hast(Q;A)_finite_irreducible_case}~%
~\ref{lem:Hast(Q;A)_finite_irreducible_case:characterization}
that $\Gamma'$ is irreducible. The exact
sequence~\eqref{A_to_Gamma_to_Q} yields the exact sequence
$$1 \to A' := A \cap \Gamma' \to \Gamma' \to Q' := \pr(\Gamma')  \to 1$$
for subgroups $A' \subseteq A$ and $Q' \subseteq Q$ of finite index.
Since $A' = A_{\Gamma' \cap \Delta}$ 
this is just the version of~\eqref{A_to_Gamma_to_Q} for $\Gamma'$.
Hence $H^1(Q';A')$ and $H^2(Q';A')$ are
finite by  Lemma~\ref{lem:Hast(Q;A)_finite_irreducible_case}%
~\ref{lem:Hast(Q;A)_finite_irreducible_case:H2}
and~\ref{lem:Hast(Q;A)_finite_irreducible_case:H1}.
The obvious sequence of abelian groups $1 \to A' \to A \to A/A' \to 0$
is an exact sequence of $\IZ Q'$-modules. It yields the long exact Bockstein sequence
\begin{multline*}
\cdots \to
H^1(Q';A') \to H^1(Q';A) \to H^1(Q';A/A') \to H^2(Q';A') \\
\to H^2(Q';A) \to H^2(Q';A/A') \to \cdots.
\end{multline*}
Since $A/A'$ is finite, $H^1(Q';A/A')$ and $H^2(Q';A/A')$ are finite.
Since $H^1(Q';A')$ and $H^2(Q';A')$ are finite, we conclude that
$H^1(Q';A)$ and $H^2(Q';A)$ are finite.

In particular we get for any subgroup $Q'\subseteq Q$ of finite index
that $H^1(Q';A)$ and $H^2(Q';A)$ are finite, just apply
the argument above to the special case,
where $\Gamma'$ is the preimage of $Q'$ under $\pr \colon \Gamma \to Q$.

Let $I$ be the set of subgroups $Q'$ of $Q$ of finite index
such that $\left[D : \pi(Q')\right] < \tau$.  We conclude for $Q' \in I$
$$[Q : Q']\le |F_{\Delta}| \cdot \left[D : \pi(Q')\right] \le |F_{\Delta}| \cdot \tau$$
from the exact sequence~\eqref{F_to_Q_to_D}. Since $Q$ contains only
finitely many subgroups of finite index bounded by $|F_{\Delta}| \cdot \tau$,
the set $I$ is finite.  Apply the assumption hyper-good to
the matrix $M$ describing the automorphism $\eta$ after identifying $A
= \IZ^n$ for the constants
\begin{eqnarray}
  \nu & = & \tau \cdot |F_{\Delta}|; \label{choice_of_nu}
  \\
  o & = & \prod_{Q' \in I} \left(\bigl|H^1(Q';A)\bigr| \cdot  \bigl|H^2(Q';A)\bigr|\right).
\label{choice_of_o}
\end{eqnarray}
Let $r$, $s$ and $k$ be the resulting natural numbers.

Recall that we have chosen a splitting $\sigma \colon Q \to \Gamma_s$
of the projection $\pr_s \colon \Gamma_s \to Q$. Let $\gamma \in
\Gamma$ be any element which is mapped under $p_s \colon \Gamma \to
\Gamma_s$ to $\sigma(t)$.  Conjugation with $\gamma$ induces on $A$
just the automorphism $\eta \colon A \to A$ since $\pr \colon \Gamma
\to Q$ maps $\gamma$ to $t$. The choice of $\gamma$ yields an explicit
identification
$$\widehat{\Gamma} = A \rtimes_{\eta} C.$$

Put
$$C_r = C/rC.$$
The epimorphism $\alpha_{r,s} := q_{r,s} \circ p_s \colon \Gamma \to
A_s \rtimes_{\rho_{r,s}} Q_r$ restricted to $\widehat{\Gamma}$ is the
composite of the inclusion $A_s \rtimes_{\rho_{r,s}|_{C_r}} C_r \to A_s
\rtimes_{\rho_{r,s}} Q_r$ with the obvious projection $\widehat{\alpha_{r,s}}
\colon A \rtimes_{\eta} C \to A_s \rtimes_{\eta_s} C_r$, where $\eta_s
\colon A_s \to A_s$ is the automorphism induced by $\eta \colon A \to
A$.

Consider any hyperelementary subgroup $H \subseteq A_s
\rtimes_{\rho_{r,s}} Q_r$.  Put
$$\widehat{H} := H \cap (A_s \rtimes_{\eta_s} C_r).$$
This is a hyperelementary subgroup of $A_s \rtimes_{\eta_s} C_r$ and we
get
\begin{eqnarray*}
   \alpha_{r,s}^{-1}(H) \cap \widehat{\Gamma}
  & = &
  \widehat{\alpha_{r,s}}^{-1}(\widehat{H});
  \\
  \alpha_{r,s}^{-1}(H)  \cap A
  & = &
  \widehat{\alpha_{r,s}}^{-1}(\widehat{H}) \cap A;
  \\
  \pr\bigl(\alpha_{r,s}^{-1}(H)\bigr) \cap C
  & = &
  \widehat{\pr}\bigl(\widehat{\alpha_{r,s}}^{-1}(\widehat{H})\bigr).
\end{eqnarray*}
Since the kernel of the epimorphism $\pi \colon Q \to D$ is the finite
group $F_{\Delta}$, we get
\begin{eqnarray*}
  \left[D: \pi\circ \pr\bigl(\alpha_{r,s}^{-1}(H)\bigr)\right]
  & \ge & \frac{\left[Q: \pr\bigl(\alpha_{r,s}^{-1}(H)\bigr)\right]}{[F_{\Delta}|};
  \\
  \left[Q: \pr\bigl(\alpha_{r,s}^{-1}(H)\bigr)\right]
  & \ge &
  \left[C:\widehat{\pr}\bigl(\widehat{\alpha_{r,s}}^{-1}(\widehat{H})\bigr)\right].
\end{eqnarray*}
This implies
\begin{eqnarray*}
  \alpha_{r,s}^{-1}(H)  \cap A \subseteq kA & \Longleftrightarrow &
  \widehat{\alpha_{r,s}}^{-1}(\widehat{H}) \cap A \subseteq kA;
  \\
  \left[D: \pi\circ \pr\bigl((\alpha_{r,s} ^{-1}(H)\bigr)\right] & \ge &
  \frac{\left[C:\widehat{\pr}\bigl(\widehat{\alpha_{r,s}}^{-1}(\widehat{H})\bigr)\right]}
  {|F_{\Delta}|}.
\end{eqnarray*}
Since the projection $C \to C_r$ maps
$\widehat{\pr}\bigl(\widehat{\alpha_{r,s}}^{-1}(\widehat{H})\bigr)$
to $\pr_{r,s}(\widehat{H})$, we get
\begin{eqnarray*}
  \widehat{\alpha_{r,s}}^{-1}(\widehat{H}) \cap A \subseteq kA
  & \Longleftrightarrow &
  \widehat{H}\cap A_s\subseteq kA_s;
  \\
  \left[C:\widehat{\pr}\bigl(\widehat{\alpha_{r,s}}^{-1}(\widehat{H})\bigr)\right]
  & \ge &
  \left[C_r: \pr_{r,s}(\widehat{H})\right],
\end{eqnarray*}
where $\pr_{r,s} \colon A_s \rtimes_{\eta} C_r \to C_r$ is the
canonical projection. We conclude
\begin{eqnarray}
  \alpha_{r,s}^{-1}(H)  \cap A \subseteq k  A & \Longleftrightarrow &
  \widehat{H}\cap A_s\subseteq kA_s;
  \label{reduction_to_hyper-good_subset_kA}
  \\
  \left[D: \pi\circ \pr\bigl((\alpha_{r,s} ^{-1}(H)\bigr)\right] & \ge &
  \frac{\left[C_r: \pr_{r,s}(\widehat{H})\right]}{|F_{\Delta}|}.
  \label{reduction_to_hyper-good_subset_large_index}
\end{eqnarray}
Now we can show that one of two conditions
appearing in assertion~\ref{prop:hyper-good:alternative} of
Proposition~\ref{prop:hyper-good} hold
with respect to the number $k$.

Suppose that ${\left[C_r: \pr_{r,s}(\widehat{H})\right]} \ge \nu$.
Then $\left[D: \pi\circ \pr\bigl((\alpha_{r,s} ^{-1}(H)\bigr)\right] \ge \tau$
by~\eqref{reduction_to_hyper-good_subset_large_index} and our choice of $\nu$
in~\eqref{choice_of_nu}.
Hence condition~\ref{prop:hyper-good:index_in_D}  appearing
in Proposition~\ref{prop:hyper-good} is true. Hence it remains to show
that condition~\ref{prop:hyper-good:index_in_A}
in Proposition~\ref{prop:hyper-good} holds provided
that 
$\left[D: \pi\circ \pr\bigl((\alpha_{r,s} ^{-1}(H)\bigr)\right] < \tau$
holds.
This implies
${\left[C_r: \pr_{r,s}(\widehat{H})\right]} <  \nu$.
Recall that the number $k$ satisfies
$$\begin{array}{l}
       k \; \text{divides} \; s;
       \\
       k \equiv 1 \mod o;
       \\
       k \ge \tau;
       \\
       \widehat{H} \cap A_s \subseteq kA_s.
     \end{array}$$
The group $\overline{H} := \alpha_{r,s}^{-1}(H)$ has the property that
$\pr(\overline{H})$ belongs to the set $I$ appearing
in the definition of $o$  in~\eqref{choice_of_o}. Now
condition~\ref{prop:hyper-good:index_in_A}  appearing
in Proposition~\ref{prop:hyper-good} follows from our choice of $o$
in~\eqref{choice_of_o} and from~\eqref{reduction_to_hyper-good_subset_kA}.
This finishes the proof of Proposition~\ref{lem:reduction_to_Zn_rtimes_Z}.
\end{proof}

Next we reduce the problem from hyperelementary groups to cyclic subgroups.

\begin{definition}
  \label{def:cyclic-good}
  Let $M \in \GL_n(\IZ)$.  We will say that $M$ is \emph{cyclic-good}
  if the following holds: given positive integers $o$ and $\nu$, there
  are prime numbers $p_1$ and $p_2$ such that the following holds.
  
  Set
  \begin{equation*}
    s  := p_1p_2 \quad \quad \text{and} \quad
    r  :=  s \cdot |\GL_n(\IZ/s)|.
  \end{equation*}
  In particular we can consider the group
  $(\IZ/s)^n \rtimes_{M_s} \IZ/r$, where $M_s$ is the reduction of
  $M$ modulo $s$.  (We will consider $(\IZ/s)^n$ as a subgroup of
  this group and denote by $\pr_{r,s}$ the canonical projection from
  this group to $\IZ/r$.)
  We require that
  \begin{enumerate}

  \item \label{def:cyclic-good:s_equiv_1}

   $$\begin{array}{l}
     p_1\not= p_2;
     \\
     p_i  \equiv 1 \mod o \quad \quad \text{for}\; i = 1,2;
     \\
     p_i \ge \nu \quad \text{for}\; i = 1,2;
   \end{array}$$

 \item \label{def:cyclic-good:index_plus_primes} If $C$ is a cyclic
   subgroup of $(\IZ/s)^n \rtimes_{M_s} \IZ/r$, 
   then \emph{at least one} of
   the following two statements is true:
   \begin{enumerate}
   \item \label{def:cyclic-good:index_plus_primes:cap-Z/sn} $C \cap
     (\IZ/s)^n = \{0\}$;

   \item \label{def:cyclic-good:index_plus_primes:one-prime} There is
     $i \in \{1,2\}$ such that $p_i$ divides both $|C|$ and
     $[\IZ/r :\pr_{r,s}(C)]$.
   \end{enumerate}
 \end{enumerate}
\end{definition}

\begin{lemma}
  \label{lem:cyclic-good-implies-hyper-good}
  Assume that $M \in \GL_n (\IZ)$ is cyclic-good.  
  Then $M$ is hyper-good.
\end{lemma}

\begin{proof}
  Suppose that $M \in \GL_n(\IZ)$ is cyclic-good. We want to show that
  it is hyper-good.  Let $\nu > 0$ be given.  Pick $p_1$ and $p_2$ and
  put $s= p_1p_2$ and $r= s \cdot |\GL_n(\IZ/s)|$ as in
  Definition~\ref{def:cyclic-good}.  Obviously
  conditions~\ref{def:hyper-good:s_equiv_1}
  and~\ref{def:hyper-good:order_of_GL_divides_r} appearing in
  Definition~\ref{def:hyper-good} are satisfied for $s$ and $r$.  It
  remains to show that
  condition~\ref{def:cyclic-good:index_plus_primes} appearing in
  Definition~\ref{def:cyclic-good} implies
  condition~\ref{def:hyper-good:index} appearing in
  Definition~\ref{def:hyper-good}.

  Let $H$ be a hyperelementary subgroup of $(\IZ/s)^n \rtimes_{M_s}
  \IZ/r$.  There is an exact sequence $1 \to C \to H \xrightarrow{f} L
  \to 1$ where $C$ is a cyclic group and $L$ is an $l$-group for a
  prime $l$ not dividing the order of $C$.  It follows that $[(\IZ /
  s)^n \cap H : (\IZ/s)^n \cap C]$ and $[\pr_{r,s}(H) : \pr_{r,s}(C)]$
  are both $l$-powers since $\bigl((\IZ / s)^n \cap H \bigr)/\bigl(
  (\IZ/s)^n \cap C\bigr)$ is a subgroup of $L$ and $\pr_{r,s}(H)
  /\pr_{r,s}(C)$ is a quotient of $L$.

  Suppose that
  condition~\ref{def:cyclic-good:index_plus_primes:cap-Z/sn} appearing
  in Definition~\ref{def:cyclic-good} is satisfied, i.e., $C \cap
  (\IZ/s)^n = \{0 \}$.  Then $H \cap (\IZ/s)^n$ is an $l$-group. If $H
  \cap (\IZ/s)^n$ is trivial,
  condition~\ref{def:hyper-good:index:cap_Zn} appearing in
  Definition~\ref{def:hyper-good} is obviously satisfied for $k = s$.
  Suppose that $H \cap (\IZ/s)^n$ is non-trivial.  Since $s = p_1p_2$,
  the prime $l$ must be $p_1$ or $p_2$. Let $k$ be $p_1$ if $l = p_2$, 
  and be $p_2$ if $l = p_1$. Then $H \cap (\IZ/s)^n \subseteq k
  \cdot (\IZ/s)^n$, i.e., condition~\ref{def:hyper-good:index:cap_Zn} 
  appearing in
  Definition~\ref{def:hyper-good} is satisfied.

  Suppose that
  condition~\ref{def:cyclic-good:index_plus_primes:one-prime}
  appearing in Definition~\ref{def:cyclic-good} is satisfied, i.e.,
  for some $i \in \{1,2\}$ the prime $p_i$ divides both $|C|$ and
  $[\IZ/r : \pr_{r,s}(C)]$. We have $p\ge \nu$.  Since $l$ does not divide
  $|C|$, $p_i$ must be different from $l$.  Since $[\pr_{r,s}(H) :
  \pr_{r,s}(C)]$ is a power of $l$, the prime $p_i$ divides $[\IZ/r :
  \pr_{r,s}(H)]$. This implies $\nu \le p_i \le [\IZ/r :
  \pr_{r,s}(H)]$.  Hence condition~\ref{def:hyper-good:index:Zr}
  appearing in Definition~\ref{def:hyper-good} is satisfied.
\end{proof}

Finally we show that every element in $\GL_n(\IZ)$ is cyclic-good.

\begin{lemma}
  \label{lem:rs-kills-v}
  Let $M \in \GL_n(\IZ)$.  Let $s$ be any natural number. Let $r$ be a
  multiple of the order of the reduction $M_s \in \GL_n(\IZ/s)$ of
  $M$.   Let $t \in \IZ/r$ be the generator.

  Then for any $v \in (\IZ/s)^n$ and $r',s' ,j\in \IZ$ and $j$
  we have
  \[
  (vt^j)^{s'r'} = t^{js'r'} \in (\IZ/s)^n \rtimes_{M_s} \IZ/r
  \]
  provided that $s'v = 0 \in (\IZ/s)^n$ and $M_s^{jr'} = \id \in
  \GL_n(\IZ/s)$.
\end{lemma}

\begin{proof}
  We have
  \[
  (vt^j)^{s'r'} = \left(\sum_{i = 0}^{s'r'-1} (M_s^j)^iv\right)
  t^{js'r'}
  \]
  and
  \begin{multline*}
    \sum_{i = 0}^{s'r'-1} (M_s^j)^i v = \sum_{k = 0}^{s'-1} \sum_{l =
      0}^{r'-1} (M_s^j)^{l + kr'} v = \sum_{k = 0}^{s'-1} \sum_{l =
      0}^{r'-1} (M_s^j)^{l} v
    \\
    = s' \sum_{l = 0}^{r'-1} (M_s^j)^{l} v = \sum_{l = 0}^{r'-1}
    (M_s^j)^{l} s'v = 0.
  \end{multline*}
\end{proof}

\begin{lemma}
  \label{lem:prime_power}
  Let $M \in \GL_n(\IZ)$.  Let $s$ be any natural number.  Let $r'$ be
  a multiple of the order of $M_s \in \GL_n(\IZ/s)$.  Let $r := r's$.
  Let $C$ be a cyclic subgroup of $(\IZ/s)^n \rtimes_{M_s} \IZ/r$ that has
  a nontrivial intersection with $(\IZ/s)^n$.

  Then there is a prime
  power $p^N$ ($N \geq 1$) such that
  \begin{enumerate}
  \item $p^N$ divides $r = r's$;
  \item $p^N$ does not divide $|\pr_{r,s}(C)|$;
  \item $p$ divides $|C \cap (\IZ/s)^n|$.
  \end{enumerate}
\end{lemma}

\begin{proof}  Let $t \in \IZ/r$ be the generator.
  Let $vt^j$ be a generator of $C$.  Clearly $v \neq 0$  and $w :=
  (vt^j)^{|\pr_{r,s}(C)|}$ is a nontrivial element of $C \cap (\IZ/s)^n$
  (otherwise $C \cap (\IZ/s)^n$ would be trivial).  Let $s'$ be the
  order of $w \in (\IZ/s)^n$.  Lemma~\ref{lem:rs-kills-v} implies that
  $(vt^j)^{sr'}$ is a power of $t$.  If $K$ is any integer with
  $(K,s') = 1$, then $Kw = (vt^j)^{|\pr_{r,s}(C)| \cdot K} \neq 0 \in C
  \cap (\IZ/s)^n$ and hence $Kw$ is not a power of $t$. 
  Using  Lemma~\ref{lem:rs-kills-v} again, this implies
  that $s'r'$ does not divide $|\pr_{r,s}(C)| \cdot K$ for any
  integer $K$ with $(K,s') = 1$.  Therefore there is a prime $p$
  dividing $s'$ and a number $N \geq 1$ such that $p^N$ divides
  $s'r'$, but not $|\pr_{r,s}(C)|$.  Clearly $s'$ divides $|C \cap
  (\IZ/s)^n|$.  Thus $p$ divides $|C \cap (\IZ/s)^n|$.
  Because $s'$ divides $s$, $p^N$ divides $r = r's$.
\end{proof}

\begin{lemma}
  \label{lem:all-M-are-cyclic-good}
  All $M \in \GL_n(\IZ)$ are cyclic-good.
\end{lemma}
\begin{proof}
  Let $o$ and $\nu$ be any positive integers.  By Dirichlet's Theorem
  (see~\cite[Lemma~3 in III.2.2 on page~25]{Serre(1993)})
  there exists infinitely many primes which
  are congruent $1$ modulo $o$. Hence we can find primes $p_1$ and
  $p_2$ satisfying condition~\ref{def:cyclic-good:s_equiv_1} appearing
  in Definition~\ref{def:cyclic-good}.
  It remains to
  show that condition~\ref{def:cyclic-good:index_plus_primes}
  appearing in Definition~\ref{def:cyclic-good} holds.

  Let $C$ be a cyclic subgroup of $(\IZ/s)^n \rtimes_{M_s} \IZ/r$.  We
  have to show
  condition~\ref{def:cyclic-good:index_plus_primes:one-prime}
  appearing in Definition~\ref{def:cyclic-good} holds, provided that
  $C \cap (\IZ/s)^n \neq 0$.  We can apply Lemma~\ref{lem:prime_power}
  with $r' = |\GL_n(\IZ/s)|$.  Thus there is a prime $p$ and a number
  $N$ such that
  \begin{enumerate}
  \item \label{prop:all-M:div-rs} $p^N$ divides $r = |\GL_n(\IZ/s)|
    \cdot s$;
  \item \label{prop:all-M:not-div} $p^N$ does not divide
    $|\pr_{r,s}(C)|$;
  \item \label{prop:all-M:div-C-cap} $p$ divides $|C \cap (\IZ/s)^n|$.
  \end{enumerate}
  We deduce from~\ref{prop:all-M:div-rs} and~\ref{prop:all-M:not-div}
  that $p$ divides $[\IZ/r:\pr_{r,s}(C)]$.  We deduce
  from~\ref{prop:all-M:div-C-cap} that $p$ divides $|C|$ and $s$.
  Because $s = p_1 \cdot p_2$ it follows that $p$ is either $p_1$ or
  $p_2$ Therefore
  condition~\ref{def:cyclic-good:index_plus_primes:one-prime}
  appearing in Definition~\ref{def:cyclic-good} holds.
\end{proof}

Now Proposition~\ref{prop:hyper-good} follows from
Lemma~\ref{lem:reduction_to_Zn_rtimes_Z},
Lemma~\ref{lem:cyclic-good-implies-hyper-good} and
Lemma~\ref{lem:all-M-are-cyclic-good}.

%%%%%%%%%%%%%%%%%%%%%%%%%%%%%%%%%%%%%%%%%%%%%%%%%%%%%%%%%%%%%%%%%%%%%

\subsection{Contracting maps for irreducible special affine groups}
\label{subsec:Contracting_maps_for_irreducible_special_affine_groups}

\begin{proposition}
  \label{prop:some-contractions}
  Let $\Gamma$ be an irreducible  special affine group. Fix a finite set of
  generators of $\Gamma$ and let $d_{\Gamma}$ be the associated word
  metric on $\Gamma$.  Then there is a natural number $N$ 
  and such that for any
  given real numbers $R >0$ and $\epsilon > 0$  there exists 
  a sequence of real numbers $(\xi_m)_{m \ge 1}$ 
  and a natural number $\mu$  such that  the following is true:

  \begin{enumerate}

  \item \label{prop:some-contractions:hard} Let $\overline{H} \subseteq \Gamma$
   be any subgroup of finite index such that
   $\big|H^1\bigl(\pr(\overline{H});A\bigr)\bigr|$
   and $\big|H^2\bigl(\pr(\overline{H});A\bigr)\bigr|$ are finite.
   Suppose that there exists an integer $k$ satisfying
        $$\begin{array}{l}
          k \ge \xi_{[D : \pi \circ \pr(\overline{H})]};
          \\
          k \equiv 1 \mod |H^2(Q;A)|;
          \\
          k \equiv 1 \mod \bigl|H^1\bigl(\pr(\overline{H});A\bigr)\bigr|;
          \\
          k \equiv 1 \mod \bigl|H^2\bigl(\pr(\overline{H});A\bigr)\bigr|;
          \\
          A\cap \overline{H} \subseteq kA.
        \end{array}$$
 
        Then there is a simplicial complex $E$ of dimension $\le N$
        with a simplicial cell preserving $\overline{H}$-action whose isotropy
        groups are virtually cyclic, and an $\overline{H}$-equivariant map
        $f \colon \Gamma \to E$ satisfying
        \[
        d_{\Gamma}\bigl(\gamma_1,\gamma_2) \leq R \quad \implies \quad
        d^{l^1}\bigl(f(\gamma_1),f(\gamma_2)\bigr) \le \epsilon
        \]
        for $\gamma_1,\gamma_2 \in \Gamma$;

\item \label{prop:some-contractions:easy}
        If $\overline{H} \subseteq \Gamma$ is  any subgroup such that 
        $[D : \pi \circ \pr(\overline{H})] \ge \mu$,
        then there exists a $1$-dimensional simplicial complex $E$
        with a simplicial cell preserving $\overline{H}$-action 
        such that for
        every $e \in E$ the isotropy group $\overline{H}_e$ satisfies $A
        \subseteq \overline{H}_e$ and $[\overline{H}_e : A] < \infty$
        and is in particular a virtually finitely generated 
        abelian subgroup
        of $\Gamma$, and an $\overline{H}$-equivariant map
        $f\colon \Gamma \to E$ satisfying
        \[
        d_{\Gamma}\bigl(\gamma_1,\gamma_2) \leq R \quad \implies \quad
        d^{l^1}\bigl(f(\gamma_1),f(\gamma_1)\bigr) \le \epsilon
        \]
        for $\gamma_1,\gamma_2 \in \Gamma$.

      \end{enumerate}
    \end{proposition}

    The proposition above will provide us with the necessary
    contracting maps when we later show that $\Gamma$ is a
    Farrell-Hsiang group in the sense of
    Definition~\ref{def:Farrell-Hsiang}.  Its proof needs some
    preparation.

    \begin{lemma} \label{lem:overline(H)_subseteq_im(phi)}
    Let $k$ be a natural number and
    $\overline{H}\subseteq \Gamma$ be a subgroup with
    $A\cap \overline{H} \subseteq kA$.  Assume that 
    $k \equiv 1 \mod \bigl|H^i\bigl(\pr(\overline{H});A\bigr)\bigr|$
    for $i = 1,2$. Let $\phi \colon \Gamma \to \Gamma$ be a
    pseudo $k$-expansive map.

    Then $\overline{H}$ is subconjugated to $\im(\phi)$.
  \end{lemma}
  \begin{proof} Recall that $p_k \colon \Gamma \to \Gamma_k:= \Gamma/kA$
    is the canonical projection.  As explained in the proof of
    Lemma~\ref{lem:existence_of_pseudo_s-expansive_maps}%
~\ref{lem:existence_of_pseudo_s-expansive_maps:splitting},
    the composite $p_k \circ \phi \colon \Gamma \to \Gamma_k$ factorizes through the
    projection $\pr \colon \Gamma \to Q$ to a homomorphism
    $\overline{\phi} \colon Q \to \Gamma_k$
    whose composite with the projection $\pr_k \colon \Gamma_k \to Q$ is the
    identity.  Let $H'$ be the image of $\overline{H}$ under the projection
    $p_k \colon \Gamma \to \Gamma_k$.  The exact sequence
    $1 \to A_k:=A/kA \to \Gamma_k \xrightarrow{\pr_k} Q \to 1$
    yields by restriction the exact sequence
    \[1 \to A_k \to \pr_k^{-1}\bigl(\pr_k(H')\bigr) \xrightarrow{\pr_k'} \pr_k(H') \to
    1.\] The section $\overline{\phi}$ of $\pr_k$ restricts to a section
    $\overline{\phi}'\colon \pr_k(H') \to \pr_k^{-1}\bigl(\pr_k(H')\bigr)$
    of $\pr_k'$. The restriction of
    $\pr_k'$ to $H'$ yields an isomorphism $H' \to \pr(H')$ since $H' \cap A_k = \{1\}$.
    Hence its inverse defines a second section of $\pr_k'$.  Since
    $\pr(\overline{H}) = \pr_k(H')$, we get by assumption for $i = 1,2$
    $$k \equiv 1 \mod H^i\bigl(\pr_k(H');A).$$
    Hence multiplication with $k$ induces isomorphisms on $H^i\bigl(\pr_k(H');A)$
    for $i = 1,2$. The Bockstein sequence associated to the exact sequence
    $0 \to A \xrightarrow{k\cdot \id} A \to A_k \to 0$ of
    $\IZ [\pr(H')]$-modules implies $H^1(\pr_k(H');A_k) = 0$.
    Hence any two sections of $\pr_k'$ are conjugated
    (see~\cite[Proposition~2.3 in Chapter~IV on page~89]{Brown(1982)}).
    This implies that $H'$ and $\im(\overline{\phi}')$
    are conjugated in $\pr_k^{-1}\bigl(\pr_k(H')\bigr)$.
    Hence $H'$ and $\im(\overline{\phi}')$ are conjugated in $\Gamma_k$.

    In order to show that $\overline{H}$ is subconjugated to
    $\im(\phi)$ it suffices to show that $p_k^{-1}(H')$ is
    subconjugated to $\im(\phi)$ since obviously $\overline{H}
    \subseteq p_k^{-1}(H')$.  Choose an element $\gamma \in \Gamma$
    such that $p_k(\gamma) H'p_k(\gamma)^{-1} =
    \im(\overline{\phi}')$.  Since
    $\gamma p_k^{-1} (H') \gamma^{-1} =
    p_k^{-1}\left(p_k(\gamma) H' p_k(\gamma)^{-1}\right)$, we can
    assume without loss of generality that $H' \subseteq
    \im(\overline{\phi}')$, otherwise replace $H'$ by $p_k(\gamma)
    (H')p(\gamma)^{-1}$. This implies $H' \subseteq
    \im(\overline{\phi})$.  It remains to show
  $$p_k^{-1}(H')  \subseteq \im(\phi).$$
  Consider $\gamma_0 \in p_k^{-1}(H')$. Because of $H' \subseteq
  \im(\overline{\phi})$ we can find $\gamma_1 \in \Gamma$ such that
  $p_k(\gamma_0) = p_k \circ \phi(\gamma_1)$. Hence there is $a \in
  kA$ with $\gamma_0 = \phi(\gamma_1) \cdot a$ since $\ker(p_k) =
  kA$. Since $\phi$ induces $k \cdot \id$ on $A$, the element $a$ lies
  in the image of $\phi$ and hence $\gamma_0$ lies in the image of
  $\phi$.
\end{proof}

    \begin{lemma} \label{lem:equivariant_affine_diffeo} Let
      $\Gamma$ be an irreducible special affine group.
      Let $\phi \colon \Gamma \to \Gamma$ be a pseudo $s$-expansive
      group homomorphism.

      Then there exists $u \in \IR^n$ such that
      the affine diffeomorphism
      $$f \colon \IR^n \xrightarrow{\cong} \IR^n, \quad x \mapsto s\cdot x + u$$
      is $\phi$-equivariant.
\end{lemma}
\begin{proof}
  Given an element $\gamma \in \Gamma$, let $M_{\gamma} \colon \IR^n \to \IR^n$
  be the linear automorphism and $v_{\gamma} \in \IR^n$ uniquely determined by
  the property that $\gamma$ is the affine map $\IR^n \to \IR^n$ sending $x$ to
  $M_{\gamma}(x) + v_{\gamma}$.  One easily checks that
  $M_{\gamma_1\gamma_2} = M_{\gamma_1} \circ M_{\gamma_2}$
  and $v_{\gamma_1\gamma_2} = v_{\gamma_1} + M_{\gamma_1}(v_{\gamma_2})$
   hold for all $\gamma_1, \gamma_2$ in $\Gamma$ and
  $M_a = \id$ and $v_a = a$ hold for $a \in A$.  Consider $\gamma \in \Gamma$.
  Then there exists $a \in A$ such that $a \cdot \phi(\gamma) = \gamma$
  holds in $\Gamma$.  This implies that for all $x \in \IR^n$ we have
$$ M_{\phi(\gamma)}(x) + a + v_{\phi(\gamma)}
= M_{a \cdot \phi(\gamma) }(x) + v_{a \cdot \phi(\gamma)} = M_{\gamma}(x) +
v_{\gamma}.$$ We conclude
$$M_{\gamma} = M_{\phi(\gamma)}.$$
Consider the function
$$d \colon \Gamma \to \IR^n, \quad \gamma \mapsto v_{\phi(\gamma)} -s \cdot v_{\gamma}.$$
It factorizes through the projection $\pr \colon \Gamma \to Q$ to a
function $\overline{d} \colon Q \to \IR^n$,
since for any $a \in A$ and $\gamma \in \Gamma$ we have
\begin{multline*}
  v_{\phi(a \cdot \gamma)} -s \cdot v_{a \cdot \gamma} = v_{\phi(a)
    \cdot \phi(\gamma)} -s \cdot v_{a \cdot \gamma} = \phi(a) +
  v_{\phi(\gamma)} - s \cdot (a + v_{\gamma})
  \\
  = s \cdot a + v_{\phi(\gamma)} - s \cdot (a + v_{\gamma}) =
  v_{\phi(\gamma)} - s \cdot v_{\gamma}.
\end{multline*}
The conjugation action of $\gamma \in \Gamma$ on $A$ is given by
$M_{\gamma}$ by the following calculations
\[M_{\gamma a\gamma^{-1}} = M_{\gamma} \circ M_a \circ M_{\gamma^{-1}}
=  M_{\gamma} \circ \id \circ M_{\gamma^{-1}} = M_{\gamma \gamma^{-1}}
= M_{1} = \id \]
and
\begin{multline*}
v_{\gamma a \gamma^{-1}}
= v_{\gamma} + M_{\gamma}(v_{a\gamma^{-1}})
= v_{\gamma} +  M_{\gamma}(v_{a} + v_{\gamma^{-1}})
=  v_{\gamma} +  M_{\gamma}(v_{a}) + M_{\gamma}(v_{\gamma^{-1}})
\\
= M_{\gamma}(v_a) +   v_{\gamma} +  M_{\gamma}(v_{\gamma^{-1}})
= M_{\gamma}(a) + v_{\gamma\gamma^{-1}} = M_{\gamma}(a).
\end{multline*}
This action extends to an action on $\IR^n = A \ox \IR$ and is used 
in the following calculation, that shows that
$\overline{d}$ is a derivation. 
\begin{eqnarray*}
  \overline{d}\bigl(\pr(\gamma_1)\pr(\gamma_2)\bigr)
  & = &
  d(\gamma_1\gamma_2)
  \\
  & = &
  v_{\phi(\gamma_1\gamma_2)} -s \cdot v_{\gamma_1\gamma_2}
  \\
  & = &
  v_{\phi(\gamma_1)\phi(\gamma_2)} -s \cdot v_{\gamma_1\gamma_2}
  \\
  & = &
  v_{\phi(\gamma_1)} + M_{\phi(\gamma_1)}\bigl(v_{\phi(\gamma_2)}\bigr) -
s \cdot v_{\gamma_1} - s\cdot M_{\gamma_1}\bigl(v_{\gamma_2}\bigr)
  \\
  & = &
  v_{\phi(\gamma_1)} - s \cdot v_{\gamma_1} + M_{\gamma_1}\bigl(v_{\phi(\gamma_2)}\bigr)
 - s\cdot M_{\gamma_1}\bigl(v_{\gamma_2}\bigr)
  \\
  & = &
  v_{\phi(\gamma_1)} - s \cdot v_{\gamma_1} + M_{\gamma_1}\bigl(v_{\phi(\gamma_2)}
 - s\cdot v_{\gamma_2}\bigr)
  \\
  & = &
  d(\gamma_1)  + \gamma_1 \cdot d(\gamma_2)
  \\
  & = &
  \overline{d}\bigl(\pr(\gamma_1) \bigr) + \pr(\gamma_1)
\cdot \overline{d}\bigl(\pr(\gamma_2)\bigr).
\end{eqnarray*}
Since $H^1(Q;\IR^n) = H^1\bigl(Q;A \otimes_{\IZ} \IR\bigr) = H^1(Q;A)
\otimes_{\IZ} \IR$ and $H^1(Q;A)$ is finite by
Lemma~\ref{lem:Hast(Q;A)_finite_irreducible_case}%
~\ref{lem:Hast(Q;A)_finite_irreducible_case:H1} ,
we conclude $H^1(Q;\IR^n) = 0$. The description of
the cocycles as derivations and coboundaries as principal derivations
(see~\cite[Exercise~2 in III.1 on page~60]{Brown(1982)}) implies that
there exists $u \in \IR^n$ such that for all $\gamma \in \Gamma$
$$u - M_{\gamma}(u)  = v_{\phi(\gamma)} -s \cdot v_{\gamma}$$
holds. Hence the affine map $f \colon \IR^n \to \IR^n$ sending $x$ to
$sx + u$ is $\phi$-linear by the following calculation
\begin{eqnarray*}
  \phi(\gamma) \cdot f(x)
  & = &
  M_{\phi(\gamma)}(f(x)) + v_{\phi(\gamma)}
  \\
  & = &
  M_{\gamma}(s\cdot x + u ) + v_{\phi(\gamma)}
  \\
  & = &
  M_{\gamma}(s\cdot x) + M_{\gamma}(u)  + v_{\phi(\gamma)}
  \\
  & = &
  s\cdot M_{\gamma}(x) + s \cdot v_{\gamma} + u
  \\
  & = &
  f\big(M_{\gamma}(x) + v_{\gamma}\bigr)
  \\
  & = & f(\gamma \cdot x).
\end{eqnarray*}
\end{proof}

\begin{lemma} \label{lem:l1-metric_and_subdivision} Let $N$ be a
  natural number and $\epsilon > 0$. Then there exists a number $D_N$
  depending only on $N$ such that the following holds:

  Let $X$ be a simplicial complex of dimension $\le N$ and let $X'$ be
  its barycentric subdivision.  Then we get for every $x,y \in X$
$$d^{l^1}_X(x,y) \le D_N \cdot d^{l^1}_{X'}(x,y),$$
where $d^{l^1}_X$ and $d^{l^1}_{X'}$ denote the $l^1$-metric on $X$
and $X'$
\end{lemma}
\begin{proof}
  If $X$ is the standard $(2N+1)$-simplex $\Delta_{2N+1}$ , a direct
  inspection shows the existence of a number $D_N$ such that for every
  $x,y \in \Delta_{2N+1}$ we have
$$d^{l^1}_{\Delta_{2N+1}}(x,y) < D_N \cdot d^{l^1}_{(\Delta_{2N+1})'}(x,y).$$
Now consider $x,y \in X$. There is a subcomplex $Y \subseteq X$ with
at most $2(\dim(X)+ 2)$ vertices containing these four points. We can
identify $Y$ with a simplicial subcomplex of $\Delta_{2N +1}$.  Now
the claim follows for the number $D_N$ above since the $l^1$-metric is
preserved under inclusions of simplicial subcomplexes and barycentric
subdivision is compatible with inclusions of simplicial subcomplexes.
\end{proof}

Since $\Gamma$ acts properly and cocompactly on $\IR^n \times \IR$,
we can choose a $\Gamma$-invariant Riemannian metric $b^{\Gamma}$.
Let $d^{\Gamma}$ be the associated metric on $\IR^n \times \IR$.
Notice that $d^{\Gamma}$ is $\Gamma$-invariant, whereas the standard
Euclidean metric on $\IR^n\times \IR$ is not necessarily $\Gamma$-invariant.
We will denote by $B^{\Gamma}_r(x,s)$ the
closed ball of radius $r$ around the point $(x,s) \in \IR^n \times \IR$ with
respect to the metric $d^{\Gamma}$. By $B^{\euc}_r(x)$ we denote the closed ball
of radius $r$ around $x \in \IR^n$ with respect to the Euclidean metric.

In the sequel we fix a word metric $d_{\Gamma}$
on $\Gamma$.  The \v{S}varc-Milnor Lemma (see~\cite[Proposition~8.19
in Chapter~I.8 on page~140]{Bridson-Haefliger(1999)}) implies

\begin{lemma} \label{lem:Gamma_toRn_timesR} Let $\ev \colon \Gamma \to
  \IR^n \times \IR$ be the map given by evaluating the $\Gamma$-action
  on the origin. There exists positive real numbers $C_1$ and $C_2$
  such that for $\gamma_2,\gamma_2 \in \Gamma$
$$d^{\Gamma}\bigl(\ev(\gamma_1),\ev(\gamma_2)\bigr)
\le C_1 \cdot d_{\Gamma}\bigl(\gamma_1,\gamma_2\bigr) + C_2.$$
\end{lemma}

\begin{lemma} \label{lem:metric_estimate_for_Gamma_to_D}
  If $D$ is $\IZ$, denote by $t$ a
  generator of $\IZ$ and equip $D$ with the associated word metric
  $d_D$. If $D$ is $D_{\infty}$, consider the standard presentation
  $\langle s,t \mid sts=t^{-1}, s^2 = 1\rangle$ and equip $D$ with the
  associated word metric $d_D$.

  Then there exists a constant $C_3> 0$
  such that for all $\gamma_1, \gamma_2 \in \Gamma$ we get
$$d_{D}\bigl(\pi \circ \pr(\gamma_1),\pi \circ \pr(\gamma_1)\bigr)
\le C_3 \cdot d_{\Gamma}\bigl(\gamma_1,\gamma_2\bigr).$$
\end{lemma}
\begin{proof}
  The word metrics for two different set of generators are Lipschitz
  equivalent.  Hence it suffices to prove
  the claim for a particular choice of finite set of generators on
  $\Gamma$.  Fix a set of generators of $\Gamma$ such that each
  generator is sent under the epimorphism $\pi \circ \pr \colon \Gamma
  \to D$ to the unit element in $D$, to $t$ or to $s$. Equip $\Gamma$
  with the associated word metric.  Then we get for
  $\gamma_1,\gamma_2\in \Gamma$
$$d_{D}\bigl(\pi \circ \pr(\gamma_1),\pi \circ \pr(\gamma_2)\bigr) \le
d_{\Gamma}(\gamma_1,\gamma_2).$$
\end{proof}

Let $\calw$ be an open cover of $\IR^n \times \IR$ which is
$\Gamma$-invariant, i.e., for $W \in \calw$ and $\gamma \in \calw$ we
have $\gamma \cdot W = \{\gamma\cdot w\mid w \in W\} \in \calw$.
Recall that points in the realization of the nerve $| \calw |$ of the
open cover $\calw$ are formal sums $z = \sum_{W \in \calw} z_W \cdot
W$, with $z_W \in [0,1]$ such that $\sum_{W \in \calw} z_W = 1$ and
the intersection of all the $W$ with $z_W \neq 0$ is non-empty, i.e.,
$\{ W \; | \; z_W \neq 0 \}$ is a simplex in the nerve of
$\calw$. There is a map
\begin{eqnarray}
  && \beta^{\calw}  \colon \IR^n \times \IR \to | \calw |,
  \quad x \mapsto \sum_{W \in  \calw} (\beta^{\calw})_W (x) \cdot  W,
  \label{fcalu}
\end{eqnarray}
where
\[
(\beta^{\calw})_W (x) = \frac{a_W (x)}{\sum_{W \in \calw} a_W (x)}
\]
if we define
$$a_W(x) := d^{\Gamma} \bigl(x , (\IR^n\times \IR)\setminus W\bigr )
= \inf \{ d^{\Gamma}(x,w) \; | \; w \notin W \}.$$ Since $\calw$ is
$\Gamma$-invariant, the $\Gamma$-action on $\calw$ induces a simplicial
$\Gamma$-action on $|\calw|$.  Since $d^{\Gamma}$ is $\Gamma$-invariant, the map
$\beta^{\calw}$ is $\Gamma$-equivariant.  Let $d^{l^1}_{|\calw|}$ be the
$l^1$-metric on $| \calw |$.

\begin{lemma}
  \label{lem:contracting_estimate_d1_d_Gamma}
  Consider a natural number $N$ and a real number $\omega > 0$.  Suppose that for
  every $(x,s) \in \IR^n \times \IR$ there exists $W \in \calw$ such
  that $B^{\Gamma}_{\omega}(x,s)$ lies in $W$.  Suppose that the dimension of $\calw$ is
  less or equal to $N$.

  Then we get for $(x,s), (y,t) \in \IR^n \times \IR$ with
  $d^{\Gamma}\bigl((x,s),(y,t)\bigr) \le \frac{\omega}{8N}$
  \begin{eqnarray*}
    d^{l^1}_{|\calw|}(\beta^{\calw}(x,s),\beta^{\calw}(y,t)\bigr)
    & \le &
    \frac{64 \cdot N^2}{\omega} \cdot d^{\Gamma}\bigl((x,s),(y,t)\bigr).
  \end{eqnarray*}
\end{lemma}
\begin{proof}
This follows from~\cite[Proposition~5.3]{Bartels-Lueck-Reich(2008hyper)}.
\end{proof}

\begin{lemma} \label{lem:ball_inclusions}
Consider a real number $\omega > 0$ and a compact subset $I \subseteq \IR$.
Then there are positive real numbers $\sigma$ and $\alpha$
such that for all $x \in \IR^n$ and $s \in I$
\begin{eqnarray*}
  B^{\Gamma}_{\omega}(x,s)
  & \subseteq &
  B^{\euc}_{\sigma}(x) \times [s-\alpha/2,s+\alpha/2].
\end{eqnarray*}
\end{lemma}
\begin{proof}
  Choose a compact subset $K \subseteq \IR^n$ such that $\Delta \cdot K = \IR^n$.
  Since $B^{\Gamma}_{\omega}\bigl(K \times I)$ is a compact subset of
  $\IR^n \times \IR$, we can find $\sigma_0 > 0$ and $\alpha_0 > 0$ such that
  $B^{\Gamma}_{\omega}(K \times I) \subseteq B_{\sigma_0}^{\euc}(0) \times
  [-\alpha_0/2,\alpha_0/2]$ holds.  Choose $\sigma_1 > 0$ and $\alpha_1 > 0$
  such that $K \subseteq B^{\euc}_{\sigma_1}(0)$ and $I \subseteq [-\alpha_1/2,\alpha_1/2]$.
  Put $\sigma := \sigma_0 + \sigma_1$ and $\alpha = \alpha_0 + \alpha_1$. Then
  we get for all $(x,s) \in K \times I$ by the triangle inequality
  $$ B_{\sigma_0}^{\euc}(0) \times  [-\alpha_0/2,\alpha_0/2]
  \subseteq B^{\euc}_{\sigma}(x) \times [s-\alpha/2,s+\alpha/2].$$
  Hence we get for all $(x,s) \in K \times I$
  \begin{eqnarray*}
  B^{\Gamma}_{\omega}(x,s)
  & \subseteq &
  B^{\euc}_{\sigma}(x) \times [s-\alpha/2,s+\alpha/2]\bigr).
\end{eqnarray*}
  Since $\Delta \cdot K = \IR^n$ and $\Delta$ acts isometrically with respect
  to both $d^{\euc}$ and $d^{\Gamma}$, Lemma~\ref{lem:ball_inclusions} follows.
\end{proof}

Consider the following setup. Let $\pr \colon \Gamma \to Q$ and $\pi \colon Q \to D$ 
be the canonical projections appearing in~\eqref{A_to_Gamma_to_Q}
and~\eqref{F_to_Q_to_D}.  Choose an element $\sigma \in \Gamma$ such that the
action $\rho''\colon D \times \IR \to \IR $ of $\pi(\sigma)$ is given by the map
$\IR \to \IR, \;t \mapsto t +1$.  Put $\Gamma_0 := \pr^{-1}(C)$, where 
$C := \langle \pr(\sigma) \rangle$ is the infinite cyclic subgroup of $Q$ generated by
$\pr(\sigma)$.  We obtain an exact sequence
\[
1 \to A \to \Gamma_0 \xrightarrow{\pr|_{\Gamma_0}} C \to 1.
\]
Obviously $\sigma$ lies in $\Gamma_0$ and is mapped under $\pr|_{\Gamma_0} \colon \Gamma_0 \to C$ to a
generator of $C$.  The subgroup $\Gamma_0$ of $\Gamma$ has finite index. 

Consider on $\IR^n \times \IR$ the flow $\Phi_{\tau}(x,t) = (x, t + \tau)$.  Fix
an integer $l \ge 1$. Let $\call_{\le l}$ be the set of non-constant orbits under
$\Phi$ whose $\Gamma_0$-period is bounded by $l$, i.e., there exists $(x,s) \in
\IR^n \times \IR$ which lies in the orbit, an element $\gamma \in \Gamma_0$ and
an element $\tau \in \IR$ such that $\gamma \cdot (x,s) = \Phi_{\tau}(x,s)$ and  $0 < \tau \le l$. 
The $\Gamma_0$-action on $\IR^n
\times \IR$ induces a $\Gamma$-action on $\call_{\le l}$.  We want to prove

\begin{lemma} \label{lem:fin_periodic_orbits} The set $\call_{\le l}/\Gamma_0$
  is finite.
\end{lemma}
\begin{proof}
For an integer $k \ge 1$ put
  \[
  \call'_k := \bigl\{x \in \IR^n \;\bigl| \; \exists\; \gamma \in \Gamma_0
  \; \text{with}\; \gamma \cdot x = x\; \text{and}\;\pr(\gamma) =
  \pr(\sigma)^k \bigr\}.
\]
We obtain a surjection 

\[
\coprod_{k= 1}^l\call'_k \twoheadrightarrow \call_{\le l}
\]
by sending $x \in \call'_k$ to the orbit through $(x,0) \in \IR^n \times
\IR$.  This surjection is compatible with the obvious $A$-action on the source
and the restriction of the $\Gamma_0$-action on the target to $A$.  Hence it
suffices to show that $\call'_{k}/A$ is finite for $k = 1,2,3 \ldots$.

An element $\gamma \in \Gamma$ satisfies $\pr(\gamma) = \pr(\sigma)^k$ if and
only if there exists $a \in A$ with $\gamma = a\sigma^k$. Let $\phi \colon A \to
A$ be the $\IZ$-automorphism given by conjugation with $\sigma^k$. We will
consider $A$ as a subgroup of $\IR^n$ so that the action of $a$ on $\IR^n$ is
given by $x \mapsto x + a$. 
Let $\alpha \colon \IR \otimes_{\IZ} A \xrightarrow{\cong} \IR^n$ be the
$\IR$-isomorphisms which sends $\lambda \otimes a$ to $\lambda\cdot a$.
Let
$\phi_{\IR} \colon \IR^n \to \IR^n$ be the $\IR$-automorphism for which the following diagram commutes
\[
\xymatrix@!C=5em{
\IR \otimes_{\IZ} A \ar[d]^{\alpha} \ar[r]^{\id_{\IR} \otimes_{\IZ} \phi} 
&
\IR \otimes_{\IZ} A \ar[d]^{\alpha} 
\\
\IR^n \ar[r]^{\phi_{\IR}} 
&
\IR^n
}
\]
Notice that the element $\sigma^k$ acts on $\IR^n \to \IR^n$ by some affine motion. One easily checks
that there is an element $w \in \IR^n$ such that $\sigma^k \cdot x = \phi_{\IR}(x) + w$
holds for all $x \in \IR^n$. 
We conclude for $a \in A$ and $x \in \IR^n$
\[
a\sigma^k \cdot x = x \Leftrightarrow   (\id - \phi_{\IR})(x) -w = a.
\]
Hence we get 
\[
\call'_{k} = \bigl\{x \in \IR^n \; \bigl| \; (\id - \phi_{\IR})(x) -w \in A\bigr\}.
\]
Let $\Gamma_1$ be the preimage of $k \cdot C$ under
$\pr \colon \Gamma \to Q$. Since
$\Gamma_1$ has finite index in $\Gamma$ and  $\Gamma$ is special affine by assumption,
$\Gamma_1$ is special affine. Since $\Gamma$ is irreducible by assumption, we conclude from
Lemma~\ref{lem:Hast(Q;A)_finite_irreducible_case}~%
\ref{lem:Hast(Q;A)_finite_irreducible_case:characterization} that $\Gamma_1$ is irreducible.
Lemma~\ref{lem:Hast(Q;A)_finite_irreducible_case}~\ref{lem:Hast(Q;A)_finite_irreducible_case:H1} implies
that $H^1(k \cdot C;A) = \coker\bigl((\id_A - \phi) \colon A \to A\bigr)$ is finite. We conclude
that $(\id - \phi_{\IR}) \colon \IR^n \to \IR^n$ is bijective. Choose $v \in \IR^n$ with
$(\id - \phi_{\IR})(v) = w$. Obviously $A$ is a subgroup of $(\id - \phi_{\IR})^{-1}(A)$.  
We obtain a bijection of $A$-sets
\[
\call'_{k} \xrightarrow{\cong} (\id - \phi_{\IR})^{-1}(A)
\]
by sending $x$ to $x + v$. 
Hence it remains to show that $(\id - \phi_{\IR})^{-1}(A)/A$ is finite.

Since $\coker\bigl(\id_A - \phi) \colon A \to A\bigr)$ is finite, we can find a integer $d \ge 1$ such that
$d \cdot A$ lies in the image of $(\id_A - \phi)$. This implies that 
$(\id - \phi_{\IR})^{-1}(d \cdot A) \subseteq A$. Hence we obtain an epimorphism
\[
A/(d \cdot A)  \to (\id - \phi_{\IR})^{-1}(A)/A \to  \quad \overline{a} \mapsto \overline{(\id - \phi_{\IR})^{-1}(a)}.
\]
Since $A/d A$ is finite, $(\id - \phi_{\IR})^{-1}(A)/A$ is finite. This finishes the proof of Lemma~\ref{lem:fin_periodic_orbits}.
\end{proof}

Now we are ready to give the proof of
Proposition~\ref{prop:some-contractions}.

\begin{proof}[Proof of Proposition~\ref{prop:some-contractions}]
 Consider the setup introduced before Lemma~\ref{lem:fin_periodic_orbits}.
 Using Lemma~\ref{lem:fin_periodic_orbits} one checks that the condition appearing
 in~\cite[Theorem~1.4]{Bartels-Lueck-Reich(2008cover)} are satisfied for $\Gamma_0$ and the flow $\Phi$.
 Hence we obtain a
  natural number $N$ such that for every $\alpha > 0$ there exists a
  $\VCyc$-cover $\calu$ of $\IR^n \times \IR$ with the following
  properties 
  \begin{enumerate}
  \item $\dim \calu \leq N/[\Gamma:\Gamma_0]$;  
  \item For every $x \in X$ there exists $U \in \calu$ such that
    \[
    \Phi_{[-\alpha,\alpha]} (x,t) := \{ \Phi_\tau(x,t) \; | \; \tau
    \in [-\alpha,\alpha] \} = \{x\} \times [t-\alpha,t+\alpha]
    \subseteq U;
    \]
  \item $\Gamma_0 \backslash \calu$ is finite.
  \end{enumerate}
  The number $N$ above is the number $N$ we are looking for in
  Proposition~\ref{prop:some-contractions}.

  Let $D_N$, $C_1$, $C_2$, and $C_3$ be the constants
  appearing in Lemma~\ref{lem:l1-metric_and_subdivision},
  Lemma~\ref{lem:Gamma_toRn_timesR}, and
  Lemma~\ref{lem:metric_estimate_for_Gamma_to_D}. 
  Consider any real numbers $R > 0$ and $\epsilon > 0$.
  Put
  \begin{eqnarray}
    \omega
    & := &
    \max\left\{\frac{64 \cdot D_N \cdot N^2
    \cdot (C_1 \cdot R + C_2)}{\epsilon},  8 \cdot N \cdot (C_1 \cdot R + C_2)\right\}.
    \label{value_of_omega}
  \end{eqnarray}

Next we show that the statement~\ref{prop:some-contractions:hard}
appearing in Proposition~\ref{prop:some-contractions} is true.
Fix a natural number $m$. Let $\sigma_m$ and $\alpha_m > 0$ be the real numbers coming from
Lemma~\ref{lem:ball_inclusions} for $\omega$ defined
in~\eqref{value_of_omega} and  for $I = [-m,m] \subseteq \IR$. Hence we get
\begin{eqnarray}
& & B^{\Gamma}_{\omega}(x,s)
\subseteq
  B^{\euc}_{\sigma_m}(x) \times [s-\alpha_m/2,s+\alpha_m/2]
\quad \text{for}\;  x \in \IR^n, s \in [-m,m].
\label{ball_inclusion_in_proof}
\end{eqnarray}

  For this $\alpha_m$ choose the $\VCyc$-cover $\calu_m$ of $\IR^n \times \IR$ as above.
  Recall that a $\VCyc$-cover $\calu_m$ is an open cover such that for
  $U \in \calu_m$, and $\gamma \in \Gamma$ we have $\gamma U \in \calu_m$
  and $\gamma U \cap U \not= \emptyset \implies \gamma U = U$ and for
  every $U \in \calu_m$ the subgroup $\Gamma_U := \{\gamma \in U \mid
  \gamma U = U\}$ of $\Gamma$ is virtually cyclic. 

Choose elements  $\gamma_1, \gamma_2 \ldots, \gamma_{[\Gamma:\Gamma_0]}$  in $\Gamma$
such that $\{\overline{\gamma_1}, \overline{\gamma_2} \ldots, \overline{\gamma_{\Gamma:\Gamma_0}}\}
= \Gamma/\Gamma_0$.  Put
\[
\calv_m := \{\gamma_i\cdot  U \mid U \in \calu_m, i = 1,2, \ldots, [\Gamma:\Gamma_0]\}
\]
Then $\calv_m$ is an open cover satisfying 
\begin{enumerate}
\item $\calv_m$ is $\Gamma$-invariant cover, i.e., $\gamma \in \Gamma, V
  \in \calv_m \implies \gamma V \in \calv_m$;
\item $\dim \calv_m \leq N$;
\item For every $(x,y) \in \IR^n \times \IR $ there exists $V \in
  \calv_m$ such that
  \[
  \Phi_{[-\alpha_m,\alpha_m]} (x,y) := \{ \Phi_\tau(x,y) \; | \; \tau \in
  [-\alpha_m,\alpha_m] \} \subseteq V;
  \]
\item $\Gamma\backslash \calv_m$ is finite.
\end{enumerate}

Next we show that we can find $\eta_m > 0$ such that for 
every $(x,s) \in \IR^n \times [-m,m]$
there exists $V \in \calv_m$ such that
\begin{eqnarray}
  \label{thickening_to_B_epsilon_timesB_alpha/2}
  & B_{\eta_m}^{\euc}(x) \times [s-\alpha_m/2,s + \alpha_m/2] \subseteq V, &
\end{eqnarray}
Suppose the contrary.  Then we can find sequences $(x_i)_i$ in $\IR^n$ and
$(s_i)_i$ in $[-m,m]$ 
such that for no $i \ge 1$ there exists $V \in \calv_m$ with the property 
$B^{\euc}_{1/i}(x_i) \times [s_i - \alpha_m/2,s_i + \alpha_m/2] \subset V$.  
Since the $\Gamma$-action on
$\IR^n \times \IR$ is proper and cocompact, there is a
compact subset $K \subseteq \IR^n \times \IR$ with $\Gamma \cdot K = \IR^n \times \IR$. 
Hence we can find a sequence
$(\gamma_i)_i$ in $\Gamma$ and an element $(x,s)$ in $\IR^n \times
\IR$ such that
$$\lim_{i \to \infty} \gamma_i \cdot (x_i,s_i) = (x,s).$$

Recall that $\Gamma$ acts diagonally on $\IR^n \times \IR$, where the
action on $\IR$ comes from the epimorphism $\Gamma \to D$ with $\Delta$ as kernel and
a proper $D$-action on $\IR$. Since $[-m,m]$ is compact,  the set
$\{\gamma_i \Delta \mid i \ge 1\} \subseteq \Gamma/\Delta$ is finite.
By passing to a subsequence, we can arrange that it consists of precisely
one element, in other words, there exists an element $\gamma \in \Gamma$ and a sequence
$(\delta_i)$ of elements in $\Delta$ such that  $\gamma_i = \gamma \cdot \delta_i$
holds for $i \ge 1$. Hence we can assume 
$$\lim_{i \to \infty} \delta_i \cdot (x_i,s_i ) = (x,s),$$
otherwise replace $(x,s)$ by $\gamma^{-1}\cdot (x,s)$. 

Choose $V\in \calv_m$ such that $\{x\} \times [s-\alpha_m,s+\alpha_m] \in V$.  
Since $\{x\} \times [s-\alpha_m,s+\alpha_m]$ is compact and $V$ is
open, we can find $\xi > 0$ with 
$B^{\euc}_{\xi}(x) \times [s-\alpha_m,s+\alpha_m] \subseteq V$. 
We can choose $i$ such that 
$(\delta_i \cdot x_i,s_i) \in B^{\euc}_{\xi/2}(x) \times [s-\alpha_m/2,s+\alpha_m/2]$ 
and $1/i \le \xi/2$.  Hence 
$B^{\euc}_{1/i}(\delta_i \cdot  x_i) \times [s_i-\alpha_m/2,s_i+\alpha_m/2]$ 
is contained in $B^{\euc}_{\xi}(x) \times [s-\alpha_m,s+\alpha_m]$. We conclude
$$B^{\euc}_{1/i}(\delta_i \cdot x_i) \times [s_i-\alpha_m/2,s_i+\alpha_m/2] \subseteq V.$$
Since $\Delta$ acts isometrically on $\IR^n$, we obtain
$$B^{\euc}_{1/i}(x_i) \times [s_i-\alpha_m/2,s_i+\alpha_m/2] \subseteq \delta_i^{-1} \cdot V.$$
Since $\delta_i^{-1} V \in \calv_m$, we get a contradiction.
Hence~\eqref{thickening_to_B_epsilon_timesB_alpha/2} is true.

Now we define the desired number
\begin{eqnarray}
\xi_m & := & \frac{\sigma_m}{\eta_m}.
\label{def_of_xi_overlineH}
\end{eqnarray}
Next consider a subgroup $\overline{H} \subseteq \Gamma$ of finite index,
and a natural number $k$ satisfying the
assumptions appearing in assertion~\ref{prop:some-contractions:hard}
of Proposition~\ref{prop:some-contractions}.
{}From now on put $m = [D : \pi \circ \pr(\overline{H})]$.
We can choose a pseudo $k$-expansive map
$$\phi \colon \Gamma \to \Gamma$$
by Lemma~\ref{lem:existence_of_pseudo_s-expansive_maps}%
~\ref{lem:existence_of_pseudo_s-expansive_maps:existence}.
Because of Lemma~\ref{lem:overline(H)_subseteq_im(phi)} we can assume
\begin{eqnarray}
  \overline{H} & \subseteq & \im(\phi),
  \label{H_subset_im(phi)}
\end{eqnarray}
since the desired claim holds for $\overline{H}$
if it holds for some conjugate of $\overline{H}$, 
compare Remark~\ref{rem:Farrell-Hisiang_plus_conjugation}.
There exists $u \in \IR$ such that the affine map $a \colon \IR^n \to
\IR^n$ sending $x$ to $k \cdot x + u$ is $\phi$-equivariant (see
Lemma~\ref{lem:equivariant_affine_diffeo}).
Since
$$a \times \id_{\IR}\bigl(B^{\euc}_{\eta_m}(x) \times [s-\alpha_m/2,s+\alpha_m/2]\big)
= B^{\euc}_{k \cdot \eta_m}(a(x)) \times [s-\alpha_m/2,s+\alpha_m/2],$$
and $a$ is bijective, we
conclude from~\eqref{thickening_to_B_epsilon_timesB_alpha/2} that for
every $x \in \IR^n$ and $s \in [-m,m]$  there exists $V \in \calv_m$ satisfying
\begin{eqnarray*}
  B^{\euc}_{k\cdot \eta_m}(x) \times [s-\alpha_m/2,s + \alpha_m/2]
  & \subseteq & a \times \id_{\IR}(V).
\end{eqnarray*}
Since $k \ge \xi_m$ implies $k \cdot \eta_m \ge \sigma_m$
by our choice~\eqref{def_of_xi_overlineH} of
$\xi_m$ 
we conclude for every $x \in \IR^n$ and $s \in [-m,m]$
\begin{eqnarray*}
  B^{\euc}_{\sigma_m}(x) \times [s-\alpha_m/2,s + \alpha_m/2]
  & \subseteq &  a \times \id_{\IR}(V).
\end{eqnarray*}
Now~\eqref{ball_inclusion_in_proof} implies 
that for every $x \in \IR^n$ and $s \in [-m,m]$
there exists $V \in \calv_m$ satisfying
\begin{eqnarray}
  B^{\Gamma}_{\omega}(x,s)
  & \subseteq &  a \times \id_{\IR}(V)
\label{B_omega_condition_for_K}
\end{eqnarray}

Next consider the open covering $\calw_m := \{a \times \id(V) \mid V \in
\calv_m\}$ of $\IR^n \times \IR$.  This is an $\im(\phi)$-invariant
covering, since the diffeomorphism $a \times \id \colon \IR^n \times
\IR \to \IR^n \times \IR$ is $\phi$-equivariant and $\calv_m$ is
$\Gamma$-invariant.  By~\eqref{H_subset_im(phi)} we can consider
$\calw_m$ as a $\overline{H}$-invariant open covering of $\IR^n \times \IR$.
Since by definition $m = [D : \pi \circ \pr(\overline{H})]$,
we conclude that $\pi \circ \pr(\overline{H})$ contains $m \cdot \IZ$.
This implies $\pi \circ \pr(\overline{H}) \cdot [-m,m] = \IR$.
Since for every $\gamma \in \overline{H}$ we have
$\gamma \cdot B^{\Gamma}_{\omega}(x,s) = B^{\Gamma}_{\omega}\bigl(\gamma \cdot (x,s)\bigr)$,
we conclude  from~\eqref{B_omega_condition_for_K} that for every $(x,s) \in \IR^n \times \IR$
there exists $W \in \calw$ with $ B^{\Gamma}_{\omega}\bigl((x,s)\bigr) \subseteq W$. Hence
Lemma~\ref{lem:contracting_estimate_d1_d_Gamma}
(applied in the case $\Gamma = \overline{H}$)
implies that the $\overline{H}$-equivariant map
$$\beta^{\calw}  \colon \IR^n \times \IR \to | \calw |$$ defined in~\eqref{fcalu}
has the property that for $(x,s), (y,t) \in \IR^n \times \IR$ with
$d^{\Gamma}\bigl((x,s),(y,t)\bigr) \le \frac{\omega}{8N}$ we get
\begin{eqnarray}
  d^{l^1}_{|\calw|}(\beta^{\calw}(x,s),\beta^{\calw}(y,t)\bigr)
  & \le &
  \frac{64 \cdot N^2}{\omega} \cdot d^{\Gamma}\bigl((x,s),(y,t)\bigr).
  \label{estimate_dl11_calw_by_dGamma}
\end{eqnarray}
Now consider the composite
$$f \colon \Gamma \xrightarrow{\ev}
\IR^n \times \IR \xrightarrow{\beta^{\calw}} |\calw| \xrightarrow{\id}
|\calw'|.$$ The map $\ev$ is $\Gamma$-invariant and in particular
$\overline{H}$-equivariant.  Hence $f$ is $\overline{H}$-equivariant. The
$\overline{H}$-action on $|\calw|$ is simplicial.  Hence the
$\overline{H}$-action on the barycentric subdivision $|\calw'|$ is
simplicial and cell preserving.

Next we show that the isotropy groups of the $\overline{H}$-action on the
space $|\calw| = |\calw'|$ are all virtually cyclic. Consider $ z\in
|\calw|$. Choose a simplex $\sigma$ such that $z$ lies in its
interior. Let the simplex $\sigma$ be given by $\{W_0,W_1, \ldots
W_l\}$ for pairwise distinct elements $W_i \in \calw$.  Then for every
$\gamma$ in the isotropy group $\overline{H}_z$ we must have
$$\gamma \cdot \{W_0,W_1, \ldots W_l\} = \{W_0,W_1, \ldots W_l\}.$$
Hence $\overline{H}_z$  operates on the finite set $\{W_0,W_1, \ldots
W_r\}$. We conclude that $\overline{H}_z$ contains a subgroup
$\overline{H}_z'$ of finite index such that $\gamma \cdot W_0 = W_0$
holds for all $\gamma \in \overline{H}_z'$.  By construction there is $U
\in \calu$ such that $W = f(U)$ or $W = f(\overline{s} \cdot U)$ for
some fixed element $\overline{s} \in \Gamma$.  Let $\Gamma_z''
\subseteq \Gamma$ be the preimage of $ \overline{H}_z'$ under the
isomorphism $\phi \colon \Gamma \to \im(\phi)$. Hence either $\gamma''
\cdot U = U$ for all $\gamma'' \in \Gamma_z''$ or
$\overline{s}^{-1}\gamma'' \overline{s} \cdot U = U$ for all $\gamma''
\in \Gamma_z''$. Since $\calu$ is a $\VCyc$-covering, the group
$\Gamma_z''$ is virtually cyclic. Since it is isomorphic to a subgroup
of finite index of $ \overline{H}_z$, the isotropy group $ \overline{H}_z$ is
virtually cyclic.

Consider $\gamma_1,\gamma_2$ in $\Gamma$ with
$d_{\Gamma}(\gamma_1,\gamma_2) \le R$.  We want to show
\begin{eqnarray}
  d^{l^1}_{|\calw'|}\big(f(\gamma_1),f(\gamma_2)\bigr) & \le & \epsilon.
  \label{dl1(g(gamma_1),g(gamma_2))_le_epsilon}
\end{eqnarray}
Lemma~\ref{lem:Gamma_toRn_timesR} implies
\begin{eqnarray*}
  & d^{\Gamma}\bigl(\ev(\gamma_1),\ev(\gamma_2)\bigr)
  \le
  C_1 \cdot d_{\Gamma}\bigl(\gamma_1,\gamma_2\bigr) + C_2
  \le
  C_1 \cdot R  + C_2.
\end{eqnarray*}
Our choice of $\omega$ in~\eqref{value_of_omega} guarantees $C_1 \cdot R
+ C_2 \le \frac{\omega}{8N}$. Hence
\begin{eqnarray*}
  d^{\Gamma}\bigl(\ev(\gamma_1),\ev(\gamma_2)\bigr)
  & \le  &
  \frac{\omega}{8N}.
\end{eqnarray*}
We conclude from~\eqref{estimate_dl11_calw_by_dGamma}
\begin{eqnarray*}
  d^{l^1}_{|\calw|}\bigl(\beta^{\calw} \circ
\ev(\gamma_1),\beta^{\calw} \circ \ev(\gamma_2)\bigr)
  & \le &
  \frac{64 \cdot N^2}{\omega} \cdot d^{\Gamma}\bigl(\ev(\gamma_1),\ev(\gamma_2)\bigr)
  \\
  & \le &
  \frac{64 \cdot N^2}{\omega} \cdot \bigl(C_1 \cdot R + C_2\bigr).
\end{eqnarray*}
Lemma~\ref{lem:l1-metric_and_subdivision} implies
\begin{eqnarray*}
  d^{l^1}_{|\calw'|}\bigl(f(\gamma_1),f(\gamma_2)\bigr)
  & \le &
  \frac{64 \cdot N^2 \cdot D_N \cdot \bigl(C_1 \cdot R + C_2\bigr)}{\omega}.
\end{eqnarray*}
Our choice of $\omega$ in~\eqref{value_of_omega} implies
$$\frac{64 \cdot N^2 \cdot D_N \cdot \bigl(C_1 \cdot R + C_2\bigr)}{\omega}
\le \epsilon.$$
This finishes the proof
of~\eqref{dl1(g(gamma_1),g(gamma_2))_le_epsilon}.

Since $\overline{H}$-acts simplicially on $|\calw|$, it acts simplicially
and cell preserving on $|\calw'|$. Put $E := |\calw'|$. We have already shown that
all isotropy groups of the $\overline{H}$-action on $F$ are virtually cyclic.
The $\overline{H}$-map
$f \colon \Gamma \to E$ has the desired properties because
of~\eqref{dl1(g(gamma_1),g(gamma_2))_le_epsilon}.  This finishes the
proof of statement~\ref{prop:some-contractions:hard} appearing in
Proposition~\ref{prop:some-contractions}.

Next we prove statement~\ref{prop:some-contractions:easy} of
Proposition~\ref{prop:some-contractions}.   Choose an integer m satisfying
 \begin{eqnarray}
  m
  & \ge  &
 \frac{2\cdot C_3 \cdot R}{\epsilon}.
  \label{value_of_mu}
\end{eqnarray}
We conclude from
Lemma~\ref{lem:metric_estimate_for_Gamma_to_D} that for all
$\gamma_1,\gamma_2\in \Gamma$
$$d_{D}\bigl(\pi \circ \pr(\gamma_1),\pi \circ \pr(\gamma_2)\bigr) \le
C_3 \cdot d_{\Gamma}(\gamma_1,\gamma_2)$$ holds.
Let $\ev \colon D \to \IR$ be the map  given by
evaluation of the standard group action of $D$ on the origin $0$.  One
easily checks for $\delta_1, \delta_2 \in D$
$$d^{\euc}\bigl(\ev(\delta_1),\ev(\delta_2)\bigr)
\le d_{D}\bigl(\delta_1,\delta_2\bigr).$$
Let the desired map $f$ be the composite
$$f \colon \Gamma \xrightarrow{\pr} Q\xrightarrow{\pi} D
\xrightarrow{\ev} \IR \xrightarrow{\frac{1}{m} \cdot \id} \IR.$$ It
satisfies for all $\gamma_1,\gamma_2\in \Gamma$
\begin{eqnarray*} d_{\IR}\bigl(f(\gamma_1),f(\gamma_2)\bigr) & \le &
  \frac{C_3}{ m} \cdot d_{\Gamma}(\gamma_1,\gamma_2).
\end{eqnarray*}
Let $E$ be the simplicial complex whose underlying space is $\IR$ and
for which the set of zero-simplices is $\frac{1}{2}\cdot \IZ$. Then we
get for $x,y \in \IR$
$$d_E^{l^1}(x,y) \le 2 \cdot d_{\IR}(x,y).$$
Hence we obtain for all $\gamma_1,\gamma_2\in \Gamma$
\begin{eqnarray*} d_E^{l^1}\bigl(f(\gamma_1),f(\gamma_2)\bigr) & \le &
  \frac{2 \cdot C_3}{ m} \cdot d_{\Gamma}(\gamma_1,\gamma_2).
  \label{d_R_les_equal_1_over_mu_cdotd_l1}
\end{eqnarray*}
The choice of the integer $ m$ in~\eqref{value_of_mu} guarantees
$$\frac{2 \cdot C_3 \cdot R}{ m} \le \epsilon$$
Hence
$$d_{\Gamma}(x,y) \leq R \quad \implies \quad d_E^{l^1}(f(x),f(y)) \le  \epsilon $$
for $\gamma_1,\gamma_2 \in \Gamma$.

The standard operation of $D$ on $\IR$ is simplicial and cell
preserving.  Consider the group homomorphism $\phi_{ m} \colon D \to
D$ which sends $t$ to $t^{ m}$ and, if $D = D_{\infty}$, $s$ to $s$,
where we use the standard
presentations of $\IZ$ and $D_{\infty}$.
The map $ m \cdot \id \colon \IR \to \IR$ is $\phi_{ m}$-equivariant
if we equip source and target with the standard $D$-action.

Now consider any subgroup $\overline{H} \subseteq \Gamma$ with
$[D : \pi \circ \pr(\overline{H})] \ge 2 \cdot  m$.
We conclude $\pi \circ \pr(\overline{H}) \subseteq \im(\phi_{ m})$.
Since $\phi_{ m}$ is injective, we
can define an $\overline{H}$-action on $E$ by defining
$\overline{h} \cdot e = \delta \cdot
e$ for $e \in E$ and any $\delta \in D$ for which $\phi_{ m}(\delta)
= \pi \circ \pr(\overline{h})$ holds. One easily checks that the map $f$ is
$\overline{H}$-equivariant. Since the isotropy groups of the standard $D$-action
on $\IR$ are finite and the epimorphism $\pi \colon Q \to D$ has a
finite kernel and the kernel of $\pr$ is $A$, the isotropy group
$\overline{H}_e$ of any $e \in E$ satisfies
$A \subseteq \overline{H}_e$ and $[\overline{H}_e : A] < \infty$.
Now define the desired natural number $\mu$ by  $\mu = 2m$.
This finishes the proof of
Proposition~\ref{prop:some-contractions}.
\end{proof}

%%%%%%%%%%%%%%%%%%%%%%%%%%%%%%%%%%%%%%%%%%%%%%%%%%%%%%%%%%%%%%%%%%%%%

\subsection{Proof of the Farrell-Jones Conjecture for irreducible
  special affine groups}
\label{subsec:Proof_of_the_Farrell-Jones_Conjecture_for_irreducible_special_affine_groups}

\begin{proposition}[The Farrell-Jones Conjecture for irreducible
  special affine groups]
  \label{pro:FJC_irreducible_special_affine_groups}
  Both the $K$-theoretic and the $L$-theoretic FJC hold for all
  irreducible special affine groups.
\end{proposition}
\begin{proof}
  Because of Theorem~\ref{the:transitivity},
  Theorem~\ref{the:Farrell-Jones_Conjecture_for_Farrell-Hsiang_groups} and
  Theorem~\ref{the:The_Farrell-Jones_Conjecture_for_virtually_finitely_generated_abelian_groups}
  it suffices to show that a special affine group $G$ is a Farrell-Hsiang group
  with respect to the family $\calf$ of virtually finitely generated abelian
  groups in the sense of Definition~\ref{def:Farrell-Hsiang}.

  Let $N$ be the natural number appearing in
  Proposition~\ref{prop:some-contractions}.  Consider any real numbers $R > 0$
  and $\epsilon > 0$.  Let $\mu$ be the natural number and 
  $(\xi_n)_{n \le 1}$ be the sequence appearing in 
  Proposition~\ref{prop:some-contractions}. Now choose a natural number $\tau$
  such that $\mu < \tau$ and $\xi_{n} \le \tau$ for all $n \le \mu$.  
  For this choice of $\tau$ we choose $r,s$ as appearing in
  Proposition~\ref{prop:hyper-good}.  Let $\alpha_{r,s} \colon \Gamma \to A_s
  \rtimes_{\rho_{r,s}} Q_r$ be the epimorphism appearing in
  Proposition~\ref{prop:hyper-good}.  The map $\alpha_{r,s}$ will play the role
  of the map $\alpha_{R,\epsilon}$ appearing in
  Definition~\ref{def:Farrell-Hsiang}.

  Let $H$ be a hyperelementary subgroup of $A_s \rtimes_{\rho_{r,s}}
  Q_r$. Recall that $\overline{H}$ is the preimage of $H$ under
  $\alpha_{r,s}$. We have to construct the desired simplicial complex $E_H$ and
  the map $f_H \colon G \to E_H$ as demanded in
  Definition~\ref{def:Farrell-Hsiang}. 
  If $[D:\pi \circ \pr(\overline{H})] \ge \mu$,
  then we obtain the desired pair $(E_H,f_H)$ from
  assertion~\ref{prop:some-contractions:easy} of
  Proposition~\ref{prop:some-contractions}.  Suppose that 
  $[D:\pi \circ \pr(\overline{H})]  \le \mu$. Then  by our choice
  of $\tau$ we have $\tau \ge \xi_{[D : \pi \circ \pr(\overline{H})]}$ and $\mu < \tau$.
  In particular $[D:\pi \circ \pr(\overline{H})]  \ge \tau$ is not true.
  Hence by
  Proposition~\ref{prop:hyper-good} we obtain an integer $k$ such that the
  assumption appearing in assertion~\ref{prop:some-contractions:hard} of
  Proposition~\ref{prop:some-contractions} are satisfied and the conclusion of
  assertion~\ref{prop:some-contractions:hard} of
  Proposition~\ref{prop:some-contractions} gives the desired pair
  $(E_H,f_H)$. Hence $G$ is a Farrell-Hsiang group with respect to the family
  $\calf$ of virtually finitely generated abelian groups.  This finishes the
  proof of Proposition~\ref{pro:FJC_irreducible_special_affine_groups}.
\end{proof}

%%%%%%%%%%%%%%%%%%%%%%%%%%%%%%%%%%%%%%%%%%%%%%%%%%%%%%%%%%%%%%%%%%%%%
%%%%%%%%%%%%%%% Section 5: Virtually poly-IZ-groups %%%%%%%%%%%%%%%%%
%%%%%%%%%%%%%%%%%%%%%%%%%%%%%%%%%%%%%%%%%%%%%%%%%%%%%%%%%%%%%%%%%%%%%

\typeout{------------ Virtually poly-IZ-groups  --------------------}

\section{Virtually poly-$\IZ$-groups}
\label{sec:Virtually_poly-Z-groups}

This section is devoted to the proof of
Theorem~\ref{the:FJC_virtually_poly_Z-groups}.
It will be done by induction over the virtual cohomological
dimension.  We will need the following ingredients.

\begin{definition}[(Virtually) poly-$\IZ$]
\label{def:virtually_poly-Z}
We call a group $G'$ \emph{poly-$\IZ$}  if there exists a finite sequence
$$\{1\} = G_0' \subseteq G_1' \subseteq \ldots \subseteq G_n' = G'$$
of subgroups such that $G_{i-1}'$ is normal in $G_{i}'$ with infinite cyclic
quotient $G_{i}'/G_{i-1}'$ for $i = 1,2, \ldots , n$.

We call a group $G$ virtually poly-$\IZ$ if it contains a subgroup $G'$ of finite
index such that $G'$ is poly-$\IZ$.
\end{definition}

Let $G$ be a virtually poly-$\IZ$-group. Let $G' \subseteq G$ be any
subgroup of finite index, for which there exists a finite sequence
$\{1\} = G_0' \subseteq G_1' \subseteq \ldots \subseteq G_n' = G'$ of
subgroups such that $G_{i-1}'$ is normal in $G_{i}'$ with infinite
cyclic quotient $G_{i}'/G_{i-1}'$ for $i = 1,2, \ldots , n$. We call
the number $r(G) := n$ the \emph{Hirsch rank} of $G$.  We will see
that it depends only on $G$ but not on the particular choice of
subgroup $G' \subseteq G$ and the filtration
$\{1\} = G_0' \subseteq G_1' \subseteq \ldots \subseteq G_n' = G'$.

\begin{lemma}[Virtual cohomological dimension of virtually
  poly-$\IZ$-groups]
  \label{lem:virt_dim_of_virt_poly_Z-groups}
  Let $G$ be a virtually poly-$\IZ$-group. Then
  \begin{enumerate}
  \item \label{lem:virt_dim_of_virt_poly_Z-groups:Hirsch} $r(G) =
    \vcd(G)$;

  \item \label{lem:virt_dim_of_virt_poly_Z-groups:homological} We get
    $r(G) = \max\{i \mid H_i(G';\IZ/2) \not= 0\}$ for one (and hence
    all) poly-$\IZ$ subgroup $G' \subset G$ of finite index;

  \item \label{lem:virt_dim_of_virt_poly_Z-groups:Eunderbar} There
    exists a finite $r(G)$-dimensional model for the classifying space of
    proper $G$-actions $\underline{E}G$ and for any model
    $\underline{E}G$ we have $\dim(\underline{E}G) \ge r(G)$;

  \item \label{lem:virt_dim_of_virt_poly_Z-groups:subgroups_and_quotients}
    Subgroups and a quotient groups of virtually poly-$\IZ$ groups
    are again virtually poly-$\IZ$;

  \item \label{lem:virt_dim_of_virt_poly_Z-groups:extensions} Consider
    an extension of groups
$$1 \to G_0 \to G_1 \to G_2 \to 1.$$
Suppose that two of them are virtually poly-$\IZ$.  Then all of them
are virtually poly-$\IZ$ and we get for their cohomological dimensions
$$\vcd(G_1) = \vcd(G_0) + \vcd(G_2).$$

\end{enumerate}
\end{lemma}
\begin{proof}
  Assertions~\ref{lem:virt_dim_of_virt_poly_Z-groups:Hirsch},%
\ref{lem:virt_dim_of_virt_poly_Z-groups:homological}
  and~\ref{lem:virt_dim_of_virt_poly_Z-groups:Eunderbar} are proved
  in~\cite[Example~5.2.6]{Lueck(2005s)}.  The proof of the other
  assertions is now obvious using induction over the Hirsch rank.
\end{proof}

The next result is taken from~\cite[Lemma~4.4]{Farrell-Jones(1993a)}.

\begin{lemma} \label{lem:reduction_to_affine_groups} Let $G$ be a
  virtually poly-$\IZ$ group. Then there exists an exact sequence
$$1 \to G_0\to G \to \Gamma \to 1$$
satisfying
\begin{enumerate}
\item The group $G_0$ is either finite or a virtually poly-$\IZ$-group
  with $\vcd(G_0) \le \vcd(G) -2$;

\item $\Gamma$ is either a crystallographic or special affine group.

\end{enumerate}
\end{lemma}

Now we are ready to prove
Theorem~\ref{the:FJC_virtually_poly_Z-groups}.

\begin{proof} We use induction over the virtual cohomological
  dimension of the virtually poly-$\IZ$ group $G$. The induction
  beginning $\vcd(G) \le 1$ is trivial since in this case $G$ must be
  virtually cyclic by Lemma~\ref{lem:virt_dim_of_virt_poly_Z-groups}.
  For the induction step choose an extension
$$1 \to G_0\to G \xrightarrow{\pr} \Gamma \to 1$$ as appearing in
Lemma~\ref{lem:reduction_to_affine_groups}.  Consider any virtually
cyclic subgroup $V \subseteq \Gamma$. Then we obtain an exact sequence
$$1 \to G_0 \to \pr^{-1}(V) \to V \to 1.$$
Since $G_0$ is virtually poly-$\IZ$ with $\vcd(G_0) \le \vcd(G) -2$,
we conclude from Lemma~\ref{lem:virt_dim_of_virt_poly_Z-groups} that
$\pr^{-1}(V)$ is virtually poly-$\IZ$ with
$\vcd\bigl(\pr^{-1}(V)\bigr) < \vcd(G)$. Hence both the
$K$-theoretic and the $L$-theoretic FJC
hold for $\pr^{-1}(V)$. Because of  Theorem~\ref{the:extensions}
it remains to prove that both the
$K$-theoretic and the $L$-theoretic FJC
hold for $\Gamma$. If $\Gamma$ is crystallographic or an irreducible
special affine group, this follows from
Theorem~\ref{the:The_Farrell-Jones_Conjecture_for_virtually_finitely_generated_abelian_groups}
and Proposition~\ref{pro:FJC_irreducible_special_affine_groups}.
Hence it remains to prove both the $K$-theoretic and the $L$-theoretic
FJC for the special affine group $\Gamma$
provided that it admits an epimorphism $p \colon \Gamma \to \Gamma'$
to some virtually finitely generated abelian group $\Gamma'$ with
$\vcd(\Gamma')\ge 2$. If $K$ is the kernel of $p$, we obtain the exact
sequence $1 \to K \to \Gamma \xrightarrow{p} \Gamma'\to 1$. We
conclude from Lemma~\ref{lem:virt_dim_of_virt_poly_Z-groups} that $K$
is a virtually poly-$\IZ$ group with $\vcd(K) \le \vcd(\Gamma) -2 \le
\vcd(G) -2$.  Hence for any virtually cyclic subgroup $V$of $\Gamma'$ the
preimage $p^{-1}(V)$ is a virtually poly-$\IZ$-group with
$\vcd\bigl(p^{-1}(V)\bigr) < \vcd(G)$ by Lemma~\ref{lem:virt_dim_of_virt_poly_Z-groups}.
By the induction hypothesis $p^{-1}(V)$ satisfies both the
$K$-theoretic and the $L$-theoretic {FJC}. Since the same is true for $\Gamma'$ by
Theorem~\ref{the:The_Farrell-Jones_Conjecture_for_virtually_finitely_generated_abelian_groups},
we conclude from Theorem~\ref{the:extensions}  that $\Gamma$ satisfies both the
$K$-theoretic and the $L$-theoretic {FJC}. This finishes the proof of
Theorem~\ref{the:FJC_virtually_poly_Z-groups}.
\end{proof}

%%%%%%%%%%%%%%%%%%%%%%%%%%%%%%%%%%%%%%%%%%%%%%%%%%%%%%%%%%%%%%%%%%%%%
%% Section 6: Cocompact lattices in virtually connected Lie groups %%
%%%%%%%%%%%%%%%%%%%%%%%%%%%%%%%%%%%%%%%%%%%%%%%%%%%%%%%%%%%%%%%%%%%%%

\typeout{--Cocompact lattices in virtually  connected Lie groups  --}

\section{Cocompact lattices in virtually connected Lie groups}
\label{sec:Cocompact_lattices_in_virtually_connected_Lie_groups}

In this section we prove Theorem~\ref{the:FJC_lattices}. 

The main work which remains to be done is to give the proof  of 
Proposition~\ref{pro:reduction_to_poly_and_non-neg_sec-curvature} below
which is very similar to the one 
of~\cite[pages 264-265]{Farrell-Jones(1993a)}.
We call a Lie group \emph{semisimple} if its Lie algebra is semisimple.
A subgroup $G \subseteq L$ of a Lie group $L$ is called \emph{cocompact lattice}
if $G$ is discrete and $L/G$ compact.

\begin{proposition} \label{pro:reduction_to_poly_and_non-neg_sec-curvature}
  In order to prove 
  Theorem~\ref{the:FJC_lattices} it suffices to prove that every
  virtually poly-$\IZ$ group and every group which operates
  cocompactly, isometrically and properly on a complete, simply connected 
  Riemannian manifold
  with non-positive sectional curvature satisfy the 
  $K$- and $L$-theoretic {FJC}.
\end{proposition}

  Its proof needs some preparation.
  \begin{lemma} \label{lem:reduction_to_no_compact_normal_connected_subgroups}
  Let $L$ be a virtually connected Lie group. Let $K$ be the maximal connected normal compact subgroup of $L$.
  Let $G \subseteq L$ be a cocompact lattice. Let $\overline{G}$
  be the image of $G$ under the projection  $L \to L/K$. Then

  \begin{enumerate}

   \item \label{lem:reduction_to_no_compact_normal_connected_subgroups:semisimple}
    If $L$ is semisimple, then $L/K$ is semisimple;

   \item \label{lem:reduction_to_no_compact_normal_connected_subgroups:mnccsub}
   Every connected normal compact subgroup of $L/K$ is trivial;

   \item \label{lem:reduction_to_no_compact_normal_connected_subgroups:lattice}
   $\overline{G} \subseteq L/K$ is a cocompact lattice;

   \item \label{lem:reduction_to_no_compact_normal_connected_subgroups:FJC}
    If $\overline{G}$ satisfies the  FJC,
then $G$ satisfies the {FJC}.

  \end{enumerate}
\end{lemma}
  \begin{proof}~\ref{lem:reduction_to_no_compact_normal_connected_subgroups:semisimple}
  Any quotient of a semisimple Lie algebra is again semisimple.
  \\[1mm]–\ref{lem:reduction_to_no_compact_normal_connected_subgroups:mnccsub}
  If $H$ is a normal compact connected subgroup of $L/K$, then its preimage under
  the projection $L \to L/K$ is a normal compact connected subgroup of $L$.
  \\[1mm]–\ref{lem:reduction_to_no_compact_normal_connected_subgroups:lattice}
  Since $K$ is compact, $G \cap K$ is a finite group.
  \\[1mm]–\ref{lem:reduction_to_no_compact_normal_connected_subgroups:FJC}
  We have the exact sequence $1 \to G\cap K \to G \to \overline{G} \to 1$. Now
  apply Corollary~\ref{cor:from_Transitivity_Principle}.
  \end{proof}

  In the sequel we denote by $L^e$ the component of the identity,

  \begin{lemma} \label{lem:semisimple_case}
   Proposition~\ref{pro:reduction_to_poly_and_non-neg_sec-curvature} 
   is true
   provided that $G$ is a cocompact
   lattice in a virtually connected semisimple Lie group $L$.
 \end{lemma}

   The statement of this lemma unravels as follows.
   Let $G$ be a cocompact lattice in a virtually connected 
   semisimple Lie group $L$.
   Assume that every virtually poly-$\IZ$ group and every group 
   which operates cocompactly, isometrically and properly on a 
   complete, simply connected  Riemannian manifold
   with non-positive sectional curvature satisfies the 
   $K$- and $L$-theoretic {FJC}.
   Then $G$ satisfies the $K$- and $L$-theoretic {FJC}.

   \begin{proof}[Proof of Lemma~\ref{lem:semisimple_case}] 
   Because of
   Lemma~\ref{lem:reduction_to_no_compact_normal_connected_subgroups}
   we can assume without loss of generality that
    $L$ is a virtually connected semisimple Lie group
   for which  every connected normal compact subgroup $K \subset L$  is trivial.
    Let $Z \subseteq L$ be the normal subgroup of elements in $L$ which commute with
    every element in $L^e$. Put $\overline{L} := L/Z$.
    Let $G_Z$ be the intersection $G \cap Z$ and $\overline{G}$
    the image of $G$ under the projection $\pr \colon L \to \overline{L}$. Then
    the following statements are true:
    \begin{enumerate}
    \item   \label{G_Z_virt_abelian} $G_Z$ is virtually finitely generated abelian;
    \item   \label{overlineG_lattice} $\overline{G}$ is a cocompact lattice
     of $\overline{L}$;
    \item   \label{L/Z} $\overline{L}$ is a virtually connected semisimple
     Lie group whose  center is trivial.
    \end{enumerate}
   For the proof of assertion~\ref{G_Z_virt_abelian} 
   we can assume 
   without loss of generality that $L$ is connected
   since $L$ is virtually connected and 
   since a group is already finitely generated if it contains a 
   finitely generated subgroup of finite index.
   Then $Z$ is just the center of $L$ and in particular an 
   abelian Lie group.
   The intersection $G_Z$ of $G$ and $Z$ is a 
   cocompact discrete subgroup of an abelian Lie group $Z$ 
   and hence a
   finitely generated abelian group. 

   Assertion~\ref{overlineG_lattice}  follows
   by inspecting the proof of~\cite[Corollary~5.17 on page~84]{Raghunathan(1972)}
   which applies directly to our case since all compact connected normal subgroups
   of $L$ are trivial.

   Next we prove assertion~\ref{L/Z}. Obviously $\overline{L}$ is
   virtually connected and semisimple since the quotient of a semisimple Lie
   algebra is again semisimple.  Let $\overline{Z} \subseteq \overline{L}$ be
   the center of $\overline{L}$.  Let $Z' \subseteq L$ be its preimage under the
   projection $L \to \overline{L}$. Consider $g \in L^e$ and $g' \in Z'$. Then
   $g' g g'^{-1}g^{-1}$ belongs to $Z$. Choose a path $w$ in $L$ connecting $1$ and
   $g$ in $L^e$.  Then $g'w(t)(g')^{-1}w(t)^{-1}$ is a path in $Z$ connecting $1$ and
   $g' g g'^{-1}g^{-1}$.  Since $L$ is semisimple, $Z \subseteq L$ is discrete. Hence
   $g'g g'^{-1}g^{-1} = 1$. This implies $g' \in Z$. Hence $Z = Z'$ and we conclude that
   the center of $\overline{L}$  is trivial.

   Because of Corollary~\ref{cor:from_Transitivity_Principle_and_poly_Z}
   it suffices to show that $\overline{G}$ satisfies the {FJC}.

   By~\cite[Theorem~A.5]{Abels(1974)} there exists a maximal compact
   subgroup $K \subseteq \overline{L}$ and the space $\overline{L}/K$
   is contractible. Then $K \cap \overline{L}^e$ is a maximal
   compact subgroup of $\overline{L}^e$ and $\overline{L}/K =
   \overline{L}^e/(K \cap \overline{L}^e)$.  Since $\overline{L}$ is
   semi-simple, its Lie algebra contains no compact ideal and its
   center is finite, the quotient
  $$M:= \overline{L}/K = \overline{L}^e/(K \cap \overline{L}^e)$$
  equipped with a $\overline{L}$-invariant Riemannian metric is a
  symmetric space of non-compact type such that $\overline{L}^e= \Isom(M)^e$
   and $K \cap \overline{L}^e = (\Isom(M)^e)_x$ for
  $\Isom(M)$ the group of isometries (see~\cite[Section 2.2 on page
  70]{Eberlein(1996)}).  Hence $M$ has non-positive sectional
  curvature (see~\cite[Proposition 4.2 in V.4 on page 244, Theorem 3.1
  in V.3 on page 241]{Helgason(1978)}.  Obviously $\overline{G}$ acts
  properly cocompactly and isometrically on $M$.  By assumption $\overline{G}$
  satisfy the {FJC}. This finishes the proof of
  Lemma~\ref{lem:semisimple_case}.
\end{proof}

\begin{lemma}
  \label{lem:is-a-lattice}
  Let $G$ be a  lattice in a virtually connected Lie group $L$.
  Assume that every compact connected normal subgroup of $L$ is trivial.
  Let $N$ be the nilradical in $L$.
  Then $G_N := G \cap N$ is a lattice in $N$. 
\end{lemma}

\begin{proof}
  Let $S$ be the semi-simple part of $L^{e}$.
  By~\cite[Theorem 1.6 on page 106]{Vinberg(2000)} it suffices to show
  that $S$ has no non-trivial 
  compact factors that act trivially on $R$ and $L$.
  Assume that $K$ is such a factor.
  Let $L^{e} = RS$ be the Levi decomposition of $L^{e}$.
  (We mention as a caveat that $S$ is not necessarily a \emph{closed}
  subgroup of $L^{e}$; nor is $R \cap S$ necessarily discrete; 
  although $R \cap S$ is countable.)
  Since $K$ is a factor of $S$ it is a normal subgroup of $S$
  and therefore $sKs^{-1} \subseteq  K$ for all $s \in S$.
  Because $K$ acts trivially on $R$ we have
  $r K r^{-1} \subseteq  K$ for all $r \in R$. 
  Since $L^{e} = RS$, we conclude
  that $K$ is a normal subgroup of $L^{e}$.
  Consequently, $K$ is a normal compact connected subgroup of $L^{e}$
  and therefore contained in the unique maximal normal compact connected
  subgroup $K_{\mathit{max}}$ of $L^e$. 
  This subgroup $K_{\mathit{max}}$ is a characteristic subgroup of $L^e$.
  Thus $K_{\mathit{max}}$ is in addition normal in $L$ and therefore trivial.
  Hence $K$ is trivial.  
\end{proof}

\begin{proof}[Proof of
  Proposition~\ref{pro:reduction_to_poly_and_non-neg_sec-curvature}]
  We proceed by induction on (the manifold) dimension of $L$, i.e., we assume
  that Proposition~\ref{pro:reduction_to_poly_and_non-neg_sec-curvature}
  is true for all virtually connected Lie groups $L'$ where 
  $\dim L' < \dim L$.
  We may assume, because of 
  Lemma~\ref{lem:reduction_to_no_compact_normal_connected_subgroups},
  that every compact connected normal subgroup of $L$ is trivial.
  Consider the sequence of normal subgroups of $L$,
  \begin{equation*}
    N \lhd R \lhd L^{e} \lhd L
  \end{equation*}
  where $L^{e}$ is the connected component of $L$
  containing the identity;
  $R$ is the radical of $L$;
  and $N$ is the nilradical of $L$.
  And let 
  \begin{equation*}
    G_N := G \cap N
  \end{equation*}
  By Lemma~\ref{lem:is-a-lattice} $G_N$ is a cocompact lattice in $N$.
  Therefore $G/G_N$ is a cocompact lattice in $L/N$ as well.
 
  We now distinguish two cases. 
  First consider the case that $N$ is nontrivial.
  Then $\dim L/N < \dim L$ and $G/G_N$ satisfies the FJC
  by our inductive assumption.
  Now consider the following exact sequence
  \begin{equation*}
    1 \to G_N \to G \to G/G_N \to 1
  \end{equation*}
  and observe that $G_N$ is a virtually poly-$\IZ$ group by a result of 
  Mostow (This follows from
  Theorem~\ref{lem:virt_dim_of_virt_poly_Z-groups}~%
\ref{lem:virt_dim_of_virt_poly_Z-groups:subgroups_and_quotients}
  and~\cite[Proposition~3.7 on page 52]{Raghunathan(1972)}.)
  Hence $G$ satisfies the   FJC because
  of Corollary~\ref{cor:from_Transitivity_Principle_and_poly_Z} 
  
  Next consider the case that $N$ is trivial. Then $R = R/N$ is abelian. Hence $R = N = 1$.
  Therefore $L$ is semi-simple and $G$ satisfies the  FJC
  because of Lemma~\ref{lem:semisimple_case}.    
\end{proof}

Now we are ready to prove Theorem~\ref{the:FJC_lattices}.

\begin{proof}[Proof of Theorem~\ref{the:FJC_lattices}.]
  We have proved the  FJC for
  virtually poly-$\IZ$ groups in
  Theorem~\ref{pro:FJC_irreducible_special_affine_groups}.  
  Since every group
  $G$ which acts cocompactly, isometrically and properly on 
  a complete, simply connected Riemannian
  manifold with non-positive sectional curvature is a 
  cocompact $\CAT(0)$-group,
  it satisfies the  FJC
  by the main results from~\cite{Bartels-Lueck(2012annals), 
          Wegner(2012higher-cat0)}.
  Now apply
  Proposition~\ref{pro:reduction_to_poly_and_non-neg_sec-curvature}.
\end{proof}

%\begin{remark} \label{rem:Bartels-Rosenthal}
%The assembly map 
%\begin{eqnarray*}
%    & \asmb^{G,\cala}_n \colon
%    H_n^G\bigl(\EGF{G}{\Fin};\bfK_{\cala}^{\langle - \infty\rangle}\bigr) \to
%    H_n^G\bigl(\pt;\bfK_{\cala}^{\langle - \infty\rangle}\bigr)
%    = K_n^{\langle - \infty\rangle}\left(\intgf{G}{\cala}\right)
%    &
%  \end{eqnarray*}
%is injective in all dimension $n \in \IZ$ if $G$ is a discrete
%subgroup in a virtually connected Lie group
%(see~\cite{Bartels-Rosenthal(2008asympdim)}).
%\end{remark}

%%%%%%%%%%%%%%%%%%%%%%%%%%%%%%%%%%%%%%%%%%%%%%%%%%%%%%%%%%%%%%%%%%%%%
%%%%%%%%%%%%% Section 7: Fundamental groups of 3-manifolds %%%%%%%%%%
%%%%%%%%%%%%%%%%%%%%%%%%%%%%%%%%%%%%%%%%%%%%%%%%%%%%%%%%%%%%%%%%%%%%%

\typeout{-- Fundamental groups of 3-manifolds ----------------------}

\section{Fundamental groups of $3$-manifolds}
\label{sec:Fundamental_groups_of_3-manifolds}

In this section we  sketch the proof of
Corollary~\ref{cor:fundamental_groups_of_3-manifolds}

\begin{remark}[Pseudo-isotopy] \label{rem:pseudo_isotopy_in_dimension_three} Let
  $\pi$ be the fundamental group of a $3$-manifold.  Roushon
  (see~\cite{Roushon(2008FJJ3)}, \cite{Roushon(2008IC3surf)}) gives a proof of
  Farrell-Jones Conjecture for pseudo-isotopy with wreath product for the family
  $\VCyc$ for $\pi$. Its proof relies on the assumption that the Farrell-Jones
  Conjecture for pseudo-isotopy is true for poly-$\IZ$-groups as stated in
  Farrell-Jones~\cite{Farrell-Jones(1993a)}. 
  Unfortunately that proof 
  depends on~\cite[Theorem~4.8]{Farrell-Jones(1993a)} 
  whose proof in turn has never appeared. 
  Hence the proof of the Farrell-Jones Conjecture for pseudo-isotopy 
  with wreath product
  for the family $\VCyc$ for $\pi$ is not complete.
\end{remark}

\begin{proof}[Discussion of proof of
  Corollary~\ref{cor:fundamental_groups_of_3-manifolds}]
  In this paper we have proved both the $K$-theoretic and the $L$-theoretic FJC
  for virtually poly-$\IZ$-groups in
  Theorem~\ref{the:FJC_virtually_poly_Z-groups}.  One can  check that the
  rather involved argument by Roushon (see~\cite{Roushon(2008FJJ3)},
  \cite{Roushon(2008IC3surf)}) for pseudo-isotopy goes through in our setting.

  This check above has been carried out in detail and in a comprehensible way in
  the Diplom-Arbeit by Philipp K\"uhl~\cite{Kuehl(2009)} axiomatically.
  A group $G$ satisfies
  \emph{the FJC with wreath products} if for any finite group $F$ the wreath
  product $G\wr F$ satisfies the {FJC}. K\"uhl proves following Roushon that the
  FJC with wreath products holds for the fundamental group of every
  $3$-manifold, if the following is true:
  \begin{itemize}

  \item The FJC  with wreath products  holds for $\IZ^2 \rtimes_{\phi} \IZ$
   for any automorphism $\phi \colon \IZ^2 \to \IZ^2$;

  \item The FJC holds for fundamental groups of closed Riemannian
    manifolds with non-positive sectional curvature;

  \item Theorem~\ref{cor:directed_colimits} and Theorem~\ref{the:extensions} are true.

  \end{itemize}

  Since a wreath product $G\wr F$ for a finite group $F$ and a group which is
  virtually poly-$\IZ$ is a again virtually poly-$\IZ$, the FJC with wreath
  product holds for all virtually poly-$\IZ$-groups if 
  and only if the FJC holds
  for all virtually poly-$\IZ$-groups.  
  Hence the axioms above are satisfied.
\end{proof}

\begin{remark}[Virtually weak strongly poly-surface groups]
  \label{rem:Virtually_weak_strongly_poly-surface_groups}
  Roushon defines weak strongly poly-surface groups
  in~\cite[Definition~1.2.1]{Roushon(2008IC3surf)}.  His argument in
  the proof of~\cite[Theorem~1.2.2]{Roushon(2008IC3surf)} carries
  over to our setting and shows that virtually weak 
  strongly poly-surface groups
  satisfy the $K$- and $L$-theoretic Farrell-Jones Conjecture with
  additive categories as coefficients with respect to the family
  $\VCyc$ (see
  Definitions~\ref{def:K-theoretic_Farrell-Jones_Conjecture}
  and~\ref{def:L-theoretic_Farrell-Jones_Conjecture}). 
\end{remark}

%%%%%%%%%%%%%%%%%%%%%%%%%%%%%%%%%%%%%%%%%%%%%%%%%%%%%%%%%%%%%%%%%%%%%
%%%%%%%%%%%% Section 8 : Reducing the family $\VCyc$  %%%%&%%%%%%%%%%%
%%%%%%%%%%%%%%%%%%%%%%%%%%%%%%%%%%%%%%%%%%%%%%%%%%%%%%%%%%%%%%%%%%%%%

\typeout{--------------- Reducing the family $\VCyc$ ----------------}

\section{Reducing the family $\VCyc$}
\label{sec:Reducing_the_family_VCyc}

In this subsection we explain how one can reduce the
family of subgroups in our setting of equivariant additive categories as
coefficients. 

\begin{definition}[Hyperelementary group]
\label{def_hyperelementary}
Let $l$ be a prime. A  (possibly infinite) group $G$ is
called \emph{$l$-hyper\-ele\-men\-tary}
if it can be written as an extension
$1 \to C \to G \to L \to 1$ for a cyclic group $C$ and a
finite group $L$ whose order is a power of $l$.

We call $G$ \emph{hyperelementary} if $G$ is $l$-hyperelementary for some prime $l$.
\end{definition}

If $G$ is finite, this reduces to the usual definition.
Notice that for a finite $l$-hyperelementary
group $L$ one can arrange that the order of the
finite cyclic group $C$ appearing in the extension
$1 \to C \to G \to L \to 1$ is prime to $l$. Subgroups and quotient groups of
$l$-hyperelementary groups are $l$-hyperelementary again. For a group
$G$ we denote by $\calh$ the family of hyperelementary subgroups of $G$.

The following result has been proved for $K$-theory and untwisted coefficients by
Quinn~\cite{Quinn(2012virtab)} and our proof is strongly motivated by his argument.

\begin{theorem}[Hyperelementary induction] \label{the:hyperelementary_induction}
Let $G$ be a group and let $\cala$ be an additive $G$-category (with involution).
Then both relative assembly maps
\[\asmb^{G,\calh,\VCyc}_n \colon H_n^G\bigl(\EGF{G}{\calh};\bfK_{\cala}\bigr)
 \to H_n^G\bigl(\EGF{G}{\VCyc};\bfK_{\cala}\bigr)\]
and
\[\asmb^{G,\calh,\VCyc}_n \colon
H_n^G\bigl(\EGF{G}{\calh};\bfL^{\langle -\infty \rangle}_{\cala}\bigr)
 \to H_n^G\bigl(\EGF{G}{\VCyc};\bfL_{\cala}^{\langle -\infty \rangle}\bigr)\]
induced by the up to $G$-homotopy unique $G$-map  $\EGF{G}{\calh} \to \EGF{G}{\VCyc}$
are bijective for
all $n \in \IZ$.
\end{theorem}
\begin{proof}
  Because of the Transitivity Principle~\ref{the:transitivity} we can assume
  without loss of generality that $G$ is virtually cyclic. If $G$ is finite, the
  claim follows from
  Bartels-L\"uck~\cite[Theorem~2.9 and Lemma~4.1]{Bartels-Lueck(2007ind)}. There only
  the case of the fibered FJC for coefficients in a ring (without $G$-action) is
  treated but the proof carries directly over to the case of coefficients in an
  additive $G$-category with coefficients. Notice that action of the Swan group of
  a group $G$ which is well-known in the untwisted case carries directly over
  to the case of additive $G$-category with coefficients.
  Hence we can assume in the sequel that $G$ is an infinite
  virtual cyclic group and Theorem~\ref{the:hyperelementary_induction} holds for
  all finite groups.

  The \emph{holonomy number} $h(G)$ of an infinite virtually cyclic group $G$ is
  the minimum over all integers $n \ge 1$ such that there exists an extension
  $1 \to C \to G \to Q \to 1$ for an infinite cyclic group $C$ and a finite group
  $Q$ with $|Q| = n$. We will use induction over the holonomy number $h(G)$.
  The induction beginning $h(G) = 1$ is trivial since in this case $G$ is
  infinite cyclic and both $\calh$ and $\VCyc$ consists of all subgroups.  It
  remains to explain the induction step.

So fix an infinite virtually cyclic subgroup $G$ with holonomy number $h(G)\ge 2$.
We have to prove Theorem~\ref{the:hyperelementary_induction} for $G$ under the
assumption that we know Theorem~\ref{the:hyperelementary_induction}  already
for all finite groups and for  all infinite virtually cyclic
subgroups whose holonomy number is smaller than $h(G)$.

Fix an extension
\begin{eqnarray}
& 1 \to C \xrightarrow{i} G \xrightarrow{\pr} Q \to 1.
\label{C_to_G_to_Q}
\end{eqnarray}
for an infinite cyclic group $C$ and a finite group $Q$ with $|Q| = h(G)$.  Let
$\calf$ be the family of subgroups of $G$ which are either finite, infinite
virtually cyclic groups $H \subseteq G$ with holonomy number $h(H) < h(G)$ or
hyperelementary.  This is indeed a family since for any infinite virtually
cyclic subgroups $H \subseteq K$ we have $h(H) \le h(K)$ and subgroups of
hyperelementary groups are hyperelementary.  Since the claim holds for all
finite groups and all infinite virtually cyclic groups whose holonomy number is
smaller than the one of $G$, it suffices because of the Transitivity
Principle~\ref{the:transitivity} to prove that $G$ satisfies the Farrell-Jones
Conjecture with respect to the family $\calf$.  This will be done by proving
that $G$ is a Farrell-Hsiang group with respect to the family $\calf$ in the
sense of Definition~\ref{def:Farrell-Hsiang} (see
Theorem~\ref{the:Farrell-Jones_Conjecture_for_Farrell-Hsiang_groups}).

For an integer $s$ define $C_{s} := C/sC$ and $G_{s}:=G/sC$. We obtain an
induced exact sequence
$$1 \to C_{s}\xrightarrow{i_{s}} G_{s} \xrightarrow{\pr_{s}} Q \to 1.$$
Denote by
$$\alpha_s \colon G \to G_{s}$$
the projection.

In the sequel we abbreviate $\overline{H} := \alpha_s^{-1}(H)$ for a
subgroup $H \subseteq G_s$.

\begin{lemma} \label{lem:reduction_by_induction} In order to prove
  Theorem~\ref{the:hyperelementary_induction} it suffices to find for given real numbers
  $R, \epsilon > 0$ a natural number $s$ with the following property: For every
  hyperelementary subgroup $H \subseteq G_s$ there exists a $1$-dimensional simplicial
  complex $E_H$ with cell preserving simplicial $\overline{H}$-action and a
  $\overline{H}$-map $f_H \colon G \to E$
  such that $d_G(g_1,g_2) \le R$ implies
  $d^{l^1}\bigl(f_H(g_1),f_H(g_1)\bigr) \le \epsilon$
  and all $\overline{H}$-isotropy groups of $E_H$ belong to $\calf$.
\end{lemma}
\begin{proof}
This follows from Theorem~\ref{the:Farrell-Jones_Conjecture_for_Farrell-Hsiang_groups}.
\end{proof}

In the next step we reduce the claim further to a question about indices.
Choose an epimorphism $\pi^G \colon G \to \Delta$ with finite
kernel onto a crystallographic group
(see~\cite[Lemma~4.2.1]{Quinn(2012virtab)}). Then $\Delta$ is either $D_{\infty}$ or $\IZ$.
The subgroup $A_ {\Delta}$ is infinite cyclic. If $\Delta = \IZ$, then $\Delta = A_{\Delta}$.
If $\Delta = D_{\infty}$, then $A_{\Delta} \subseteq \Delta$ has index two.

\begin{lemma} \label{lem:reduction_to_high_index}
In order to prove Theorem~\ref{the:hyperelementary_induction} it suffices to find
for a given natural number $i$  a natural number $s$ with the following property:
For every hyperelementary subgroup $H \subseteq G_s$ we have  $\overline{H} \in \calf$ or
$[\Delta:\pi^G(\overline{H})] \ge i$.
\end{lemma}
\begin{proof}[Proof of Lemma~\ref{lem:reduction_to_high_index}]
We show that the assumptions in Lemma~\ref{lem:reduction_to_high_index} imply the ones
appearing in Lemma~\ref{lem:reduction_by_induction}.
We only treat the difficult case $\Delta = D_{\infty}$,
the case $\Delta = \IZ$ is then obvious.

We have fixed a word metric $d_G$ on $G$. Equip $D_{\infty}$ with respect to the
word metric with respect to the standard presentation.
Since $\pi^G$ is surjective, we can find constants $C_1$ and $C_2$ such that for
$g_1, g_2 \in G$ we get
\begin{eqnarray}
d_{D_{\infty}}\bigl(\pi^G(g_1),\pi^G(g_2)\bigr) & \le & C_1 \cdot d_G(g_1,g_2) + C_2.
\label{d_d_infty_versus_d_G}
\end{eqnarray}

Fix real numbers $r,\epsilon > 0$.  Put
\begin{eqnarray}
i  & := & \frac{2 C_1R + 2C_2}{\epsilon}.
\label{choice_of_i}
\end{eqnarray}
Now choose $s$ such that we have  $\overline{H} \in \calf$ or
$[\Delta:\pi^G(\overline{H}] \ge i$ for every hyperelementary subgroup
$H \subseteq G_s$. If $\overline{H} \in \calf$, we can choose
$f_H$ to be the projection $G \to \pt$. Hence we can assume in the sequel
$[\Delta:\pi^G(\overline{H}] \ge i$.

By Lemma~\ref{lem:expansive_maps}~\ref{lem:expansive_maps:subgroup_in_the_image}
we can find a $i$-expansive map $\phi \colon D_{\infty} \to D_{\infty}$ with
$\pi^G(\overline{H}) \subseteq \im(\phi)$ and an element
$u \in \IR$ such that the affine map $a \colon \IR \to \IR$
sending $x$ to $i \cdot x  + u$ is
$\phi$-invariant. Let $E_H$ be the simplicial complex
whose underlying space is $\IR$ and whose
set of $0$-simplices is $\{m/2 \mid m \in \IZ\}$.
The standard $D_{\infty}$-action on $\IR$ yields
a cell preserving simplicial action on $E_H$ with finite stabilizers. Define a map
$$f_H \colon G \xrightarrow{\pi^G} D_{\infty} \xrightarrow{\ev} E \xrightarrow{a^{-1}} E,$$
where $\ev$ is given by evaluating the $D_{\infty}$-action on $0 \in \IR$.

One easily checks for $d_1,d_2 \in D_{\infty}$
\[d^{\euc}\bigl(\ev(d_1),\ev(d_2)\bigr) \le d_{D_{\infty}}(d_1,d_2).\]
We get for $x_1,x_2 \in E$
\[d^{l^1}(x_1,x_2) \le 2 \cdot d^{\euc}(x_1,x_2).\]
This implies together with~\eqref{d_d_infty_versus_d_G} and~\eqref{choice_of_i}
for $g_1, g_2 \in G$ with $d_G(g_1,g_2) \le R$
\begin{eqnarray*}
d^{l^1}\bigl(f_H(g_1),f_H(g_2)\bigr)
& = &
d^{l^1}\bigl(a^{-1} \circ \ev \circ \pi^G(g_1),a^{-1} \circ \ev \circ \pi^G(g_2)\bigr)
\\
& \le &
2 \cdot d^{\euc}\bigl(a^{-1} \circ \ev \circ \pi^G(g_1),a^{-1}
\circ \ev \circ \pi^G(g_2)\bigr)
\\
& \le &
\frac{2}{i} \cdot d^{\euc}\bigl(\ev \circ \pi^G(g_1),\ev\circ \pi^G(g_2)\bigr)
\\
& \le &
\frac{2}{i} \cdot d_{D_{\infty}}\bigl(\pi^G(g_1),\pi^G(g_2)\bigr)
\\
& \le &
\frac{2}{i} \cdot \bigl(C_1 \cdot d_G(g_1,g_2) + C_2\bigr)
\\
& \le  & \frac{2 C_1R + 2C_2}{i}
\\ & = & \epsilon.
\end{eqnarray*}
Since $\pi^G(\overline{H}) \subseteq \im(\phi)$, we can define an
$\overline{H}$-action on $E$ by requiring that $\overline{h} \in \overline{H}$
acts on $E_H$ by the standard $D_{\infty}$-action for
the element $d \in D_{\infty}$ which is uniquely determined by
$\phi(d) = \pi^G(h)$. This $\overline{H}$- action is a cell preserving
simplicial action and the map $f_H$ is $\overline{H}$-equivariant.
Hence $f_H$ has all the desired properties.  This finishes the proof
of Lemma~\ref{lem:reduction_to_high_index}.
\end{proof}
Now we continue with the proof of
Theorem~\ref{the:hyperelementary_induction}.  We will show that the
assumptions appearing in Lemma~\ref{lem:reduction_to_high_index} are
satisfied.

If the group $Q$ appearing in~\eqref{C_to_G_to_Q} is a $p$-group,
then $G$ itself is hyperelementary and the claim is true 
because $G \in \calf$. 
Hence we can assume in the sequel that we can fix  two different
primes $p$ and $q$  which divide the order of $Q$.

Let $i$ be a given natural number. Let $\log_p(|Q|)$ the integer $n$
for which $|Q| = p^n \cdot m$
for some natural number $m$ prime to $p$ holds.
Choose a natural number $r$ satisfying
\begin{eqnarray*}
\frac{p^{r-\log_p(|Q|)}}{\bigl|\ker(\pi^G \colon G \to \Delta)\bigr|} & \ge & i;
\\
\frac{q^{r-\log_q(|Q|)}}{\bigl|\ker(\pi^G \colon G \to \Delta)\bigr|} & \ge & i;
\\
r & \ge & \log_p(|Q|);
\\
r & \ge & \log_q(|Q|).
\end{eqnarray*}
Our desired number $s$  will be
$$s = p^rq^r.$$
We have to show for any hyperelementary subgroup $H \subseteq G_s$
\begin{eqnarray}
H \in \calf & \text{or} & \bigl[A_{\Delta} : (\pi^G(\overline{H}) \cap A_{\Delta})\bigr]
\ge i.
\label{need_to_show}
\end{eqnarray}
Consider an $l$-hyperelementary subgroup $H \subseteq G_{s}$.  Since $p$ and $q$
are different we can assume without loss of generality that $p \not= l$.  Denote
by $H_p \subseteq H$ the $p$-Sylow subgroup of $H$. Since $H$ is
$l$-hyperelementary and $l \not = p$, the subgroup $H_p$ is normal in $H$ and a
cyclic $p$-group.  Denote by $Q_p \subseteq Q$ the image of $H_p$ under the
projection $\pr_{s} \colon G_{s} \to Q$.  Suppose that $\pr_{s}(H) \not=
Q$. Then the holonomy number of $\overline{H}$ is smaller than the one of $G$
and $\overline{H}$ 
belongs by the induction hypothesis to $\calf$. Hence we can assume in the
sequel
\begin{eqnarray}
\pr_{s}(H) & = & Q.
\label{pr_s(H)_is_Q}
\end{eqnarray}
This implies that $Q_p \subseteq Q$ is a normal cyclic
$p$-subgroup of $Q$ and is the $p$-Sylow subgroup of $Q$.

Denote by $\overline{Q_p}$ the preimage of $Q_p$ under $\pr \colon G \to Q$.
The conjugation action
$\rho \colon Q \to \aut(C)$ of $Q$ on
$C$ associated to the exact sequence~\eqref{C_to_G_to_Q} yields
by restriction a $Q_p$-action.

We begin with the case, where this $Q_p$-action is non-trivial. Then
we must have $p = 2$ and the target of the epimorphism
$\pi^G \colon G \to \Delta$ is $\Delta = D_{\infty} = \IZ \rtimes_{-\id} \IZ/2$.
We obtain a commutative diagram
$$\xymatrix{1 \ar[r] & C \ar[d]^{j} \ar[r]^{i} & G \ar[d]^{\pi^G} \ar[r]^{\pr}
& Q  \ar[r] \ar[d]^{\pi^Q} & 1
\\
1 \ar[r] & \IZ \ar[r] & \IZ\rtimes_{-\id} \IZ/2 \ar[r] & \IZ/2 \ar[r] & 1}
$$
where $j$ is injective and both $\pi^G$ and $\pi^Q$ are surjective.
Let $H'$ be the image of  $H$ under the composite
$G_s = G/p^rq^rC \to G/q^rC \xrightarrow{\overline{\pi^G}} \IZ/j(q^rC) \rtimes_{-\id} \IZ/2
\to \IZ/q^r\IZ \rtimes_{-\id} \IZ/2$, where $\overline{\pi^G}$ is induced by $\pi^G$.
Then $H'$ agrees with the image of $\overline{H}$ under the composite
$G \xrightarrow{\pi^G} \IZ\rtimes_{-\id} \IZ/2 \to \IZ/q^r\rtimes_{-\id} \IZ/2$.
Since $H$ and hence $H'$ is $l$-hyperelementary for $l \not = 2$ and
$\IZ/q^r\IZ \rtimes_{-\id} \IZ/2$ is $2$-hyperelementary, $H'$ is cyclic.
Since $\pr_s(H_2) = Q_2$ and the surjectivity 
of $\pi^Q \colon Q \to \IZ/2$ implies $\pi^Q(Q_2) = \IZ/2$,
the image of $H'$ under the projection $\IZ/q^r\rtimes_{-\id} \IZ/2 \to \IZ/2$
is $\IZ/2$. 
Since $H'$ is cyclic, $q$ is different from $p = 2$ and hence odd,
we conclude $\IZ/q^r \cap H' = \{0\}$. Hence
$$
\bigl[\IZ: (\IZ \cap \pi^G(\overline{H}))\bigr]
\ge  \bigl[\IZ/q^r: (\IZ/q^r \cap H')\bigr]= q^r.$$
Since by our choice of $r$  we have $q^r\ge i$, assertion~\eqref{need_to_show} holds.

Hence it remains to treat the case where $Q_p$ acts trivially on $C$.
By restriction the exact sequence~\eqref{C_to_G_to_Q} yields the exact sequence
$$1 \to C \xrightarrow{i} \overline{Q_p}  \xrightarrow{\pr|_{\overline{Q_p}}} Q_p \to 1.$$
In the sequel  we identify $C$ with its image $i(C)$
under the injection $i \colon C \to G$.
Since $Q_p$ acts trivially on $C$ and is a finite cyclic $p$-group, the group
$\overline{Q}_p$ is a finitely generated abelian group
of rank one. Let $T \subseteq \overline{Q}_p$
be the torsion subgroup. Fix an infinite cyclic subgroup
$\IZ \subseteq \overline{Q_p}$ such that
$$T \oplus \IZ = \overline{Q}_p.$$
Recall that $p$ divides the order of $Q$ and that $Q_p$ is a $p$-Sylow
subgroup of $Q$.
In particular $Q_p$ is non-trivial.
Let $n \ge 1$ be the natural number for which $|Q_p| = p^n$.  Since
$\pr|_T \colon T \to Q_p$ is injective, $T$ is a cyclic $p$-group of
order $p^m$ for some natural number $m \le n$.  Since $r \ge n$ by our
choice of $r$ holds, $p^rC\subseteq \{0\} \times \IZ$. We get
$$T \oplus \IZ/p^rC = \overline{Q_p}/p^rC.$$
Suppose that $\overline{H} \cap T= \{0\}$.  Let $K$ be the kernel of
the composite
$$\overline{H} \xrightarrow{\pr|_{\overline{H}}} Q \to Q/Q_p.$$
We have $K \subseteq \overline{Q_p}$. Since $\overline{H} \cap T = \{0\}$
implies $K \cap T = \{0\}$, the restriction of the canonical
projection $\overline{Q_p} = T \oplus \IZ \to \IZ$ to $K$ is injective
and hence $K$ is infinite cyclic.  This implies
$h(\overline{H}) \le |Q/Q_p| < |Q| =h(G)$ and hence $H \in \calf$.
Therefore we can assume in the sequel
\begin{eqnarray}
\overline{H} \cap T & \not=&  \{0\}.
\label{overlineH_cap_T_is_non-trivial}
\end{eqnarray}
Let $H'$ be the image of $H$ under the projection $G/sC \to G/p^rC$.
Recall that $H_p$ is a cyclic $p$-group
and a normal  $p$-Sylow group of $H$ and is mapped under the projection
$G_s \to Q$ to $Q_p$. Let
$$H_p' \subseteq \overline{Q_p}/p^rC = T \oplus \IZ/p^rC$$
be the image of $H_p$ under the projection $G/sC \to G/p^rC$.
Then $H_p'$ is normal in $H'$ and  is the $p$-Sylow subgroup of $H'$.
Since $C/p^rC \subseteq  \overline{Q_p}/p^rC$ is a subgroup of order $p^r$, we conclude
$$H' \cap C/p^rC = H_p' \cap C/p^rC.$$
The intersection $\overline{H} \cap T$ is $p$-torsion, because $T$ is
$p$-torsion.
Thus $\overline{H} \cap T \subseteq \overline{H_p} := \alpha_s^{-1}(H_p)$.
Therefore we can conclude $H_p' \cap T\not= \{0\}$ 
from~\eqref{overlineH_cap_T_is_non-trivial}.
Since $H_p'$ is cyclic and $H_p' \cap T\not= \{0\}$, we must have
$H_p' \cap \IZ/p^rC = \{0\}$. Since $|T| \cdot H_p'$ is contained
in $\IZ/p^rC$, we conclude $|T| \cdot H_p' = \{0\}$ and hence
the order of $|H_p'|$ divides the order of $|T|$. This implies
\begin{equation}
 |H_p'| \le |T| \le p^n. 
\label{|H_p'|_le_pn}
\end{equation}

We conclude
\begin{eqnarray*}
[G: \overline{H}]
& = &
[G_s \colon H]
\\
& \ge &
[G/p^rC : H']
\\
& \ge &
[C/p^rC:(C/p^rC \cap H')]
\\
& = &
[C/p^rC:(C/p^rC \cap H_p')]
\\
& = &
\frac{|C/p^rC|}{|C/p^rC \cap H_p'|}
\\
& \ge  &
\frac{|C/p^rC|}{|H_p'|}
\\
& \ge &
\frac{p^r}{p^n}
\\
& = &
p^{r-n}.
\end{eqnarray*}
This and our choice of $r$ implies
\[
\bigl[\Delta : \pi^G(\overline{H})\bigr]
\ge
\frac{\bigl[G : \overline{H}\bigr]}{|\ker(\pi^G)|}
\ge
\frac{p^{r-n}}{|\ker(\pi^G)|}
\ge i.
\]
Hence assertion~\eqref{need_to_show} is true.
This finishes the proof of Theorem~\ref{the:hyperelementary_induction}.
\end{proof}

We will use the following result 
from~\cite[Corollary~1.2, Remark~1.6]{Davis-Quinn-Reich(2011)}.
See also~\cite{Davis-Khan-Ranicki(2011)}.

\begin{theorem}\label{the:reduction_to_type_I_in_k-theory}
  Let $G$ be a group.  Let $\VCyc_I$ be the family of subgroups which are either
  finite or admit an epimorphism onto $\IZ$ with a finite kernel.  Obviously
  $\VCyc_I \subseteq \VCyc$.  Then for an any additive $G$-category $\cala$ the
  relative assembly map
  \[\asmb^{G,\VCyc_I,\VCyc}_n \colon
  H_n^G\bigl(\EGF{G}{\VCyc_I};\bfK_{\cala}\bigr) \to
  H_n^G\bigl(\EGF{G}{\VCyc};\bfK_{\cala}\bigr)\] is bijective for all $n \in
  \IZ$.
\end{theorem}

The Transitivity Principle~\ref{the:transitivity},
 Theorem~\ref{the:hyperelementary_induction}
and Theorem~\ref{the:reduction_to_type_I_in_k-theory} imply

\begin{corollary}\label{the_reduction_in_K-theory-to_calh_I}
Let $G$ be a group.
Let $\calh_I$ be the family of subgroups which are either
finite or which are hyperelementary and
admit an epimorphism onto $\IZ$ with a finite kernel.
Obviously $\calh_I \subseteq \VCyc$.
Then for an any additive $G$-category $\cala$ the relative assembly map
\[\asmb^{G,\calh_I,\VCyc}_n \colon H_n^G\bigl(\EGF{G}{\calh_I};\bfK_{\cala}\bigr)
 \to H_n^G\bigl(\EGF{G}{\VCyc};\bfK_{\cala}\bigr)\]
is bijective for all $n \in \IZ$.
\end{corollary}

Every infinite $p$-hyperelementary group for odd $p$ admits an epimorphism
to $\IZ$ with finite kernel. A $2$-hyperelementary group $G$ admits an epimorphism
to $\IZ$ with finite kernel if and only if there exists a central extension
$1 \to \IZ \to G \to P \to 1$ for a finite $2$-group $P$.

\begin{theorem}\label{the:reduction_to_type_I_in_L-theory}
  Let $G$ be a group.  Let $\VCyc_I$ be the family of subgroups which
  are either finite or admit an epimorphism onto $\IZ$ with a finite
  kernel.  Let $\Fin$ be the family of finite groups.  Obviously
  $\Fin\subseteq \VCyc_I$.  Then for an any additive $G$-category
  $\cala$ the relative assembly map
  \[\asmb^{G,\Fin,\VCyc}_n \colon
  H_n^G\bigl(\EGF{G}{\Fin};\bfL_{\cala}^{\langle -
    \infty\rangle}\bigr) \to
  H_n^G\bigl(\EGF{G}{\VCyc_I};\bfL_{\cala}^{\langle -
    \infty\rangle}\bigr)\] is bijective for all $n \in \IZ$.
\end{theorem}
\begin{proof}[Sketch of proof]
  The argument given in~\cite[Lemma~4.2]{Lueck(2005heis)} goes through
  since it is based on the Wang sequence for a semi-direct product $F \rtimes
  \IZ$ which can be generalized for additive categories as coefficients.
\end{proof}

\addcontentsline{toc<<}{section}{References} 
%\bibliographystyle{abbrv}
%\bibliography{dbdef,dbpub,dbpre,dbpoly}

\def\cprime{$'$} \def\polhk#1{\setbox0=\hbox{#1}{\ooalign{\hidewidth
  \lower1.5ex\hbox{`}\hidewidth\crcr\unhbox0}}}

\end{document}